\documentclass[preprint,aap]{imsart}

\RequirePackage{amsthm,amsmath,amsfonts,amssymb}
\RequirePackage[numbers]{natbib}
\RequirePackage[colorlinks,citecolor=blue,urlcolor=blue]{hyperref}

\startlocaldefs

\usepackage{fix-cm}
\usepackage{graphicx}
\usepackage{mathtools,  stmaryrd}
\usepackage{ulem}
\usepackage[english]{babel}
\usepackage{todonotes}
\usepackage[T1]{fontenc}
\usepackage{appendix}

\usepackage{latexsym,amsfonts,amssymb,mathrsfs,float}
\usepackage{amsthm}
\usepackage{amsmath}

\usepackage{tikz}
\usetikzlibrary{arrows.meta}
\usetikzlibrary{calc}
\usepackage{xcolor}
\usepackage{mathtools,stmaryrd}
\usepackage{ulem}
\usepackage[english]{babel}
\usepackage{todonotes}
\usepackage[T1]{fontenc}
\usepackage{enumitem}

\newcommand{\Prune}{\mathsf{Prune}}
\newcounter{customitem}
\makeatletter
\newcommand{\customitemlabel}[2]{\refstepcounter{customitem}\def\@currentlabel{#1}\label{#2}}
\makeatother
\allowdisplaybreaks

\newtheorem{theorem}{Theorem}[section]
\newtheorem{corollary}[theorem]{Corollary}
\newtheorem{lemma}[theorem]{Lemma}

\newtheorem{proposition}[theorem]{Proposition}
\newtheorem{definition}[theorem]{Definition}

\newtheorem{remark}[theorem]{Remark}

\numberwithin{equation}{section}

\newcommand{\norm}[1]{\left\lVert #1 \right\rVert}
\newcommand{\abs}[1]{\left\lvert #1 \right\rvert}

\endlocaldefs

\begin{document}

\begin{frontmatter}
\title{Loop pruning and downward deviations for maximum local time of discrete-time simple random walks}
\runtitle{Loop Pruning and Downward Deviations}

\begin{aug}
\author[A]{\fnms{Xinyi}~\snm{Li}\ead[label=e1]{xinyili@bicmr.pku.edu.cn}}
\and
\author[B]{\fnms{Yushu}~\snm{Zheng}\ead[label=e2]{yszheng666@gmail.com}}
\address[A]{Beijing International Center for Mathematical Research, Peking University\printead[presep={,\ }]{e1}}
\address[B]{Academy of Mathematics and Systems Science, Chinese Academy of Sciences\printead[presep={,\ }]{e2}}
\end{aug}

\begin{abstract}
We study downward deviations of the maximum local time of the discrete-time simple random walk on $\mathbb{Z}^d$, $d\ge 3$. 
In our previous paper \cite{li2026ldmaxlocal}, the corresponding upper bound was established, while the matching lower bound was left open. 
In the present paper, we prove this lower bound and hence obtain the sharp asymptotic formula for the downward-deviation probability. 
To provide a discrete-time analogue of the jump-chain/holding-time structure used in the continuous-time argument, we introduce a new random structure which we name as {\it loop-pruned random walk} and the associated loop-pruning decomposition, which is also of independent interest.
\end{abstract}



\end{frontmatter}

\section{Introduction}

A basic question in the study of random walks is how strongly the path can concentrate its visits on a small number of sites. 
A natural quantity for measuring this concentration is the maximum local time.
For a random walk $(S_n)_{n\ge 0}$ on $\mathbb{Z}^d$, let
\[
\xi(n,x):=\sum_{k=0}^n \mathbf{1}_{\{S_k=x\}},
\qquad
\xi^*(n):=\max_{x\in\mathbb{Z}^d}\xi(n,x)
\]
denote the local time at $x$ up to time $n$ and the maximum local time, respectively. 
In transient dimensions $d\ge 3$, it is classical, going back to Erd\H{o}s and Taylor~\cite{ErdosTaylor1960}, that
\[
\lim_{n\to\infty}\frac{\xi^*(n)}{\log n}=\alpha
\qquad \text{almost surely}.
\]
Here
\[
\alpha:=-\frac{1}{\log(1-\gamma)};
\qquad
\gamma:=\mathbb{P}(S_n\neq 0\ \text{for all }n\ge 1).
\]
See also R\'ev\'esz~\cite{revesz2004maximum,revesz2013random} for further background on maximum local time and related favorite-site questions in the transient regime.

The almost-sure law above identifies the typical height of the most visited site, but it does not describe the probability of atypical behavior. 
This naturally leads to the large-deviation problem for $\xi^*(n)$.
In the transient regime, the upper and lower tails have rather different structures. 
Upper deviations are produced by creating an unusually thick point, or a small cluster of such points, and are therefore essentially local in nature. 
Downward deviations are global: to keep $\xi^*(n)$ below a level $\beta\alpha\log n$, one must suppress all potential $\beta$-thick points created along a trajectory of length $n$. 
Since a typical site has probability of order $n^{-\beta}$ to become $\beta$-thick, the walk has about $n^{1-\beta}$ opportunities to create such points, suggesting a stretched-exponential downward probability of order $\exp\{-\mathrm{const}\cdot n^{1-\beta}\}$. See our companion work~\cite{li2026ldmaxlocal} for a more detailed introduction and for relevant calculations. See also \cite{li2024large} for discussions on similar phenomena on cover times of tori in $d\geq 3$.

This paper is the concluding and technically most difficult part of our study of the large deviations of the maximum local time of simple random walk in transient dimensions. 
In our companion work~\cite{li2026ldmaxlocal}, the corresponding upper bounds and the continuous-time mechanisms were developed, and the discrete-time downward-deviation upper bound was obtained. 
The remaining problem was the matching lower bound for the discrete-time walk. 
This is the most delicate part of the program because the continuous-time argument relies crucially on the jump-chain/holding-time decomposition: after a site becomes dangerous, one can force the corresponding holding times to be short. 
No such independent holding-time variables are available in discrete time. 
The main contribution of the present paper is to replace that missing structure by a new discrete-time decomposition, based on loop pruning, and thereby complete the sharp large-deviation asymptotics for $\xi^*(n)$.

\subsection{Main results}

We now state the main result of the present paper.

\begin{theorem}\label{thm:downward}
Assume $d\ge 3$. For any $\beta\in(0,1]$ and $u\in\mathbb{R}$, as $n\to\infty$,
\begin{equation}\label{eq:discrete_DD}
\mathbb{P}\bigl(\xi^*(n)\le \beta \alpha \log n + u\bigr)
\ge
\exp\left\{-(1+o(1))\,\gamma\,n\,(1-\gamma)^{\lfloor \beta\alpha\log n+u\rfloor}\right\}.
\end{equation}
\end{theorem}

Combined with the matching upper bound established in~\cite[Theorem~1.2]{li2026ldmaxlocal}, Theorem~\ref{thm:downward} yields the sharp asymptotic formula
\begin{equation}\label{eq:discrete_DD_full}
\mathbb{P}\bigl(\xi^*(n)\le \beta \alpha \log n + u\bigr)
=
\exp\left\{-(1+o(1))\,\gamma\,n\,(1-\gamma)^{\lfloor \beta\alpha\log n+u\rfloor}\right\}.
\end{equation}

Because the local time is integer-valued, the downward asymptotics exhibit a mild arithmetic oscillation. 
To make this explicit, define
\[
c_{\beta,n,u}
:=
n^\beta(1-\gamma)^{-u}(1-\gamma)^{\lfloor \beta\alpha\log n+u\rfloor}
\in[1,(1-\gamma)^{-1}).
\]
Then \eqref{eq:discrete_DD_full} can be rewritten as
\[
\mathbb{P}\bigl(\xi^*(n)\le \beta \alpha \log n + u\bigr)
=
\exp\!\left\{-(1+o(1))\,c_{\beta,n,u}\,\gamma(1-\gamma)^u\,n^{1-\beta}\right\}.
\]
In particular, up to the bounded oscillatory factor $c_{\beta,n,u}$, the downward-deviation probability is of stretched-exponential order $\exp\{-\mathrm{const}\cdot n^{1-\beta}\}$. 
Heuristically, this reflects that a fixed site is $\beta$-thick with probability of order $n^{-\beta}$, so one expects about $n^{1-\beta}$ candidate $\beta$-thick points up to time $n$, and suppressing all of them costs a constant-order factor per candidate.

\subsection{Loop-pruning and lower-bound strategy}
To motivate our approach, let us first recall the continuous-time strategy from~\cite{li2026ldmaxlocal}. 
One fixes a preliminary threshold slightly below the target level and monitors the sites whose local time first reaches this threshold. 
Once a site becomes dangerous, one forces its subsequent holding times to be very short, so that the total additional local-time accumulation remains negligible and the target level is never exceeded.

What makes this strategy work is the jump-chain/holding-time decomposition of the continuous-time walk. 
Indeed, conditionally on the jump chain, the holding times form an i.i.d.\ family of exponential random variables, which makes the above forcing procedure possible. 
Our goal is therefore to extract from the discrete-time simple random walk a structure that can play a similar role.

This is the motivation for introducing the loop-pruned random walk. 
The corresponding structural tool is the loop-pruning decomposition (see \eqref{eq:decomposition}), which separates the trajectory into a retained part and a collection of erased pieces, and thereby provides a discrete-time framework for controlling the return patterns responsible for large local times.

The loop-pruning construction is related in spirit to the classical loop-erasure procedure, in the definition of {\it loop-erased random walk} but differs from it in a fundamental way. 
In loop erasure, one chronologically removes {\it all} loops from the trajectory, thereby producing a self-avoiding path. 
In our setting, by contrast, we fix {\it in advance} a finite family of loop shapes and delete only those loops. 
These loop shapes are taken to be simple (apart from the common starting and ending point, they have no self-intersections), so each time the path completes one of them, the loop to be pruned is unambiguous. 
This restricted pruning rule preserves enough of the original trajectory to retain useful statistical structure. 
In particular, unlike loop erasure, the retained-step process arising from loop pruning enjoys an exponentially mixing property (see Proposition~\ref{prop:mixing}), and this plays an important role in our proof.

We now briefly describe the lower-bound strategy. 
We first fix a finite family of loops and apply the pruning procedure to the path $S[0,n]$. 
The loop family is chosen large enough so that the pruned path typically has maximum local time below the level $\beta\alpha\log n+u$. 
Starting from this pruned path, we then add back the pruned loops one by one. 
Whenever the insertion of a loop would cause the maximum local time to exceed the target level $\beta\alpha\log n+u$, we simply forbid that insertion. 
By construction, this guarantees that the maximum local time never exceeds the prescribed threshold. 
The main point is therefore to estimate the cost of forbidding exactly those dangerous loop insertions, in a way inspired by~\cite{li2026ldmaxlocal}.

\begin{remark}\label{rem:scope-extensions}
We state and prove the main result for nearest-neighbor simple random walk. 
The loop-pruning mechanism is, however, not intrinsically tied to this particular jump distribution. 
For more general transient random walks with finite-range increments, one expects an analogous approach to apply. 
The main additional input would be the corresponding control of long-range returns and block intersections.

For a general increment law $\mu$, a natural parameter is the genuine dimension
\[
d_{\mathrm{gen}}
:=
\dim\bigl(\operatorname{span}(\operatorname{supp}\mu)\bigr).
\]
In the centered case, the relevant return estimates are governed by this dimension; under the usual nondegeneracy assumptions, transience corresponds to $d_{\mathrm{gen}}\ge 3$, and one has polynomial return tails of the form
\[
\mathbb{P}\bigl(0\in S[n,\infty)\bigr)
\asymp n^{-(d_{\mathrm{gen}}/2-1)}.
\]
In the non-centered case, the walk has a drift and the corresponding return tails are typically exponential. 
Both regimes should be compatible with the block-multiplicity estimates used in the present proof, although the constants and some technical details would have to be modified.
We do not pursue such extensions here. 

Let us also mention the corresponding moderate-deviation consequence. If \(a_n\) is positive, \(a_n\to\infty\), and \(a_n=o(\log n)\), then
\begin{align*}
\mathbb P\bigl(\xi^*(n)>\alpha\log n+a_n\bigr)
&=(1+o(1))\,\gamma\,n\,(1-\gamma)^{\lfloor \alpha\log n+a_n\rfloor},\\
\mathbb P\bigl(\xi^*(n)\le \alpha\log n-a_n\bigr)
&=\exp\left\{-(1+o(1))\,\gamma\,n\,(1-\gamma)^{\lfloor \alpha\log n-a_n\rfloor}\right\}.
\end{align*}
These estimates follow from the same methods as the large-deviation results.

\end{remark}

\subsection{Organization}
The rest of the paper is organized as follows. 
In Section~\ref{sec:loop-pruned-decomposition}, we introduce the operation of loop pruning and develop its basic structural properties, including the relevant representations and factorization formulas. 
In Section~\ref{sec:loop-pruning-srw}, we specialize this framework to simple random walk and study the retained portion, the local time of the pruned path, and the conditional distribution of the erased pieces. 
Section~\ref{sec:proof-downward} is devoted to the proof of Theorem~\ref{thm:downward}. 
Finally, Section~\ref{sec:cut} develops the auxiliary structural estimates needed to prove the exponentially mixing property of the retained-step process.
\begin{funding}
XL is supported by the National Key R\&D Program of China (No.~2021YFA1002700). YZ is supported by the China Postdoctoral Science Foundation (Nos.~2023M743721 and 2025T180850).
\end{funding}
\section{Loop-pruned decomposition}\label{sec:loop-pruned-decomposition}
\subsection{Basic definitions}\label{sec:basic}
Throughout this section, we regard a finite sequence 
\[
\eta = (\eta_0, \eta_1, \ldots, \eta_m) \in (\mathbb{Z}^d)^{m+1}, \quad m \in \mathbb{N},
\]
as a \textit{path} in \(\mathbb{Z}^d\).  
The \textit{length} of such a path is defined as \( |\eta| := m \). 
(In particular, a single point \((s_0)\) is a path of length \(0\).)

For any path \(\eta\) and \(x \in \mathbb{Z}^d\), we denote by \(x+\eta\) the translated path
\[
x+\eta := (x+\eta_0, x+\eta_1, \ldots, x+\eta_m).
\]

\subsubsection*{Loop-pruned path}
As a first step, we introduce the notion of loops in \( \mathbb{Z}^d \), along with the operation of pruning loops from a path.

\begin{definition}
A \textit{simple loop} is a path \( e = (e_0, e_1, \ldots, e_m) \) with \( m \in \mathbb{N}^+ \) such that
\[
e_0=e_m=0
\qquad \text{and} \qquad
e_i\neq e_j \ \text{for all } 0\le i<j<m.
\]
\end{definition}
In this paper, whenever we write \textit{loop}, we mean a \textit{simple loop}.

We first define the one-step pruning operation.
\begin{definition}
Let \(\mathbf{s} = (s_n : 0 \le n \le N)\), where \(N\) may be finite or infinite, be a path in \(\mathbb{Z}^d\).  
Let \(E\) be a finite collection of loops. Define
\begin{align}\label{def:tau}
	\tau := \inf \Big\{ i \ge 0 : s[i-|e|,i] = s_{i-|e|} + e \ \text{for some } e \in E,\ |e|\le i \Big\}
\end{align}
to be the first time a loop from \(E\) occurs in \(\mathbf{s}\). 
Note that, since a loop does not revisit \(0\) before its endpoint, at time \(\tau\) there can be at most one such loop.

If \(\tau < \infty\), we define a new path \(\mathbf{s}' = (s'_n : 0\le n\le N-|e|)\) by removing this occurrence of \(e\) from \(\mathbf{s}\):
\[
s'_n := 
\begin{cases}
s_n, & \text{for } n \le \tau - |e|, \\
s_{n + |e|}, & \text{for } \tau - |e| < n \le N-|e|.
\end{cases}
\]
If \(\tau = \infty\), i.e., no such loop appears in \(\mathbf{s}\), we simply set \(\mathbf{s}' := \mathbf{s}\).
This operation is denoted by \(\mathsf{Prune}_1(\mathbf{s}, E) := \mathbf{s}'\). 
\end{definition}

Iterating the one-step operation gives the full pruning map.
\begin{definition}
Define recursively
\[
\mathsf{Prune}_k(\mathbf{s}, E) := \mathsf{Prune}_1(\mathsf{Prune}_{k-1}(\mathbf{s}, E), E), \quad \text{for } k \ge 2,
\]
with the initial step \( \mathsf{Prune}_1(\mathbf{s}, E) \) as above.

If the sequence \( \{ \mathsf{Prune}_k(\mathbf{s}, E) \}_{k \ge 1} \) converges componentwise, we define
\[
\mathsf{Prune}(\mathbf{s}, E) := \lim_{k \to \infty} \mathsf{Prune}_k(\mathbf{s}, E).
\]
\end{definition}
\begin{remark}
Note that if the segment \( (s_{j+1}, s_{j+2}, \ldots, s_{j+m}) \) is removed during the first pruning, this does not imply that the initial segment \( (s_0, s_1, \ldots, s_j) \) will be preserved in all subsequent prunings. It is easy to construct examples in which part of this initial segment is removed in the second or later steps of the pruning procedure. 

This illustrates that the pruning process may affect earlier portions of the path through interactions between overlapping or nested loops.
\end{remark}

\begin{proposition}\label{prop:prune}
For any transient path \(\mathbf{s}\) (i.e., $\abs{s_n} \rightarrow \infty, \; \text{as } n \rightarrow \infty$) and finite set $E$ of loops,
\begin{align*}
    \mathsf{Prune}(\mathbf{s}, E) = \lim_{k \to \infty} \mathsf{Prune}_k(\mathbf{s}, E)\text{ exists.}
\end{align*}
Moreover, $\mathsf{Prune}(\mathbf{s}, E)$ is also a transient path.
\end{proposition}

The proof is similar to the corresponding argument for loop-erased random walk, and is therefore omitted; see \cite[Section 7.3]{lawler2012intersections}. The proposition also extends trivially to finite paths. 

Let \(\eta = (\eta_0, \ldots, \eta_j)\) and \(\eta' = (\eta'_0, \ldots, \eta'_{j'})\) be two finite paths.
If \(\eta_j=\eta'_0\), their \textit{direct concatenation} is defined by
\begin{align}
\eta \circ_d \eta' &:= (\eta_0, \ldots, \eta_j, \eta'_1, \ldots, \eta'_{j'}). \label{eq:direct-concatenation}
\end{align}
If \(\eta'_0=0\), their \textit{translated concatenation} is defined by
\begin{align}
\eta \circ_t \eta' &:= (\eta_0, \ldots, \eta_j, \eta_j + \eta'_1, \ldots, \eta_j + \eta'_{j'}). \label{eq:translated-concatenation}
\end{align}

We now record a simple staging property of loop-pruning, which will be used later and follows directly from the definition.

\begin{proposition}\label{prop:stage}
Let $\mathbf{s}[0,n]$ be a finite path and $E$ a finite set of loops. Then for any $0<m<n$,
\[
\mathsf{Prune}(\mathbf{s}[0,n],E)
= \mathsf{Prune}\Big(\, \mathsf{Prune}(\mathbf{s}[0,m],E)\ \circ_d\ \mathbf{s}[m,n],\ E\Big).
\]
\end{proposition}
In words, pruning may be performed in stages: one first prunes the initial segment and then continues pruning after appending the remaining path, without changing the final result.

\subsubsection*{Loop-pruned decomposition}
We now present a decomposition of a transient or finite path that arises naturally from the loop-pruning procedure

Let \(\mathbf{s}\) be a transient or finite path and let \(E\) be a finite set of loops. Let
\[
\mathbf{s}' = (\mathbf{s}'_n : n \ge 0) := \mathsf{Prune}(\mathbf{s}, E).
\]
For each \(n \in [0,\abs{\mathbf{s}}]\), define
\begin{equation}\label{def:Nn}
N_n=N_n(\mathbf{s},E)
:=
\sup\big\{
m\ge 0:\;
\mathsf{Prune}(\mathbf{s}[0,i],E)[0,m]
=
\mathbf{s}'[0,m]
\ \forall\, i\in[n,\abs{\mathbf{s}}]
\big\}.
\end{equation}
That is, \(N_n\) is the longest initial segment of \(\mathbf{s}'\) that remains unchanged under further pruning of \(\mathbf{s}[0, i]\) for \(i \in [n,\abs{\mathbf{s}}]\).

We define the inverse of \(N_n\) as
\[
N^{-1}(m) := \inf\{ n \in [0,\abs{\mathbf{s}}] : N_n = m \},\quad\text{for }0\le m\le \abs{\mathbf{s}'},
\]
which marks the ``creation time'' of the \(m\)-th step of \(\mathbf{s}'\). 
For \(n\ge 1\), a time \(n\) is of the form \(n=N^{-1}(m)\) for some \(1\le m\le \abs{\mathbf{s}'}\) if and only if \(N_n-N_{n-1}=1\). All the steps \((\mathbf{s}_{n-1}, \mathbf{s}_n)\) with such \(n\) form the ``skeleton'' of \(\mathbf{s}\) that is retained in \(\mathbf{s}'\), while all others are removed.

We define the \textit{\(m\)-th pruned segment} as
\begin{align}\label{def:xi_m}
\begin{aligned}
    \xi_m = \xi_m(\mathbf{s}, E) &:= \mathbf{s}\big[N^{-1}(m),\, N^{-1}(m+1)\big) - \mathbf{s}'_m, \quad \text{for } 0 \le m < \abs{\mathbf{s}'}.
\end{aligned}
\end{align}
(Note that \(\mathbf{s}'_m = \mathbf{s}_{N^{-1}(m)}\).)
If \(\mathbf{s}\) is finite, we further define the final pruned segment as
\[
\xi_{\abs{\mathbf{s}'}} = \xi_{\abs{\mathbf{s}'}}(\mathbf{s}, E) := \mathbf{s}\big[N^{-1}(\abs{\mathbf{s}'}),\, \abs{\mathbf{s}}\big] - \mathbf{s}'_{\abs{\mathbf{s}'}}.
\]
This yields a decomposition of the original path \(\mathbf{s}\) into its skeleton \(\mathbf{s}'\) together with the pruned segments \((\xi_m)_{m \in [0,|\mathbf{s}'|]}\), written as
\begin{align}\label{eq:decomposition}
  \mathbf{s} = \mathbf{s}' \oplus (\xi_m)_{m \in [0,\abs{\mathbf{s}'}]}.
\end{align}
By inserting each segment \(\xi_m\) at the \(m\)-th step of \(\mathbf{s}'\), one recovers the original path \(\mathbf{s}\).

\subsubsection*{Prunings for two-sided paths}

In addition to one-sided paths, we also consider the pruning procedure for two-sided paths.
To this end, we introduce the notion of $E$-cut steps, which allow us to define a two-sided pruning by patching together finite-window prunings.

Recall that in the pruning of a transient one-sided path $\mathbf{s}[0,\infty)$, the retained steps are exactly those $(\mathbf{s}_{n-1},\mathbf{s}_n)$ with $N_n-N_{n-1}=1$.

\begin{definition}[$E$-cut steps]\label{def:E_cut_steps}
Let \(E\) be a finite set of loops and let \(\mathbf{s}=\mathbf{s}(-\infty,\infty)\) be a two-sided path.
A step \((\mathbf{s}_{i-1},\mathbf{s}_i)\), \(i\in\mathbb{Z}\), is called an \textit{\(E\)-cut step} for \(\mathbf{s}\) if for every \(m_1< i\le m_2\), the step \((\mathbf{s}_{i-1},\mathbf{s}_i)\) is retained in the pruning
\[
\mathsf{Prune}\big(\mathbf{s}[m_1,m_2],E\big).
\]
Equivalently, writing $N_n := N_n\big(\mathbf{s}^{(m_1)}[0,m_2-m_1],E\big)$ for the translated path $\mathbf{s}^{(m_1)}_k := \mathbf{s}_{k+m_1}$, the condition is
\[
N_{i-m_1} - N_{i-m_1-1} = 1.
\]
\end{definition}

We define the two-sided pruning map as follows.
\begin{definition}\label{def:erase_two_sided}
Let \(E\) be a finite set of loops and let \(\mathbf{s}=\mathbf{s}(-\infty,\infty)\) be a two-sided path. 
Suppose that the set of indices \(i\in\mathbb Z\) for which \((\mathbf{s}_{i-1},\mathbf{s}_i)\) is an \(E\)-cut step is unbounded both above and below.

We say that a step \((\mathbf{s}_{i-1},\mathbf{s}_i)\) is retained in the pruning of \(\mathbf{s}\) by \(E\) if there exist indices \(m_1< i\le m_2\) such that both boundary steps \((\mathbf{s}_{m_1-1},\mathbf{s}_{m_1})\) and \((\mathbf{s}_{m_2-1},\mathbf{s}_{m_2})\) are \(E\)-cut steps, and \((\mathbf{s}_{i-1},\mathbf{s}_i)\) is retained in the finite-window pruning
\[
\mathsf{Prune}\big(\mathbf{s}[m_1,m_2],E\big).
\]
(It is straightforward to check that this notion does not depend on the particular choice of \(m_1,m_2\).)

Collecting all retained steps in chronological order yields a two-sided path (defined up to translation), which we denote by \(\mathsf{Prune}(\mathbf{s},E)\).
\end{definition}

\subsection{Structure of pruned segments}\label{sec:segment}
Return to the transient or finite one-sided path $\mathbf{s}$.
We next focus on the internal structure of each pruned segment \(\xi_m\) in the loop-pruned decomposition. 

By construction, for the pruned path \(\mathbf{s}' = \mathsf{Prune}(\mathbf{s}, E)\), each pruned segment \(\xi_m\) satisfies the following properties:
\begin{enumerate}[label=(\roman*)]
    \item\label{item:seg1} It is \textit{\(E\)-erasable}, meaning that
    \[
    \mathsf{Prune}(\eta,E)=(0).
    \]

    \item\label{item:seg2} It is compatible with the retained prefix \(\mathbf{s}'[0,m]\): after appending \(\eta\) to this prefix, all steps of \(\mathbf{s}'[0,m]\) remain retained, namely
    \begin{equation}\label{eq:compatible-retained-prefix}
    N_n(\mathbf{s}'[0, m] \circ_t \eta, E) \ge m \quad \text{for all } n \in [m, m + |\eta|].
    \end{equation}
\end{enumerate}

This motivates the following definition.
\begin{definition}[Segment classes]\label{def:Seg-classes}
For any finite set $E$ of loops, we call a finite path \(s'\) \textit{\(E\)-pruned} if \(\mathsf{Prune}(s',E)=s'\). Notice that, if \(\mathbf{s}'=\mathsf{Prune}(\mathbf{s},E)\), then every retained prefix \(\mathbf{s}'[0,m]\) is \(E\)-pruned.
Let \(\mathsf{Seg}(E)\) denote the collection of finite \(E\)-erasable paths.
For an \(E\)-pruned finite path $s'$, define
\begin{align}\label{def:Seg}
\begin{aligned}
	\mathsf{Seg}(s',E):=\big\{\eta\in \mathsf{Seg}(E): N_n(s'\circ_t \eta,E)\ge |s'|\quad\forall\,n\in[|s'|,|s'|+\abs{\eta}]\big\}.	
\end{aligned}
\end{align}

We further set
\begin{align}\label{def:prod_seg}
\begin{aligned}
	\mathsf{Seg}^{\otimes}(s',E)&:=\prod_{m=0}^{\abs{s'}} \mathsf{Seg}(s'[0,m], E).
\end{aligned}
\end{align}
\end{definition}
\begin{proposition}\label{pro:one-to-one}
Let $E$ be a finite set of loops and let $s'$ be an \(E\)-pruned finite path. Define
\[
\mathsf{P}(s',E) := \left\{ \text{finite paths } s : \mathsf{Prune}(s, E) = s' \right\}.
\]
Then the map
\begin{align*}
	s \mapsto \left( \xi_m(s, E) : m \in [0,\abs{s'}] \right)
\end{align*}
defines a one-to-one correspondence between \(\mathsf{P}(s',E)\) and $\mathsf{Seg}^{\otimes}(s',E)$.
\end{proposition}

The proof is straightforward and omitted.

\subsection{Representations of $\mathsf{Seg}(E)$}\label{sec:represent}
From this subsection onward, we assume that the finite loop family \(E\) satisfies \(|e|\ge2\) for every \(e\in E\).

Note that all segment classes \(\mathsf{Seg}(s',E)\) are contained in the common set \(\mathsf{Seg}(E)\).  
We now study the internal structure of \(\mathsf{Seg}(E)\), presenting two useful representations -- tree representation and elementary-sequence representation.

The idea behind both representations is simple. 
An \(E\)-erasable segment can be viewed as being generated by successively inserting loops from the prescribed family \(E\). 
The tree representation records the genealogy of these insertions as a tree carrying marks that record the types of the inserted loops. 
The elementary-sequence representation is then a marked analogue of the classical Ulam--Harris representation: besides the vertex address, each vertex also carries a sequence encoding the marks, namely the corresponding loop types.

\subsubsection{Tree representation}
In this section, we relate the class \(\mathsf{Seg}(E)\) to a specific family of marked trees, which we will formally define as \textit{E-segmented marked trees}. 
\subsubsection*{$E$-segmented marked trees}
Let \(E\) be a finite set of loops such that \(|e|\ge 2\) for every \(e\in E\). 
Fix once and for all an enumeration
\[
E=\{e^{(1)},\ldots,e^{(\# E)}\}.
\]
We consider marked trees \(T^*=(T,\sigma)\), where
\begin{itemize}
	\item \(T\) is a finite rooted ordered tree with vertex set \(V(T)\);
	\item \(\sigma\) is a labeling map
	\[
	\sigma: V(T)\setminus\{\mathrm{root}\}\longrightarrow \{1,\ldots,\# E\},
	\]
	assigning a label in \(\{1,\ldots,\# E\}\) to each non-root vertex of \(T\).
\end{itemize}

Before defining the class of \(E\)-segmented marked trees, we first introduce the notion of an \textit{elementary sequence}.

\begin{definition}[Elementary sequence]\label{def:elementary-sequence}
For \(j\in\{1,\ldots,\#E\}\), set
\(
\lambda_j:=|e^{(j)}|-1.
\)
We also set \(\lambda_0:=0\).
The elementary sequence of type \(j\) is the sequence consisting of \(\lambda_j\) repetitions of the label \(j\), namely
\[
(\overbrace{j,j,\ldots,j}^{\lambda_j}).
\]
The number \(\lambda_j\) corresponds to the possible insertion positions along \(e^{(j)}\), excluding the common endpoint at \(0\).
\end{definition}

\begin{definition}[$E$-segmented marked trees]
	We define the collection \(\mathbf{T}^*=\mathbf{T}^*_E\), also referred to as the \textit{$E$-segmented marked trees}, as the set of all marked trees \(T^*=(T,\sigma)\) satisfying the following condition: for every vertex \(v\in V(T)\), let \(u_1<u_2<\cdots< u_L\) be the children of \(v\), ordered according to the tree ordering. Then the sequence of their labels, \((\sigma(u_1), \ldots, \sigma(u_L))\), must be a concatenation of elementary sequences. That is, it must be of the form \(K_1 \circ K_2 \circ \cdots \circ K_P\) for some integer \(P \ge 0\), where each \(K_j\) is an elementary sequence of some type.

		See Figure~\ref{fig:tree1} for an example of an $E$-segmented marked tree.
		\end{definition}

		\begin{figure}[!b]
		\centering
		\resizebox{0.441\textwidth}{!}{
		\begin{tikzpicture}[
		  every node/.style={circle, draw, minimum size=6mm, fill=gray!20, inner sep=0pt},
		  level distance=12mm, sibling distance=12mm]

	\def\rad{0.38}
	\def\step{0.16}

	\node (r) at (0,0) {};

	\node (a1) at (-2.5,1.2) {};
	\node (a2) at (0,1.2)   {};
	\node (a3) at (2.5,1.2) {};

	\draw (r) -- (a1);
	\draw (r) -- (a2);
	\draw (r) -- (a3);

	\node (b1) at (-3,2.4) {};
	\node (b2) at (-2,2.4) {};
	\draw (a1) -- (b1);
	\draw (a1) -- (b2);

	\node (c1) at (1.8,2.4) {};
	\node (c2) at (2.5,2.4) {};
	\node (c3) at (3.2,2.4) {};
	\draw (a3) -- (c1);
	\draw (a3) -- (c2);
	\draw (a3) -- (c3);

	\node (d1) at (1.8,3.6) {};
	\draw (c1) -- (d1);

	\node[draw=none, fill=none, left=-8.5pt]  at (a1) {\(1\)};
	\node[draw=none, fill=none, right=-8.5pt] at (a2) {\(2\)};
	\node[draw=none, fill=none, right=-8.5pt] at (a3) {\(2\)};
	\node[draw=none, fill=none, left=-8.5pt]  at (b1) {\(2\)};
	\node[draw=none, fill=none, left=-8.5pt]  at (b2) {\(2\)};
	\node[draw=none, fill=none, left=-8.5pt]  at (c1) {\(2\)};
	\node[draw=none, fill=none, left=-8.5pt]  at (c2) {\(2\)};
	\node[draw=none, fill=none, left=-8.5pt]  at (c3) {\(1\)};
	\node[draw=none, fill=none, left=-8.5pt]  at (d1) {\(1\)};

	\draw[very thick, red, rounded corners=4pt, ->, >=stealth]
	  ($(r)+(-\rad,-\rad)$)
	   -- ($(a1)+(-\rad,-\rad)$); 
	\draw[very thick, red, rounded corners=4pt, ->, >=stealth]   
	   ($(a1)+(-\rad,-\rad)$) -- ($(a1)+(-\rad,\rad)$)
	  -- ($(b1)+(-\rad,-\rad)$) -- ($(b1)+(-\rad,\rad)$) -- ($(b1)+(\rad,\rad)$)
	  -- ($(b2)+(-\rad,\rad)$) -- ($(b2)+(\rad,\rad)$);
	 \draw[very thick, red, rounded corners=4pt, ->, >=stealth]  
	  ($(b2)+(\rad,\rad)$) -- ($(b2)+(\rad,-\rad)$)
	  -- ($(a1)+(\rad,\rad)$) -- ($(a1)+(\rad,0)$) --  ($(r)+(-\rad,\rad)$);
	  \draw[very thick, red, rounded corners=4pt, ->, >=stealth]
	   ($(r)+(-\rad,\rad)$) -- ($(a2)+(-\rad,-\rad)$) -- ($(a2)+(-\rad,\rad)$) -- ($(a2)+(\rad,\rad)$)
	   -- ($(a3)+(-\rad,\rad)$);
	   \draw[very thick, red, rounded corners=4pt, ->, >=stealth]
	 ($(a3)+(-\rad,\rad)$) -- ($(c1)+(-\rad,-\rad)$) -- ($(c1)+(-\rad,\rad)$) -- ($(d1)+(-\rad,-\rad)$)
	  -- ($(d1)+(-\rad,\rad)$) -- ($(d1)+(\rad,\rad)$) -- ($(d1)+(\rad,-\rad)$)
	  -- ($(c1)+(\rad,\rad)$) 
	  -- ($(c2)+(\rad/1.2,\rad)$);
	  \draw[very thick, red, rounded corners=4pt, ->, >=stealth]
	  ($(c2)+(\rad/1.2,\rad)$) -- ($(c2)+(\rad/1.2,-\rad)$) -- ($(a3)+(0,\rad)$) --($(a3)+(\rad/2,\rad)$)
	  -- ($(c3)+(-\rad/1.2,-\rad)$) -- ($(c3)+(-\rad/1.2,\rad)$) -- ($(c3)+(\rad,\rad)$)
	  -- ($(c3)+(\rad,-\rad)$)
	  -- ($(a3)+(\rad,\rad)$) -- ($(a3)+(\rad,-\rad)$);
	  \draw[very thick, red, rounded corners=4pt, ->, >=stealth]
	  ($(a3)+(\rad,-\rad)$)
	  -- ($(r)+(\rad,-\rad)$) -- ($(r)+(-\rad/1.5,-\rad)$);

	\fill[black] ($(r)+(-\rad,-\rad)$) circle (1.2pt);

		\end{tikzpicture}
		}
		\caption{An example of an $E$-segmented marked tree from \(\mathbf{T}^*\) with \(E=\{e^{(1)},e^{(2)}\}\), where \(|e^{(1)}|=2\) and \(|e^{(2)}|=3\), together with the associated $E$-depth-first traversal.}
		\label{fig:tree1}
		\end{figure}
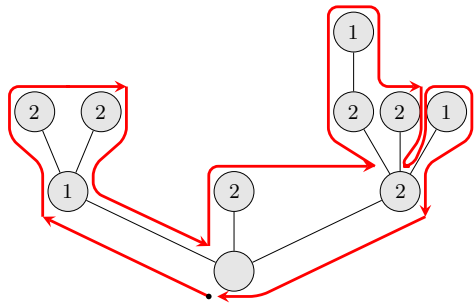

\subsection*{$E$-depth-first traversal}
We now introduce a variant of the classical depth-first traversal, adapted to the segmented structure. 

\begin{definition}[$E$-depth-first traversal]\label{def:E-DFS}
Let $T^*=(T,\sigma)\in \mathbf{T}^*$ be an $E$-segmented marked tree with root $\rho$.  
We define the \textit{$E$-depth-first traversal} of $T^*$ as a finite sequence 
\[
f=f_{T^*}:\{0,1,\ldots,M\}\longrightarrow V(T),
\]
constructed recursively as follows.

\begin{enumerate}
    \item The sequence starts at the root:
    \[
    f(0)=\rho.
    \]
    \item Suppose $f(0),\ldots,f(m)$ have already been defined for some $m\in \mathbb{N}$. We then determine $f(m+1)$ according to the position of $f(m)$ in the tree:
    \begin{itemize}
        \item[(i)] If $f(m)$ has at least one child not contained in $\{f(0),\ldots,f(m)\}$, then $f(m+1)$ is the leftmost such child with respect to the ordering of $T$.
        \item[(ii)] If all children of $f(m)$ have already appeared in the sequence and $f(m)\neq \rho$, let $u_1<u_2<\cdots <u_L$ denote $f(m)$ and all its siblings (i.e. the children of the parent of $f(m)$). Assume $f(m)=u_j$ and that the labels of $u_1,\ldots, u_L$ decompose into a concatenation of elementary sequences
        \[
        (\sigma(u_1),\ldots, \sigma(u_L))=K_1\circ K_2\circ \cdots \circ K_P.
        \]
        If the $j$-th element of this concatenation belongs to $K_\ell$, then
        \begin{itemize}
            \item if it is not the last element of $K_\ell$, we define $f(m+1)=u_{j+1}$, the next sibling of $f(m)$;
            \item if it is the last element of $K_\ell$, we define $f(m+1)$ as the parent of $f(m)$.
        \end{itemize}
    \end{itemize}
    \item The recursion terminates once the root $\rho$ is reached and all children of $\rho$ have appeared in the sequence.
\end{enumerate}

Figure~\ref{fig:tree1} illustrates this construction. The red contour line is drawn in analogy with the usual contour of a depth-first traversal: it starts from the root and follows the edges of the tree along the traversal. The only difference is that, when encountering a family of vertices corresponding to an elementary sequence, the contour does not immediately return to their parent after each visit; instead, it keeps moving through the sequence until the traversal has completely finished with the last vertex of this family (i.e., after its entire subtree has been explored), only then does it return to the parent. 

The order in which the vertices are visited along this contour line coincides with the $E$-depth-first traversal defined above.
\end{definition}

\subsubsection*{Canonical tree representation}
We now assign to every path in the class \(\mathsf{Seg}(E)\) a canonical representation by an $E$-segmented marked tree in $\mathbf{T}^*$.

We begin by introducing the insertion operation. Let $\eta=(\eta_0,\ldots,\eta_k)$ be a path and $e=(e_0,\ldots,e_\ell)$ a loop.  
For any $0\le j\le k$, define
\begin{equation}\label{eq:loop-insertion-operation}
\eta\triangleleft_j e := \big(\eta_0,\ldots,\eta_{j-1},\eta_j+e_0,\ldots,\eta_j+e_\ell,\eta_{j+1},\ldots,\eta_k\big).
\end{equation}
In words, the loop $e$ is inserted into $\eta$ immediately after its $j$-th step.

With this notation in place, we can recursively define the map
\[
\mathsf{T}^* = \mathsf{T}^*_E: \mathsf{Seg}(E) \longrightarrow \mathbf{T}^*.
\] 
\begin{itemize}
    \item For the trivial path \((0)\), set \(\mathsf{T}^*((0))\) to be the one-vertex tree consisting only of the root.
    \item Suppose \(\mathsf{T}^*(\eta)\) has been defined for all \(\eta\in \mathsf{Seg}(E)\) with \(|\eta|\le m-1\).  
    Let \(\eta\in \mathsf{Seg}(E)\) have length \(|\eta|=m\).  
    The pruning of a single loop $\mathsf{Prune}_1(\eta,E)$ gives the decomposition
    \[
    \eta = \mathsf{Prune}_1(\eta,E)\triangleleft_{\tau-|e|} e,
    \qquad e=e^{(i)},
    \]
    where $\tau$ is defined in~\eqref{def:tau}.  
    By the inductive hypothesis, \(\mathsf{T}^*(\mathsf{Prune}_1(\eta,E))\) is already defined.  
    Let $f=f(\mathsf{T}^*(\mathsf{Prune}_1(\eta,E)))$ denote its $E$-depth-first traversal.  
    We then define \(\mathsf{T}^*(\eta)\) by attaching \(|e|-1\) ordered children to the vertex \(f(\tau-|e|)\), placing them immediately after all the children already visited in $\{f(0),\ldots, f(\tau-|e|)\}$ and before all remaining children, and assigning each new child the label $i$. 
\end{itemize}
This completes the recursive definition of the map $\mathsf{T}^*$ for all paths in $\mathsf{Seg}(E)$.
We record the resulting injectivity statement.

\begin{proposition}
\label{pro:tree_represent}
The map $\mathsf{T}^*$ defines an injection from $\mathsf{Seg}(E)$ into $\mathbf{T}^*$.
\end{proposition}
\begin{remark}[Link between $\eta$ and the $E$-depth-first traversal of $\mathsf{T}^*(\eta)$]\label{rem:link}
The construction of the map $\mathsf{T}^*$ reveals a one-to-one correspondence between the steps of a path $\eta$ and the steps of the $E$-depth-first traversal of its tree $\mathsf{T}^*(\eta)$. Under this correspondence, when the traversal moves across a family of vertices forming an elementary sequence with label $k$, the path $\eta$ is exactly tracing an occurrence of the loop $e^{(k)}$.
\end{remark}
\begin{remark}[Reconstruct $\eta$ from $\mathsf{T}^*(\eta)$]\label{rem:reconstruct1}
To provide a more intuitive illustration of the map $\mathsf{T}^*$, we again consider the $E$-segmented marked tree in Figure~\ref{fig:tree1}. Suppose it is the canonical tree representation of a path $\eta$. We now describe how to reconstruct $\eta$ from its tree representation.

The tree encodes the path
\[ \eta = \big(e^{(1)} \triangleleft_{1} e^{(2)}\big) \ \circ_t\ \Big\{\, e^{(2)} \triangleleft_{2} \big[\,(e^{(2)} \triangleleft_{1} e^{(1)}) \circ_t e^{(1)}\,\big] \Big\}. \]
The reconstruction proceeds as follows:

\begin{enumerate}[label=(\roman*)]
\item The root has three children labeled $1,2,2$, indicating that, starting from the trivial path $(0)$, we first insert $e^{(1)}$ (creating a child labeled $1$) and then insert $e^{(2)}$ (creating two children labeled $2$);
\item The left child of the root (label $1$) has two children labeled $2$, corresponding to the insertion of $e^{(2)}$ at that position, which realizes the operation $e^{(1)} \triangleleft_1 e^{(2)}$;
\item The right child of the root (label $2$) has children labeled $2,2,1$, meaning that at the second step of that loop, we first insert $e^{(2)}$ (two children, placed first) and then another $e^{(1)}$ (one child, placed last), realizing $e^{(2)} \triangleleft_{2}(\cdots)$;
\item The leftmost grandchild under that right child has one child labeled $1$, encoding the concatenated $\triangleleft_{1}e^{(1)}$ inside the brackets.
\end{enumerate}
Thus, the canonical tree representation fully encodes the path $\eta$, with vertex labels and hierarchical structure reflecting precisely the order and positions of the loop insertions.
\end{remark}

\subsubsection{Elementary-sequence representation}\label{sec:ES_represent}

We now introduce a second representation of paths in $\mathsf{Seg}(E)$, which records the structure of the canonical $E$-segmented marked tree in terms of the types of elementary sequences attached to each vertex.

Let $T^* = (T,\sigma) \in \mathbf{T}^*$ be an $E$-segmented marked tree.  
For each vertex $v \in V(T)$, consider its ordered list of children, which determines a concatenation of elementary sequences
\[
K_1 \circ K_2 \circ \cdots \circ K_q,
\]
where each $K_j$ is of type $i_j$.  

We define
\[
\mathbf{K}(k) := \big\{(a_j)_{j \ge 1} \in \{0,\ldots,k\}^\infty : a_j = 0 \text{ for all } j \ge \inf\{\ell \ge 1 : a_\ell = 0\} \big\},
\]
the space of absorbing-zero sequences over $\{0,\ldots,k\}$.  
To each vertex $v$ we associate the sequence
\[
\mathbf{k}(v) := (i_1,\ldots,i_q,0,0,\ldots) \in \mathbf{K}(\# E).
\]

We view $V(T)$ canonically as a prefix-stable subset of
\[
\mathcal{U} := \bigcup_{n=0}^\infty (\mathbb{N}^+)^n, \qquad (\mathbb{N}^+)^0 := \{\emptyset\},
\]
via the usual \textit{Ulam--Harris identification} of rooted ordered trees.
For $v \in \mathcal{U}\setminus V(T)$, we set $\mathbf{k}(v) := (0,0,\ldots)$.  

This construction yields a map
\[
\mathcal{E}: \mathbf{T}^* \longrightarrow \mathbf{K}(\# E)^{\mathcal{U}}, 
\qquad 
T^* \mapsto (\mathbf{k}(v))_{v \in \mathcal{U}},
\]
which we call the \textit{elementary-sequence representation} (or \textit{ES-representation}) of the $E$-segmented marked tree.  
For a path $\eta \in \mathsf{Seg}(E)$, we define its \textit{ES-representation} by
\[
\mathcal{E}(\eta):=\mathcal{E}(\mathsf{T}^*(\eta)) \in \mathbf{K}(\# E)^{\mathcal{U}}.
\]

Since \(\mathsf{T}^*:\mathsf{Seg}(E)\to\mathbf{T}^*\) is injective by Proposition~\ref{pro:tree_represent}, and \(\mathcal{E}:\mathbf{T}^*\to\mathbf{K}(\#E)^{\mathcal U}\) is injective by construction, we immediately obtain the following.

\begin{proposition}\label{prop:ES-injection}
The map \(\mathcal{E}:\mathsf{Seg}(E)\to \mathbf{K}(\# E)^{\mathcal{U}}\), \(\eta\mapsto \mathcal{E}(\eta)\), is injective.
\end{proposition}

\begin{remark}[Reconstruct \(\eta\) from \(\mathcal E(\eta)\)]\label{rem:reconstruct2}
Finally, let us note how one can recover the path \(\eta\) from its ES-representation. 
For each \(v\in\mathcal U\), the associated sequence records, in order, the loop types to be inserted at the position encoded by \(v\).
Starting from the trivial path \((0)\), one reads these positions recursively and inserts the prescribed loops in the corresponding order.
This reconstructs the whole segment \(\eta\).
\end{remark}

Recall from Definition~\ref{def:prod_seg} the product space \(\mathsf{Seg}^{\otimes}(s',E)\). We now extend the ES-representation to \(\mathsf{Seg}^{\otimes}(s',E)\) componentwise.

\begin{definition}[Product ES-representation]\label{def:E-tensor}
For any \(E\)-pruned finite path \(s'\), define
\[
\mathcal{E}^{\otimes}=\mathcal{E}^{\otimes}_{s'}:\mathsf{Seg}^{\otimes}(s',E)\longrightarrow \big(\mathbf{K}(\#E)^{\mathcal U}\big)^{|s'|+1},
\quad
(\eta_0,\ldots,\eta_{|s'|})\longmapsto \big(\mathcal{E}(\eta_0),\ldots,\mathcal{E}(\eta_{|s'|})\big).
\]
\end{definition}

\begin{proposition}\label{prop:E-tensor-injective}
For every \(E\)-pruned finite path \(s'\), the map \(\mathcal{E}^{\otimes}\) is injective.
\end{proposition}

\subsection{Decomposition and fiber factorization}\label{subsec:decomposition-fiber-factorization}

The goal of this subsection is to decompose the path according to vertices of the \(E\)-segmented marked tree. 
Roughly speaking, a vertex \(v\) of the tree identifies a time on the path. 
Relative to this time, the path is split into the history before that time, the loop cluster inserted there, and the part of the path that comes afterward. 
We make this decomposition precise below and then use it to obtain a fiber factorization theorem. 
In Sections~\ref{sec:loop-pruning-srw} and~\ref{sec:proof-downward}, this will describe the possible future configurations when loops are inserted back into the pruned path step by step.

Recall from Definition~\ref{def:elementary-sequence} that \(\lambda_j=|e^{(j)}|-1\).

\subsubsection*{Vertex and boundary addresses}

We begin by introducing vertex and boundary addresses, which are required to formulate the decomposition relative to $v$.

\begin{definition}[Vertex addresses]
Let \(W\in \mathcal E(\mathsf{Seg}(E))\), and let \(T^*\in\mathbf T^*\) be the marked tree corresponding to \(W\). 
We define \(V(W)\subset\mathcal U\) as the set of Ulam--Harris addresses of the vertices of \(T^*\).
Elements of \(V(W)\) are called vertex addresses of \(W\).
\end{definition}

\begin{definition}[Boundary addresses]\label{def:boundary-addresses}
For \(u\in V(W)\), write
\[
W(u)=(a_1,\dots,a_\ell,0,0,\dots),
\qquad a_i\neq 0 \ \ (1\le i\le \ell).
\]
We then define
\[
V^{\partial}(W):=
\bigl\{\, u\,b_k(u) : u\in V(W),\ 0\le k\le \ell,\ b_k(u)=1+\sum_{i=1}^k \lambda_{a_i} \,\bigr\}.
\]
Elements of \(V^{\partial}(W)\) are called boundary addresses of \(W\).

For \(v\in V^{\partial}(W)\), write \(v=u\,b_k(u)\) in the above form, and define
\[
\operatorname{pre}_W(v):=k,
\qquad
\operatorname{Next}_W(v):=a_{k+1}.
\]
Equivalently, writing \(\mathrm{par}(v)\) for the parent of \(v\), namely \(\mathrm{par}(v):=(v_1,\ldots,v_{n-1})\) if \(v=(v_1,\ldots,v_n)\), we have \(v=\mathrm{par}(v)j\) with
\(
j=1+\sum_{i=1}^{\operatorname{pre}_W(v)}\lambda_{a_i}.
\)
\end{definition}
In words, \(V^{\partial}(W)\) records the block-boundary positions at each vertex address \(u\):
the start of the first block (\(k=0\)), the starts immediately after the first \(k\)
completed blocks (\(1\le k< \ell\)), and the first position after the last block (\(k=\ell\)).
For \(k<\ell\), the corresponding address belongs to \(V(W)\); for \(k=\ell\), it does not belong to \(V(W)\).
The number \(\operatorname{pre}_W(v)\) is the number of elementary blocks of \(W(\mathrm{par}(v))\) completed before the boundary \(v\), while \(\operatorname{Next}_W(v)\) is the type of the bottom-most loop in the next elementary block at \(\mathrm{par}(v)\); the value \(0\) means that there is no next block.

\subsubsection*{Decomposition at a boundary address}

We now state the structural decomposition relative to a fixed boundary address $v$, which will be the basis of the later factorization result.
For $w,v\in\mathcal U$, we write
\begin{align*}
w\prec_{\mathrm{anc}} v
&\quad\Longleftrightarrow\quad
w \text{ is a proper prefix of } v,\\
w\prec_{\mathrm{lex}} v
&\quad\Longleftrightarrow\quad
w \text{ is lexicographically smaller than } v.
\end{align*}
We use \(\preceq_{\mathrm{anc}}\) and \(\preceq_{\mathrm{lex}}\) for the corresponding non-strict orders.

\begin{definition}[Relative address decomposition]
\label{def:V-decomposition-relative-v}
Let $W\in \mathbf{K}(\#E)^{\mathcal U}$ and $v\in V^{\partial}(W)$, and write
$v=(v_1,\ldots,v_\ell)$.
For two addresses $x,y\in\mathcal U$, write $x\wedge y$ for their longest common prefix
(their youngest common ancestor in the Ulam--Harris tree).

Define
\[
V_{\mathrm{exp}}^{v,W}
=
\big\{u\in V(W):\ u\prec_{\mathrm{lex}} v\big\}.
\]
Define the \(\ell\)-th younger-sibling set by
\[
\widehat{\mathcal D}_{\ell}^{v,W}
:=
\{u\in V(W): \mathrm{par}(v)\prec_{\mathrm{anc}}u,\ v\preceq_{\mathrm{lex}} u\},
\quad
\mathcal D_{\ell}^{v,W}:=\widehat{\mathcal D}_{\ell}^{v,W}\cup\{\mathrm{par}(v)\}.
\]
Namely, $\mathcal D_{\mathrm{par}}^{v,W}$ is the subtree generated by descendants of $v$ together with those of its younger siblings, with ``root'' \(\mathrm{par}(v)\); \(\widehat{\mathcal D}_{\mathrm{par}}^{v,W}\) is the same subtree with the root removed.

For each $1\le q\le \ell-1$, define the $q$-th younger-sibling set
\[
\mathcal Y_q^{v,W}
:=
\big\{u\in V(W):\ |u|=q,\ |u\wedge v|=q-1,\ v\prec_{\mathrm{lex}}u\big\}.
\]
For $u\in\bigcup_{q=1}^{\ell-1} \mathcal Y_q^{v,W}$, define its descendant block
\[
\mathcal D_u^W:=\{x\in V(W):\ u\preceq_{\mathrm{anc}} x\}.
\]

Then
\[
V(W)
=
V_{\mathrm{exp}}^{v,W}
\ \sqcup
\widehat{\mathcal D}_{\mathrm{par}}^{v,W}
\ \sqcup
\bigsqcup_{q=1}^{\ell-1}\ \bigsqcup_{u\in\mathcal Y_q^{v,W}} \mathcal D_u^W.
\]
The decomposition is illustrated in Figure~\ref{fig:V-decomposition-relative-v}. In particular, these sets form a partition of $V(W)$.
\end{definition}

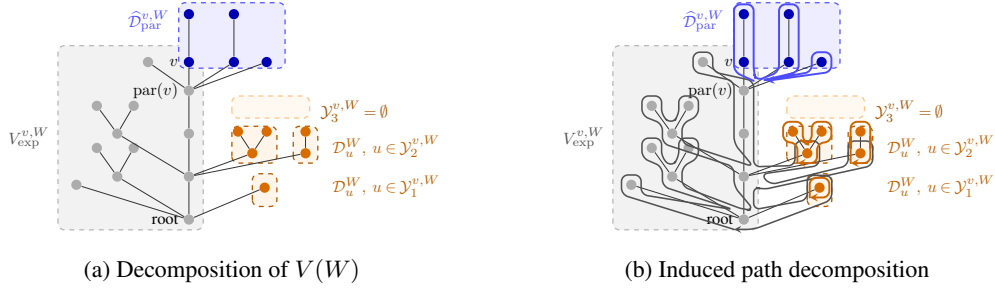
\begin{figure}[H]
\centering
\begin{minipage}[t]{0.48\textwidth}
\centering
\resizebox{0.86\linewidth}{!}{
\begin{tikzpicture}[
  scale=0.88,
  every node/.style={font=\small},
  vertex/.style={circle, draw=none, minimum size=2.2mm, inner sep=0pt, fill=black!70},
  chain/.style={vertex, fill=blue!70!black},
  side/.style={vertex, fill=orange!85!black},
  old/.style={vertex, fill=gray!65},
  edge/.style={draw=black!80, line width=0.55pt},
  block/.style={rounded corners=4pt, thick}
]
\filldraw[block, fill=gray!10, draw=gray!55, dashed] (-3.2,-0.25) rectangle (0.35,4.25);
\filldraw[block, fill=orange!8, draw=orange!70!black, dashed] (1.55,0.32) rectangle (2.15,1.12);
\filldraw[block, fill=orange!8, draw=orange!70!black, dashed] (1.05,1.38) rectangle (2.15,2.28);
\filldraw[block, fill=orange!8, draw=orange!70!black, dashed] (2.55,1.38) rectangle (3.15,2.28);
\filldraw[block, fill=orange!5, draw=orange!45, dashed] (1.05,2.48) rectangle (2.95,3.02);
\filldraw[block, fill=blue!7, draw=blue!65, dashed] (-0.25,3.7) rectangle (2.35,5.3);

\node[old, label=left:{root}] (root) at (0,0) {};
\node[old] (a1) at (0,1.05) {};
\node[old] (a2) at (0,2.1) {};
\node[old, label=left:\(\mathrm{par}(v)\)] (p) at (0,3.15) {};
\node[chain, label=left:\(v\)] (v) at (0,3.85) {};
\node[chain] (vc) at (0,5.02) {};
\draw[edge] (root)--(a1)--(a2)--(p)--(v)--(vc);

\node[old] (o1) at (-1.75,1.05) {};
\node[old] (o1a) at (-2.25,1.75) {};
\node[old] (o1b) at (-1.35,1.75) {};
\node[old] (o2) at (-1.75,2.1) {};
\node[old] (o2a) at (-2.25,2.78) {};
\node[old] (o2b) at (-1.35,2.78) {};
\node[old] (o3) at (-2.75,0.85) {};
\node[old] (op) at (-1.0,3.85) {};
\draw[edge] (root)--(o1);
\draw[edge] (o1)--(o1a);
\draw[edge] (o1)--(o1b);
\draw[edge] (a1)--(o2);
\draw[edge] (o2)--(o2a);
\draw[edge] (o2)--(o2b);
\draw[edge] (root)--(o3);
\draw[edge] (p)--(op);
\node[anchor=east, black!68] at (-3.28,2.0) {\(V_{\mathrm{exp}}^{v,W}\)};

\node[side] (yr1) at (1.85,0.78) {};
\draw[edge] (root)--(yr1);
\node[anchor=west, orange!75!black] at (3.38,0.82) {\(\mathcal D_u^W,\ u\in\mathcal Y_1^{v,W}\)};

\node[side] (y1) at (1.55,1.62) {};
\node[side] (y1a) at (1.2,2.15) {};
\node[side] (y1b) at (1.9,2.15) {};
\node[side] (y1c) at (2.85,1.62) {};
\node[side] (y1ca) at (2.85,2.15) {};
\draw[edge] (a1)--(y1);
\draw[edge] (y1)--(y1a);
\draw[edge] (y1)--(y1b);
\draw[edge] (a1)--(y1c);
\draw[edge] (y1c)--(y1ca);
\node[anchor=west, orange!75!black] at (3.38,1.78) {\(\mathcal D_u^W,\ u\in\mathcal Y_2^{v,W}\)};

\node[anchor=west, orange!65!black] at (3.08,2.65) {\(\mathcal Y_3^{v,W}=\emptyset\)};

\node[chain] (ys1) at (1.1,3.85) {};
\node[chain] (ys2) at (1.9,3.85) {};
\node[chain] (ys1a) at (1.1,5.02) {};
\draw[edge] (p)--(ys1);
\draw[edge] (p)--(ys2);
\draw[edge] (ys1)--(ys1a);

\node[anchor=east, blue!70] at (-0.35,4.85) {\(\widehat{\mathcal D}_{\mathrm{par}}^{v,W}\)};
\end{tikzpicture}}
\par\smallskip\textup{(a) Decomposition of \(V(W)\)}
\end{minipage}\hfill
\begin{minipage}[t]{0.48\textwidth}
\centering
\resizebox{0.86\linewidth}{!}{
\begin{tikzpicture}[
  scale=0.88,
  every node/.style={font=\small},
  vertex/.style={circle, draw=none, minimum size=2.2mm, inner sep=0pt, fill=black!70},
  chain/.style={vertex, fill=blue!70!black},
  side/.style={vertex, fill=orange!85!black},
  old/.style={vertex, fill=gray!65},
  edge/.style={draw=black!80, line width=0.55pt},
  block/.style={rounded corners=4pt, thick},
  contourgray/.style={black!68, line width=0.95pt, rounded corners=4pt, ->, >=stealth, line cap=round, line join=round},
  contourorange/.style={orange!80!black, very thick, rounded corners=4pt, ->, >=stealth, line cap=round, line join=round},
  contourblue/.style={blue!70, very thick, rounded corners=4pt, ->, >=stealth, line cap=round, line join=round}
]
\filldraw[block, fill=gray!10, draw=gray!55, dashed] (-3.2,-0.25) rectangle (0.35,4.25);
\filldraw[block, fill=orange!8, draw=orange!70!black, dashed] (1.55,0.32) rectangle (2.15,1.12);
\filldraw[block, fill=orange!8, draw=orange!70!black, dashed] (1.05,1.38) rectangle (2.15,2.28);
\filldraw[block, fill=orange!8, draw=orange!70!black, dashed] (2.55,1.38) rectangle (3.15,2.28);
\filldraw[block, fill=orange!5, draw=orange!45, dashed] (1.05,2.48) rectangle (2.95,3.02);
\filldraw[block, fill=blue!7, draw=blue!65, dashed] (-0.25,3.7) rectangle (2.35,5.3);

\node[old, label=left:{root}] (root) at (0,0) {};
\node[old] (a1) at (0,1.05) {};
\node[old] (a2) at (0,2.1) {};
\node[old, label=left:\(\mathrm{par}(v)\)] (p) at (0,3.15) {};
\node[chain, label=left:\(v\)] (v) at (0,3.85) {};
\node[chain] (vc) at (0,5.02) {};
\draw[edge] (root)--(a1)--(a2)--(p)--(v)--(vc);

\node[old] (o1) at (-1.75,1.05) {};
\node[old] (o1a) at (-2.25,1.75) {};
\node[old] (o1b) at (-1.35,1.75) {};
\node[old] (o2) at (-1.75,2.1) {};
\node[old] (o2a) at (-2.25,2.78) {};
\node[old] (o2b) at (-1.35,2.78) {};
\node[old] (o3) at (-2.75,0.85) {};
\node[old] (op) at (-1.0,3.85) {};
\draw[edge] (root)--(o1);
\draw[edge] (o1)--(o1a);
\draw[edge] (o1)--(o1b);
\draw[edge] (a1)--(o2);
\draw[edge] (o2)--(o2a);
\draw[edge] (o2)--(o2b);
\draw[edge] (root)--(o3);
\draw[edge] (p)--(op);
\node[anchor=east, black!68] at (-3.28,2.0) {\(V_{\mathrm{exp}}^{v,W}\)};

\node[side] (yr1) at (1.85,0.78) {};
\draw[edge] (root)--(yr1);
\node[anchor=west, orange!75!black] at (3.38,0.82) {\(\mathcal D_u^W,\ u\in\mathcal Y_1^{v,W}\)};

\node[side] (y1) at (1.55,1.62) {};
\node[side] (y1a) at (1.2,2.15) {};
\node[side] (y1b) at (1.9,2.15) {};
\node[side] (y1c) at (2.85,1.62) {};
\node[side] (y1ca) at (2.85,2.15) {};
\draw[edge] (a1)--(y1);
\draw[edge] (y1)--(y1a);
\draw[edge] (y1)--(y1b);
\draw[edge] (a1)--(y1c);
\draw[edge] (y1c)--(y1ca);
\node[anchor=west, orange!75!black] at (3.38,1.78) {\(\mathcal D_u^W,\ u\in\mathcal Y_2^{v,W}\)};

\node[anchor=west, orange!65!black] at (3.08,2.65) {\(\mathcal Y_3^{v,W}=\emptyset\)};

\node[chain] (ys1) at (1.1,3.85) {};
\node[chain] (ys2) at (1.9,3.85) {};
\node[chain] (ys1a) at (1.1,5.02) {};
\draw[edge] (p)--(ys1);
\draw[edge] (p)--(ys2);
\draw[edge] (ys1)--(ys1a);

\def\crad{0.24}
\def\gcrad{0.30}
\draw[contourgray]
  ($(root)+(-\crad,-\crad)$)
  -- ($(o3)+(-\crad,-\crad)$) -- ($(o3)+(-\crad,\crad)$) -- ($(o3)+(\crad,\crad)$) -- ($(o3)+(\crad,-\crad)$)
  -- ($(root)+(-\crad,\crad)$)
  -- ($(o1)+(-\crad,-\crad)$) -- ($(o1)+(-\crad,\crad)$)
  -- ($(o1a)+(-\crad,-\crad)$) -- ($(o1a)+(-\crad,\crad)$) -- ($(o1a)+(\crad,\crad)$) -- ($(o1a)+(\crad,-\crad)$)
  -- ($(o1)+(-0.02,\crad)$)
  -- ($(o1b)+(-\crad,-\crad)$) -- ($(o1b)+(-\crad,\crad)$) -- ($(o1b)+(\crad,\crad)$) -- ($(o1b)+(\crad,-\crad)$)
  -- ($(o1)+(\crad,\crad)$) -- ($(o1)+(\crad,-\crad)$)
  -- ($(root)+(0,\crad)$) -- ($(root)+(\crad,\crad)$)
  -- ($(a1)+(-\crad,-\crad)$) -- ($(a1)+(-\crad,\crad)$)
  -- ($(o2)+(-\crad,-\crad)$) -- ($(o2)+(-\crad,\crad)$)
  -- ($(o2a)+(-\crad,-\crad)$) -- ($(o2a)+(-\crad,\crad)$) -- ($(o2a)+(\crad,\crad)$) -- ($(o2a)+(\crad,-\crad)$)
  -- ($(o2)+(-0.02,\crad)$)
  -- ($(o2b)+(-\crad,-\crad)$) -- ($(o2b)+(-\crad,\crad)$) -- ($(o2b)+(\crad,\crad)$) -- ($(o2b)+(\crad,-\crad)$)
  -- ($(o2)+(\crad,\crad)$) -- ($(o2)+(\crad,-\crad)$)
  -- ($(a1)+(\crad,\crad)$)
  -- ($(a2)+(-\crad,-\crad)$) -- ($(a2)+(-\crad,\crad)$)
  -- ($(p)+(-\crad,-\crad)$) -- ($(p)+(-\crad,\crad)$)
  -- ($(op)+(-\crad,-\crad)$) -- ($(op)+(-\crad,\crad)$) -- ($(op)+(\crad,\crad)$) -- ($(op)+(\crad,-\crad)$)
  -- ($(p)+(\crad,\crad)$)
  -- ($(p)+(\crad,-\crad)$)
  -- ($(a2)+(\crad,\crad)$) -- ($(a2)+(\crad,-\crad)$)
  -- ($(a1)+(\crad,\crad)$)
  -- ($(y1)+(-\gcrad,-\gcrad)$) -- ($(y1)+(-\gcrad,\gcrad)$) -- ($(y1)+(\gcrad,\gcrad)$) -- ($(y1)+(\gcrad,-\gcrad)$)
  -- ($(a1)+(0.42,0.06)$) -- ($(a1)+(0.64,-0.20)$)
  -- ($(y1c)+(-\gcrad,-\gcrad)$) -- ($(y1c)+(-\gcrad,\gcrad)$) -- ($(y1c)+(\gcrad,\gcrad)$) -- ($(y1c)+(\gcrad,-\gcrad)$)
  -- ($(a1)+(0.46,-0.32)$) -- ($(a1)+(\crad,-\crad)$)
  -- ($(root)+(\crad,\crad)$)
  -- ($(yr1)+(-\gcrad,-\gcrad)$) -- ($(yr1)+(-\gcrad,\gcrad)$) -- ($(yr1)+(\gcrad,\gcrad)$) -- ($(yr1)+(\gcrad,-\gcrad)$)
  -- ($(root)+(\crad,-\crad)$)
  -- ($(root)+(-\crad,-\crad)$);
\draw[contourorange]
  ($(yr1)+(-\crad,-\crad)$) -- ($(yr1)+(-\crad,\crad)$) -- ($(yr1)+(\crad,\crad)$) -- ($(yr1)+(\crad,-\crad)$) -- ($(yr1)+(-\crad,-\crad)$);
\draw[contourorange]
  ($(y1)+(-\crad,-\crad)$) -- ($(y1)+(-\crad,\crad)$)
  -- ($(y1a)+(-\crad,-\crad)$) -- ($(y1a)+(-\crad,\crad)$) -- ($(y1a)+(\crad,\crad)$) -- ($(y1a)+(\crad,-\crad)$)
  -- ($(y1)+(-0.02,\crad)$)
  -- ($(y1b)+(-\crad,-\crad)$) -- ($(y1b)+(-\crad,\crad)$) -- ($(y1b)+(\crad,\crad)$) -- ($(y1b)+(\crad,-\crad)$)
  -- ($(y1)+(\crad,\crad)$) -- ($(y1)+(\crad,-\crad)$) -- ($(y1)+(-\crad,-\crad)$);
\draw[contourorange]
  ($(y1c)+(-\crad,-\crad)$) -- ($(y1c)+(-\crad,\crad)$)
  -- ($(y1ca)+(-\crad,-\crad)$) -- ($(y1ca)+(-\crad,\crad)$) -- ($(y1ca)+(\crad,\crad)$) -- ($(y1ca)+(\crad,-\crad)$)
  -- ($(y1c)+(\crad,\crad)$) -- ($(y1c)+(\crad,-\crad)$) -- ($(y1c)+(-\crad,-\crad)$);
\draw[contourblue]
  ($(p)+(\crad,\crad)$)
  -- ($(v)+(-\crad,-\crad)$) -- ($(v)+(-\crad,\crad)$)
  -- ($(vc)+(-\crad,-\crad)$) -- ($(vc)+(-\crad,\crad)$) -- ($(vc)+(\crad,\crad)$) -- ($(vc)+(\crad,-\crad)$)
  -- ($(v)+(\crad,\crad)$) -- ($(v)+(\crad,-\crad)$)
  -- ($(p)+(0.30,\crad)$)
  -- ($(ys1)+(-\crad,-\crad)$) -- ($(ys1)+(-\crad,\crad)$)
  -- ($(ys1a)+(-\crad,-\crad)$) -- ($(ys1a)+(-\crad,\crad)$) -- ($(ys1a)+(\crad,\crad)$) -- ($(ys1a)+(\crad,-\crad)$)
  -- ($(ys1)+(\crad,\crad)$) -- ($(ys1)+(\crad,-\crad)$)
  -- ($(p)+(0.40,\crad)$)
  -- ($(ys2)+(-\crad,-\crad)$) -- ($(ys2)+(-\crad,\crad)$) -- ($(ys2)+(\crad,\crad)$) -- ($(ys2)+(\crad,-\crad)$)
  -- ($(p)+(0.55,\crad)$);

\node[anchor=east, blue!70] at (-0.35,4.85) {\(\widehat{\mathcal D}_{\mathrm{par}}^{v,W}\)};
\end{tikzpicture}}
\par\smallskip\textup{(b) Induced path decomposition}
\end{minipage}
\caption{Schematic decomposition relative to a boundary address \(v\). Panel \textup{(a)} shows the decomposition of \(V(W)\) into \(V_{\mathrm{exp}}^{v,W}\), \(\widehat{\mathcal D}_{\mathrm{par}}^{v,W}\), and the descendant blocks \(\mathcal D_u^W\). Panel \textup{(b)} shows, on the same tree, the gray, blue, and orange traversal intervals that induce the corresponding path decomposition. The schematic figure depicts the case where all loops in \(E\) have length \(2\).}
\label{fig:V-decomposition-relative-v}
\end{figure}

\begin{definition}[Truncated exploration data]\label{def:truncated-exploration-data}
Let \(v\in V^{\partial}(W)\).
Write
\[
W(\mathrm{par}(v))=(a_1,\ldots,a_r,0,0,\ldots).
\]
Set
\[
\bigl(W(\mathrm{par}(v))\bigr)_{<v}:=(a_1,\ldots,a_{\operatorname{pre}_W(v)}),
\]
with the convention that \(\bigl(W(\mathrm{par}(v))\bigr)_{<v}\) is the empty tuple when
\(\operatorname{pre}_W(v)=0\).
We further define
\[
\operatorname{Res}_v W
:=
\Bigl(W\big|_{V_{\mathrm{exp}}^{v,W}\setminus\{\mathrm{par}(v)\}},\,
\bigl(W(\mathrm{par}(v))\bigr)_{<v}\Bigr).
\]
The truncated exploration data at \(v\) are
\[
\mathfrak H^{v,W}
:=
\bigl(v,V_{\mathrm{exp}}^{v,W},\operatorname{Res}_v W\bigr).
\]
This data records the previously explored part, but at the current parent
\(\mathrm{par}(v)\) it records only the prefix completed before the
boundary \(v\). In particular, it does not reveal
\(\operatorname{Next}_W(v)\), the type of the bottom-most loop in the next elementary block at \(\mathrm{par}(v)\).

\end{definition}

\begin{remark}
The sets $\mathcal Y_q^{v,W}$ are determined by the restriction
\(\operatorname{Res}_v W\).
\end{remark}

\subsubsection*{Induced path decomposition}
\phantomsection\label{subsubsec:induced-path-decomposition}

The decomposition of \(V(W)\) above has a direct interpretation along the \(E\)-depth-first traversal. 
Let \(\eta\in\mathsf{Seg}(E)\) with \(\mathcal E(\eta)=W\), and let \(f=f_{\mathsf T^*(\eta)}\) be the \(E\)-depth-first traversal of \(\mathsf T^*(\eta)\). 
Recall from Remark~\ref{rem:link} that the steps of \(\eta\) are in one-to-one correspondence with the steps of \(f\). 
Using the notation in Definition~\ref{def:V-decomposition-relative-v}, write
\[
W(\mathrm{par}(v))=(a_1,\ldots,a_r,0,0,\ldots),
\qquad
v=\mathrm{par}(v)j,
\qquad
j=1+\sum_{i=1}^{\operatorname{pre}_W(v)} \lambda_{a_i}.
\]
Let
\[
\tau_{\mathrm{exp}}^{v,W}
:=
\inf\Bigl\{t\ge 0:\ \#\{0\le r\le t:\ f(r)=\mathrm{par}(v)\}=\operatorname{pre}_W(v)+1\Bigr\}
\]
be the time of the \((\operatorname{pre}_W(v)+1)\)-st visit to \(\mathrm{par}(v)\), and, for every \(u\in V(W)\), set
\[
\tau_u^-:=\inf\{m\ge 0:\ f(m)=u\},
\qquad
\tau_u^+:=\sup\{m\ge 0:\ f(m)=u\}.
\]
We also write
\(
\tau_{\mathrm{par}}^{v,W}:=\tau_{\mathrm{par}(v)}^+
\)
for the last visit to \(\mathrm{par}(v)\). 

The boundary address \(v=\mathrm{par}(v)j\) and the vertices \(u\in\bigcup_{q=1}^{\ell-1}\mathcal Y_q^{v,W}\) locate the relevant times on the path through the traversal \(f\). 
The address \(v\) determines the interval \([\tau_{\mathrm{exp}}^{v,W},\tau_{\mathrm{par}}^{v,W}]\), while each \(u\in\bigcup_{q=1}^{\ell-1}\mathcal Y_q^{v,W}\) determines the interval \([\tau_u^-,\tau_u^+]\). 
In terms of the decomposition of \(V(W)\), \(f[\tau_{\mathrm{exp}}^{v,W},\tau_{\mathrm{par}}^{v,W}]\) explores \(\mathcal D_{\mathrm{par}}^{v,W}\), and \(f[\tau_u^-,\tau_u^+]\) explores \(\mathcal D_u^W\). The restriction of \(f\) to the remaining portions of its time interval gives the pruned contour corresponding to \(V_{\mathrm{exp}}^{v,W}\).
This correspondence is illustrated by the blue, orange, and gray contour segments in Figure~\ref{fig:V-decomposition-relative-v}.
Using these same intervals, we decompose the path \(\eta\) by
\[
I_{\mathrm{exp}}^{v,W}
:=
\{0,\ldots,|\eta|\}\setminus
\Bigg(
(\tau_{\mathrm{exp}}^{v,W},\tau_{\mathrm{par}}^{v,W}]
\cup
\bigcup_{q=1}^{\ell-1}\bigcup_{u\in\mathcal Y_q^{v,W}}
(\tau_u^-,\tau_u^+]
\Bigg).
\]
Writing \(I_{\mathrm{exp}}^{v,W}=\{r_0<r_1<\cdots<r_L\}\), set
\[
\begin{aligned}
\eta_{\mathrm{exp}}^{v,W}
&:=\eta\big|_{I_{\mathrm{exp}}^{v,W}}
:=(\eta_{r_0},\eta_{r_1},\ldots,\eta_{r_L}),\\
\eta_{\mathrm{par}}^{v,W}
&:=\eta[\tau_{\mathrm{exp}}^{v,W},\tau_{\mathrm{par}}^{v,W}]-\eta_{\tau_{\mathrm{exp}}^{v,W}},
\qquad
\eta_u^W:=\eta[\tau_u^-,\tau_u^+]-\eta_{\tau_u^-}.
\end{aligned}
\]
These objects give the path-level decomposition induced by the decomposition of \(V(W)\) relative to \(v\). 

Intuitively, \(\eta_{\mathrm{exp}}^{v,W}\) first follows the initial segment \(\eta[0,\tau_{\mathrm{exp}}^{v,W}]\), and then follows the remaining part of \(\eta\) after the loop families
\(\eta_{\mathrm{par}}^{v,W}\) and \(\eta_u^W\),
\(u\in\bigcup_{q=1}^{\ell-1}\mathcal Y_q^{v,W}\),
have been pruned out.
The segment \(\eta_{\mathrm{par}}^{v,W}\) records the loop cluster attached at the site \(\eta_{\tau_{\mathrm{exp}}^{v,W}}\). 
For \(u\in\bigcup_{q=1}^{\ell-1}\mathcal Y_q^{v,W}\), the segment \(\eta_u^W\) records the loop cluster attached at the site \(\eta_{\tau_u^-}\). 
Thus \(\eta_{\mathrm{exp}}^{v,W}\) serves as the backbone path: the segments \(\eta_{\mathrm{par}}^{v,W}\) and \(\eta_u^W\) are inserted at the retained sites \(\eta_{\tau_{\mathrm{exp}}^{v,W}}\) and \(\eta_{\tau_u^-}\), respectively.

\subsubsection*{Pruned prefixes and compatibility condition}

We now record the compatibility conditions satisfied by \(\eta_{\mathrm{par}}^{v,W}\) and \(\eta_u^W\).
For \(u\in V(W)\), define the pruned prefix before \(u\) by
\[
\pi_u^{v,\mathrm{exp}}:=\mathsf{Prune}(\eta[0,\tau_u^-],E).
\]
The notation emphasizes that these prefixes are determined by the truncated exploration data \(\mathfrak H^{v,W}\): when pruning the paths above, \(\eta_{\mathrm{par}}^{v,W}\) and any previously encountered segments \(\eta_w^W\) are pruned in the corresponding \(\mathsf{Prune}(\cdot,E)\) operations. 
Consequently, the resulting pruned prefix depends only on \(\operatorname{Res}_v W\).

Observe that \(\eta_u^W\) is the final pruned segment in the loop-pruned decomposition of \(\eta[0,\tau_u^+]\), and hence satisfies the compatibility condition \eqref{eq:compatible-retained-prefix} with retained prefix \(\pi_u^{v,\mathrm{exp}}\):
\begin{equation}\label{eq:compatibility-descendant-piece}
\eta_u^W\in \mathsf{Seg}(\pi_u^{v,\mathrm{exp}},E).
\end{equation}
For the parent segment, \(\eta_{\mathrm{par}}^{v,W}\) is the part of the final pruned segment of \(\eta[0,\tau_{\mathrm{par}}^{v,W}]\) corresponding to \(\mathcal D_{\mathrm{par}}^{v,W}\). 
The part of this final segment that comes before \(\mathcal D_{\mathrm{par}}^{v,W}\) is always pruned in \(\mathsf{Prune}(\eta[0,\tau_{\mathrm{par}}^{v,W}],E)\). 
Therefore it is not hard to see that
\begin{equation}\label{eq:compatibility-parent-piece}
\eta_{\mathrm{par}}^{v,W}\in \mathsf{Seg}(\pi_{\mathrm{par}(v)}^{v,\mathrm{exp}},E).
\end{equation}

The conditions \eqref{eq:compatibility-descendant-piece} and \eqref{eq:compatibility-parent-piece} were stated for \(W\in\mathcal E(\mathsf{Seg}(E))\). 
They extend directly to the case \(W\in\mathcal E(\mathsf{Seg}(s',E))\), where \(s'\) is a fixed \(E\)-pruned finite path. 
In this case the retained prefix is shifted from \(\pi\) to \(s'\circ_t\pi\), and the corresponding conditions become
\begin{equation}\label{eq:compatibility-pieces-prefix}
\eta_u^W\in \mathsf{Seg}(s'\circ_t\pi_u^{v,\mathrm{exp}},E)\quad
\forall\,u\in\bigcup_{q=1}^{|v|-1}\mathcal Y_q^{v,W},
\qquad
\eta_{\mathrm{par}}^{v,W}\in \mathsf{Seg}(s'\circ_t\pi_{\mathrm{par}(v)}^{v,\mathrm{exp}},E).
\end{equation}

\subsubsection*{Fiber classes and factorization theorem}

We now make precise, in the same spirit as Proposition~\ref{pro:one-to-one}, that the compatibility conditions \eqref{eq:compatibility-pieces-prefix} are not merely necessary: after the truncated exploration data are fixed, they give the precise admissible classes for \(\eta_{\mathrm{par}}^{v,W}\) and the \(\eta_u^W\)'s.

For an \(E\)-pruned finite path \(s'\), define the set of possible truncated exploration data by
\[
\mathcal{H}_{\mathrm{tr}}^{s'}
:=
\bigl\{(v,A,\omega^\partial):\ \exists\,\eta\in\mathsf{Seg}(s',E),\
W=\mathcal E(\eta),\ v\in V^{\partial}(W),\
A=V_{\mathrm{exp}}^{v,W},\
\omega^\partial=\operatorname{Res}_v W\bigr\}.
\]
For \(h=(v,A,\omega^\partial)\in\mathcal{H}_{\mathrm{tr}}^{s'}\), define the path fiber
\[
\mathsf{Path}^{s'}(h)
:=
\{\eta\in\mathsf{Seg}(s',E):\ W=\mathcal E(\eta),\ v\in V^{\partial}(W),\
V_{\mathrm{exp}}^{v,W}=A,\ \operatorname{Res}_v W=\omega^\partial\}.
\]
If \(h=(v,A,\omega^\partial)\) and the second component of \(\omega^\partial\) is
\((a_1,\ldots,a_q)\), then \(\operatorname{pre}_{W}(v)=q\) for every
\(W=\mathcal E(\eta)\) with \(\eta\in\mathsf{Path}^{s'}(h)\). We write this common value as
\(\operatorname{pre}(h)\).

\begin{theorem}[Fiber factorization]
\label{thm:fiber-factorization-fixed-H}
Fix an \(E\)-pruned finite path \(s'\) and \(h=(v,A,\omega^\partial)\in\mathcal{H}_{\mathrm{tr}}^{s'}\).
The sets \(\mathcal Y_q^{v,W}\) and the pruned prefixes \(\pi_u^{v,\mathrm{exp}}\), for \(u\in\{\mathrm{par}(v)\}\cup\bigcup_{q=1}^{|v|-1}\mathcal Y_q^{v,W}\), are constant over \(W=\mathcal E(\eta)\) with \(\eta\in\mathsf{Path}^{s'}(h)\). 
Denote these common values by \(\mathcal Y_q^{h}\) and \(\pi_u^{h}\).
Then the map
\[
\begin{aligned}
\Psi_h:\mathsf{Path}^{s'}(h)
&\longrightarrow
\mathsf{Seg}(s'\circ_t \pi_{\mathrm{par}(v)}^{h},E)
\times\prod_{q=1}^{|v|-1}\prod_{u\in\mathcal Y_q^{h}}
\mathsf{Seg}(s'\circ_t \pi_u^{h},E),\\
\eta&\longmapsto
\Bigl(\eta_{\mathrm{par}}^{v,W},\
\bigl(\eta_u^W\bigr)_{1\le q\le |v|-1,\ u\in\mathcal Y_q^{h}}\Bigr),
\qquad W=\mathcal E(\eta),
\end{aligned}
\]
is a bijection.

\end{theorem}
This is a direct variant of Proposition~\ref{pro:one-to-one}; the proof is straightforward and omitted.

We next explain how to reconstruct \(\eta\) from the pieces \(\eta_{\mathrm{par}}^{v,W}\) and \(\eta_u^W\) in the bijection above.
Note that for \(\eta\in\mathsf{Path}^{s'}(h)\) with \(W=\mathcal E(\eta)\), the path \(\eta_{\mathrm{exp}}^{v,W}\) and the cut time \(\tau_{\mathrm{exp}}^{v,W}\) are determined by \(h\).
We denote the common values by
\[
\eta_{\mathrm{exp}}^h:=\eta_{\mathrm{exp}}^{v,W},
\qquad
\tau_{\mathrm{exp}}^h:=\tau_{\mathrm{exp}}^{v,W}.
\]

The inverse operation in the bijection of Theorem~\ref{thm:fiber-factorization-fixed-H} can then be described as follows.
\begin{proposition}[Fiber reconstruction]\label{prop:path-fiber-reconstruction}
Let \(\eta\in\mathsf{Path}^{s'}(h)\) and \(W=\mathcal E(\eta)\).
Then \(\eta\) is recovered from \(\eta_{\mathrm{par}}^{v,W}\) and the collection \(\{\eta_u^W:u\in\bigcup_{q=1}^{|v|-1}\mathcal Y_q^h\}\) by inserting them into \(\eta_{\mathrm{exp}}^h\): insert \(\eta_{\mathrm{par}}^{v,W}\) at time \(\tau_{\mathrm{exp}}^h\), and insert each \(\eta_u^W\) at time \(\tau_u^-(\eta_{\mathrm{exp}}^h)\).
Here \(\tau_u^-(\eta_{\mathrm{exp}}^h)\) means the time \(\tau_u^-\) from the definition above, applied with the path \(\eta\) replaced by \(\eta_{\mathrm{exp}}^h\) (it is easy to see that \(u\in V(\mathcal E(\eta_{\mathrm{exp}}^h))\), and hence \(\tau_u^-(\eta_{\mathrm{exp}}^h)\) is well-defined).

\end{proposition}

\section{Loop-pruning of simple random walk}\label{sec:loop-pruning-srw}

We now consider the simple random walk \( S = (S_n : n \ge 0) \) on \( \mathbb{Z}^d \), \(d\ge 3\). For any finite path $\eta$ starting from $0$, we define the path probability:
\[p(\eta)=\mathbb{P}(S[0,|\eta|]=\eta).\]

Throughout this section, we fix a finite set $E$ of loops such that
\[p(e)>0\quad\forall\,e\in E.\]
Write
\[
S' := \mathsf{Prune}\big(S, E\big),
\quad
N_n := N_n\big(S, E\big).
\]

\subsection{Retained portion}
In this subsection, we investigate the asymptotic behavior of $N_n$. Our main result is a moderate deviation bound stated below.

\begin{theorem}\label{thm:moderate}
	There exist constants \(\kappa=\kappa(E)\in(0,1)\) and \(c=c(E)>0\) such that
for every \(a\in(1/2,1)\) and all sufficiently large \(n\),
	\begin{align}\label{eq:moderate}
		\mathbb{P}\!\left(|N_n-\kappa n|>n^a\right)<\exp(-c n^{2a-1}).
	\end{align}
\end{theorem}
\begin{remark}
	As an immediate consequence, with the same constant \(\kappa\) as in Theorem~\ref{thm:moderate},
	\[
	\lim_{n\to\infty}\frac{N_n}{n}=\kappa \qquad \text{a.s..}
	\]
\end{remark}

To prove Theorem~\ref{thm:moderate}, we extend $S$ to a two-sided random walk $S(-\infty,\infty)$ by setting
\[S_{-n}:=-\widetilde{S}_n,\quad\text{where $\widetilde{S}[0,\infty)$ is an i.i.d. copy of $S[0,\infty)$.}\]
Recall the $E$-cut steps defined in Definition~\ref{def:E_cut_steps}. For an integer interval \(I\subset\mathbb{Z}\), we define
\[
\mathrm{Cut}(I)
:=\Big\{\, i\in I:\ (S_{i-1},S_i)\ \text{is an $E$-cut step}\,\Big\},
\qquad |I|:=\max I-\min I.
\]
The proofs rely crucially on the following estimate, which ensures that \(E\)-cut points occur with positive density and have exponentially small lower-tail deviations.

\begin{theorem}\label{thm:cut}
There exist constants \(\theta=\theta(E)\in(0,1)\) and \(c=c(E)>0\) such that for every integer interval \(I\subset\mathbb{Z}\) whose length \(|I|\) is sufficiently large,
\begin{align}\label{eq:cut}
	\mathbb{P}\!\left(\#\,\mathrm{Cut}(I)>\theta |I|\right)\ge 1-\exp(-c|I|).
\end{align} 
\end{theorem}
\begin{remark}
The theorem implies in particular that the \(E\)-cut steps of \(S(-\infty,\infty)\) are unbounded both to the left and to the right.
This ensures that the two-sided pruning \(\mathsf{Prune}(S,E)\) is well-defined in the sense of Definition~\ref{def:erase_two_sided}.
\end{remark}

The proof of Theorem~\ref{thm:cut} is deferred to Section~\ref{sec:cut}. For the moment, we assume it and use it to derive Theorem~\ref{thm:moderate}. The key step is that \eqref{eq:cut} yields an exponential mixing property for the indicator sequence recording whether each step is retained; see Proposition~\ref{prop:mixing} below. An application of the Bernstein inequality for exponentially mixing sequences \cite[Theorem 1]{merlevede2009bernstein} then implies \eqref{eq:moderate}.

Consider the indicator sequence
\[
Y_n := \mathbf{1}\Big\{\text{the step $(S_{n-1},S_n)$ is retained in $\mathsf{Prune}(S,E)$}\Big\}, 
\quad n \in \mathbb{Z}.
\]
For integers $i \le j$, let
\[
\mathcal{F}_i^j := \sigma(Y_i, \ldots, Y_j).
\]

The strong mixing coefficients of the stationary process \(Y_n\) are defined by
\[
\alpha(r) := \sup_{A \in \mathcal{F}_{-\infty}^0,\, B \in \mathcal{F}_{r}^\infty} 
\big|\mathbb{P}(A \cap B) - \mathbb{P}(A)\mathbb{P}(B)\big|, 
\quad r \in \mathbb{N}^+.
\]

\begin{proposition}\label{prop:mixing}
The process \(Y_n\) is exponentially mixing. That is, there exists $c=c(E)>0$ such that for all sufficiently large $r \in \mathbb{N}^+$,
\[
\alpha(r) < \exp(-c r).
\]
\end{proposition}
\begin{proof}
	It is enough to show that there exists \(c>0\) such that for all sufficiently large \(r\),
	\begin{align}\label{eq:fg_exp}
		\sup_{f,g}\Big|\mathbb{E}\big[f(Y(-\infty,0])g(Y[r,\infty))\big]-\mathbb{E}\big[f(Y(-\infty,0])\big]\mathbb{E}\big[g(Y[r,\infty))\big]\Big|<\exp(-c r),
	\end{align}
	where the supremum is over all measurable functions \(f,g\ge 0\) with $\norm{f}_\infty,\norm{g}_\infty\le 1$.

	To this end, we consider the event:
	\[A_r:=\big\{\#\text{Cut}([0,r/2])\ge 1,\, \#\text{Cut}([r/2,r])\ge 1\big\}.\]
	It follows from \eqref{eq:cut} that there exists $c_1=c_1(E)>0$ such that for all $r\ge 1$,
	\begin{align}\label{eq:prob_Ar}
		\mathbb{P}(A_r)\ge 1-\exp(-c_1r).
	\end{align}

	On the other hand, with
	\begin{align*}
		\widetilde{Y}^{r,-}_n&:=\mathbf{1}\Big\{\text{the step $(S_{n-1},S_n)$ is retained in $\mathsf{Prune}(S(-\infty,r/2],E)$}\Big\}\footnotemark,\quad \text{for }n\le 0,\\
		\widetilde{Y}^{r,+}_n&:=\mathbf{1}\Big\{\text{the step $(S_{n-1},S_n)$ is retained in $\mathsf{Prune}(S[r/2,\infty),E)$}\Big\},\quad \text{for }n\ge r,
	\end{align*}
	\footnotetext{For \(n\le r/2\), the step \((S_{n-1},S_n)\) is said to be retained in \(\mathsf{Prune}(S(-\infty,r/2],E)\) if and only if, for every \(m\le n-1\), the same step is retained in \(\mathsf{Prune}(S[m,r/2],E)\).}
	we have on the event $A_r$,
	\begin{align}\label{eq:tilde_Y}
		\widetilde{Y}^{r,-}_n=Y_n\quad\text{for }n\le 0,\qquad \widetilde{Y}^{r,+}_n=Y_n\quad\text{for }n\ge r.
	\end{align}

	Combining \eqref{eq:prob_Ar} and \eqref{eq:tilde_Y}, we obtain
	\begin{align}\label{eq:fg_Ar}
	\begin{aligned}
		&\mathbb{E}\big[f(Y(-\infty,0])g(Y[r,\infty))\big]\\
		=\,&\mathbb{E}\big[f\big(\widetilde{Y}^{r,-}_n(-\infty,0]\big)g\big(\widetilde{Y}^{r,+}_n[r,\infty)\big);A_r\big]+\mathbb{E}\big[f(Y(-\infty,0])g(Y[r,\infty));(A_r)^{\rm c}\big]\\
		=\,&\mathbb{E}\big[f\big(\widetilde{Y}^{r,-}_n(-\infty,0]\big)g\big(\widetilde{Y}^{r,+}_n[r,\infty)\big)\big]+a_r\\
		\overset{(*)}{=}\,&\mathbb{E}\big[f\big(\widetilde{Y}^{r,-}_n(-\infty,0]\big)\big]\mathbb{E}\big[g\big(\widetilde{Y}^{r,+}_n[r,\infty)\big)\big]+a_r,
	\end{aligned}
	\end{align}
	where 
	\[
	\begin{aligned}
	a_r&:=\mathbb{E}\Big[f(Y(-\infty,0])g(Y[r,\infty))
	-f\big(\widetilde{Y}^{r,-}_n(-\infty,0]\big)g\big(\widetilde{Y}^{r,+}_n[r,\infty)\big);(A_r)^{\rm c}\Big],\\
	|a_r|&\overset{\eqref{eq:prob_Ar}}{\le} \exp(-c_1 r).
	\end{aligned}
	\]
	In $(*)$, we used that the prunings of \(S(-\infty,r/2]\) and \(S[r/2,\infty)\) depend on disjoint increments of the walk and are therefore independent.

	Taking $g\equiv 1$ or $f\equiv 1$ in \eqref{eq:fg_Ar} gives
	\begin{align*}
		\abs{\mathbb{E}\big[f\big(Y(-\infty,0]\big)\big]-\mathbb{E}\big[f\big(\widetilde{Y}^{r,-}_n(-\infty,0]\big)\big]}\le \exp(-c_1 r),\\
	\abs{\mathbb{E}\big[g\big(Y[r,\infty)\big)\big]-\mathbb{E}\big[g\big(\widetilde{Y}^{r,+}_n[r,\infty)\big)\big]}\le \exp(-c_1 r).
	\end{align*}
	Substituting these into the previous estimate \eqref{eq:fg_Ar} yields \eqref{eq:fg_exp}.
	\end{proof}

\begin{proof}[Proof of Theorem~\ref{thm:moderate} assuming Theorem \ref{thm:cut}]
Proposition~\ref{prop:mixing} verifies the mixing assumptions required in the Bernstein inequality of \cite[Theorem~1]{merlevede2009bernstein}. Hence, by \cite[Theorem~1, inequality (2.1)]{merlevede2009bernstein}, with \(\kappa:=\mathbb{E}[Y_0]\), there exists \(c>0\) such that for every integer interval \(I\subset\mathbb{Z}\) of sufficiently large length \(|I|\) and every \(x\ge 0\),
\begin{equation}\label{eq:bernstein}
\mathbb{P}\!\Big(\Big|\sum_{i\in I} Y_i-\kappa |I|\Big|\ge x\Big)
\le \exp\!\left(-\frac{c x^2}{|I|+x\,\log |I|\,\log\log |I|}\right).
\end{equation}

We now transfer this estimate back to the one-sided pruning. Recall that
\[
N_n=\sum_{i=1}^n Y'_i,
\qquad
Y'_i:=\mathbf{1}\Big\{\text{the step \((S_{i-1},S_i)\) is retained in \(\mathsf{Prune}(S[0,\infty),E)\)}\Big\}.
\]
Since the parameter \(a\) in the theorem satisfies \(a\in(1/2,1)\), we may choose \(b\in(2a-1,a)\). 
Consider the event
\[
E_{n,b}:=\big\{\#\,\mathrm{Cut}([0,n^b])\ge 1\big\}.
\]
By Theorem~\ref{thm:cut}, there exists \(c_1>0\) such that for all sufficiently large \(n\),
\[
\mathbb{P}(E_{n,b})\ge 1-\exp(-c_1 n^b).
\]
On the event \(E_{n,b}\), we have
\[
Y_i=Y'_i\qquad \text{for all } i>n^b.
\]

Therefore,
\begin{align*} &\mathbb{P}\Big(\Big|\sum_{i\in (n^b,n]}Y'_i- (n-n^b)\kappa\Big|\ge n^a/2\Big)\\ \le\,& \exp(-c_1 n^b)+\mathbb{P}\Big(\Big|\sum_{i\in (n^b,n]}Y_i- (n-n^b)\kappa\Big|\ge n^a/2\Big)\\ \overset{\eqref{eq:bernstein}}{\le}\,& \exp(-c_1 n^b)+\exp(-c_2n^{2a-1})\le \exp(-c_3n^{2a-1}), \end{align*}
for some constants \(c_2,c_3>0\) and all sufficiently large \(n\), since \(b>2a-1\).
Since \(b<a\), the initial block \(i\le n^b\) is negligible on the scale \(n^a\). 
This yields \eqref{eq:moderate}.
\end{proof}

\subsection{Local time of pruned path}

For any finite or infinite path $\eta=(\eta_j)$ and any $x\in\mathbb Z^d$, we write
\[
\ell_x(\eta):=\sum_j \mathbf 1{\{\eta_j=x\}}
\]
for the number of visits of $\eta$ to $x$.

The following proposition is the main conclusion of this subsection.
\begin{proposition}\label{prop:Sprime-localtime-tail}
For every \(\beta\in(0,1]\) and every \(q>0\), there exists a finite loop set
\(
E=E(\beta,q)
\)
such that, for every \(N\in\mathbb{N}^+\) and all sufficiently large \(n\in\mathbb{N}^+\),
\begin{align}\label{eq:Sprime-localtime-tail}
\mathbb{P}\bigl(\ell_0(\Prune(S[0,N],E))> \beta\alpha\log n\bigr)\le \frac{1}{n^q}.
\end{align}
\end{proposition}
The proof is based on the following preliminary lemmas.

Set
\(
M_\infty:=\xi(\infty,0).
\)
Since \(M_\infty\) has a geometric distribution with success probability \(\gamma\), we immediately obtain the following lemma.

\begin{lemma}\label{lem:N-geometric-tail}
For every \(\beta'>0\),
\begin{equation}\label{eq:N-geometric-tail}
\mathbb{P}\bigl(M_\infty>\beta'\alpha\log n\bigr)\le n^{-\beta'}.
\end{equation}
\end{lemma}

Let
\[
0=\tau_0<\tau_1<\cdots<\tau_{M_\infty}<\infty
\]
be the successive return times of \(S\) to \(0\), and define
\[
\mathcal R_r:=S[\tau_{r-1},\tau_r],\qquad 1\le r\le M_\infty,
\]
to be the \(r\)-th excursion of \(S\) from \(0\).

For a finite loop set \(E\), define
\begin{equation}\label{eq:def-varthetaE}
\vartheta_E
:=
\mathbb P\bigl(\mathcal R_1\in\mathsf{Seg}(E)\,\big|\,\tau_1<\infty\bigr).
\end{equation}

\begin{lemma}\label{lem:first-return-Seg-approx}
For every \(\varepsilon>0\), there exists a finite loop set \(E\), with \(p(e)>0\) and \(|e|\ge2\) for every \(e\in E\), such that
\[
\vartheta_E>1-\varepsilon.
\]

\end{lemma}
\begin{proof}
Let
\[
\mathcal L_0^{\mathrm{fr}}
:=\bigl\{\eta\in\mathcal L_0:|\eta|\ge1,\ 
\eta_j\ne0\text{ for }1\le j<|\eta|,\ p(\eta)>0\bigr\}
\]
be the set of first-return loops at the origin with positive probability. Since
\[
\mathbb P(\tau_1<\infty)=\sum_{\eta\in\mathcal L_0^{\mathrm{fr}}}p(\eta)=1-\gamma,
\]
we may choose a finite set \(\mathcal F\subset\mathcal L_0^{\mathrm{fr}}\) such that
\[
\sum_{\eta\in\mathcal F}p(\eta)>(1-\varepsilon)(1-\gamma).
\]
For each \(\eta\in\mathcal F\), take its canonical chronological decomposition into simple loops, and let \(E\) be the finite set of all simple loops appearing in these decompositions. Then \(p(e)>0\) and \(|e|\ge2\) for every \(e\in E\), and \(\mathcal F\subset\mathsf{Seg}(E)\). Therefore
\[
\vartheta_E
\ge
\frac{1}{1-\gamma}\sum_{\eta\in\mathcal F}p(\eta)
>1-\varepsilon.\qedhere
\]
\end{proof}

By the definition of loop-pruned decomposition, we easily obtain the following observation.

\begin{lemma}\label{lem:Ri-in-Seg-implies-boundary-pruned}
Fix \(1\le i<j\le k\). Assume that \(k\le M_\infty\) and that the boundary step
\[
(S_{\tau_i-1},S_{\tau_i})
\]
is not pruned in \(\Prune(S[0,\tau_k],E)\), and that
\[
\mathcal R_{i+1},\mathcal R_{i+2},\ldots,\mathcal R_j\in\mathsf{Seg}(E).
\]
Then every subsequent boundary step
\[
(S_{\tau_r-1},S_{\tau_r}),\qquad i+1\le r\le j,
\]
is pruned in \(\Prune(S[0,\tau_k],E)\).
\end{lemma}

For \(1\le i\le M_\infty\), let
\[
X_i:=\mathbf{1}\{\mathcal R_i\in \mathsf{Seg}(E)\}.
\]
Then for any \(k\in\mathbb{N}^+\), conditionally on \(\{k\le M_\infty\}\), the variables \((X_i)_{1\le i\le k}\) are i.i.d. Bernoulli\((\vartheta_E)\).

We record the corresponding large-deviation estimate (see \cite[Section~2.2]{dembo1998large}).

\begin{lemma}\label{lem:bernoulli-LDP-Seg}
For every $k\in\mathbb{N}^+$ and \(x\in(0,\vartheta_E)\),
\begin{equation}\label{eq:bernoulli-LDP-Seg}
\mathbb{P}\!\left(
\frac{1}{k}\sum_{i=1}^{k} X_i \le x
\,\Big|\, k\le M_\infty
\right)
\le \exp\bigl\{-I_{\vartheta_E}(x)\,k\bigr\},
\end{equation}
where
\[
I_{\vartheta_E}(x):=x\log\frac{x}{\vartheta_E}+(1-x)\log\frac{1-x}{1-\vartheta_E}.
\]

\end{lemma}
\begin{remark}\label{rem:rate-function-varthetaE-diverges}
Note that for any $x\in (0,1)$, as \(\vartheta_E\rightarrow 1\),
  \[I_{\vartheta_E}(x)\rightarrow \infty.\]

\end{remark}

\begin{lemma}\label{lem:xiS-etaM-event-inclusion}
For any \(k\in\mathbb{N}^+\) and \(\eta\in(0,1)\), on the event \(\{k\le M_\infty\}\),
\begin{equation}\label{eq:xiS-etaM-event-inclusion}
\Biggl\{
\ell_0\bigl(\Prune(S[0,\tau_k],E)\bigr)\le \lceil \eta k\rceil
\Biggr\}
\supset
\Biggl\{
\frac{1}{k}\sum_{i=1}^{k}X_i>1-\eta
\Biggr\}.
\end{equation}
\end{lemma}

\begin{proof}
The event on the right-hand side is equivalent to
\[
\#\{1\le i\le k:\mathcal R_i\notin \mathsf{Seg}(E)\}<\eta k.
\]
Suppose this set is
\[
\{i_1<i_2<\cdots<i_m\},
\qquad m<\eta k,
\]
and set \(i_{m+1}:=k+1\). By Lemma~\ref{lem:Ri-in-Seg-implies-boundary-pruned}, for each \(1\le j\le m\), there is at most one index
\(
q\in [i_j,i_{j+1})
\)
such that the boundary step \((S_{\tau_q-1},S_{\tau_q})\) is retained in \(\Prune(S[0,\tau_k],E)\). In addition, the initial segment \(S[0,\tau_{i_1-1}]\) is entirely removed by pruning. Hence there are at most \(m\le \lceil\eta k\rceil-1\) retained steps of the form \((S_{\tau_q-1},S_{\tau_q})\).

Consequently, \(\Prune(S[0,\tau_k],E)\) returns to \(0\) at most \(\lceil\eta k\rceil-1\) times after time \(0\). Therefore,
\[
\ell_0\bigl(\Prune(S[0,\tau_k],E)\bigr)\le \lceil\eta k\rceil.
\]
This proves \eqref{eq:xiS-etaM-event-inclusion}.
\end{proof}
Combining Lemmas \ref{lem:bernoulli-LDP-Seg} and \ref{lem:xiS-etaM-event-inclusion}, we have the following:
\begin{corollary}
For every $k\in\mathbb{N}^+$ and \(\eta\in(1-\vartheta_E,1)\),
\begin{equation}\label{eq:bernoulli-LDP-l0}
\mathbb{P}\!\left(
\ell_0\bigl(\Prune(S[0,\tau_k],E)\bigr) > \lceil\eta k\rceil
\,\Big|\, k\le M_\infty
\right)
\le \exp\bigl\{-I_{\vartheta_E}(1-\eta)\,k\bigr\}.
\end{equation}

\end{corollary}
\begin{proof}[Proof of Proposition~\ref{prop:Sprime-localtime-tail}]
Fix \(\beta\in(0,1]\) and \(q>0\). Note that the map
\[
m\mapsto \ell_0(\Prune(S[0,m],E))
\]
can increase only at times \(m=\tau_k\), \(1\le k\le M_\infty\). Hence, for any \(N\in\mathbb{N}^+\),
\begin{align*}
&\bigl\{\ell_0(\Prune(S[0,N],E))>\beta\alpha\log n\bigr\}\\
&\subset
\bigcup_{k\ge \lfloor \beta\alpha\log n\rfloor}
\bigl\{
k\le M_\infty,\,
\ell_0(\Prune(S[0,\tau_k],E))>\beta\alpha\log n
\bigr\}.
\end{align*}
Therefore,
\begin{align}\label{eq:ell0-prune-upper}
\begin{aligned}
&\mathbb{P}\bigl(\ell_0(\Prune(S[0,N],E))>\beta\alpha\log n\bigr)\\
\le\;&
\mathbb{P}\bigl(M_\infty>(q+1)\alpha\log n\bigr)\\
&\quad+
\sum_{k=\lfloor \beta\alpha\log n\rfloor}^{\lfloor (q+1)\alpha\log n\rfloor}
\mathbb{P}\bigl(
\ell_0(\Prune(S[0,\tau_k],E))>\beta\alpha\log n
\,\big|\, k\le M_\infty
\bigr)\\
\overset{\eqref{eq:N-geometric-tail}}{\le}\;&
\frac{1}{n^{q+1}}
+
\sum_{k=\lfloor \beta\alpha\log n\rfloor}^{\lfloor (q+1)\alpha\log n\rfloor}
\mathbb{P}\Bigl(
\ell_0(\Prune(S[0,\tau_k],E))>\frac{\beta}{q+1}k
\,\Big|\, k\le M_\infty
\Bigr).    
\end{aligned}
\end{align}

By Lemma~\ref{lem:first-return-Seg-approx} and Remark~\ref{rem:rate-function-varthetaE-diverges}, we may choose \(E\) so that \(\vartheta_E\) is sufficiently close to \(1\), and hence
\[
1-\vartheta_E<\frac{\beta}{2(q+1)}
\qquad\text{and}\qquad
I_{\vartheta_E}\Bigl(1-\frac{\beta}{q+1}\Bigr)>\frac{2(q+1)}{\beta\alpha}.
\]
By continuity, for all sufficiently large \(n\) and all
\(
\lfloor \beta\alpha\log n\rfloor\le k\le \lfloor (q+1)\alpha\log n\rfloor,
\)
\[
I_{\vartheta_E}\Bigl(1-\frac{\beta}{q+1}+\frac1k\Bigr)>\frac{q+1}{\beta\alpha}.
\]

For such \(k\), set \(\eta_k:=\frac{\beta}{q+1}-\frac1k\). Since \(\lceil \eta_k k\rceil\le \frac{\beta}{q+1}k\), it follows from \eqref{eq:bernoulli-LDP-l0} that
\[
\mathbb{P}\Bigl(
\ell_0(\Prune(S[0,\tau_k],E))>\frac{\beta}{q+1}k
\,\Big|\, k\le M_\infty
\Bigr)\le \exp\{-I_{\vartheta_E}(1-\eta_k)k\}\le \frac{1}{n^{q+1}}.
\]

Substituting this back into \eqref{eq:ell0-prune-upper} yields
\[
\mathbb{P}\bigl(\ell_0(\Prune(S[0,N],E))>\beta\alpha\log n\bigr)
\le
\frac{1}{n^{q+1}}+\frac{(q+1)\alpha\log n}{n^{q+1}}
\le \frac{1}{n^q},
\]
for sufficiently large \(n\). This completes the proof.
\end{proof}

\subsection{Conditional distribution of ES-representations}\label{subsec:conditional-ES}

Recall that $S'=\mathsf{Prune}(S,E)$. Fix \(u \in \mathbb{N}\). By definition of \(N^{-1}\), the time \(N^{-1}(u+1)-1\) is exactly when the first \(u + 1\) pruned segments have just been completely generated. These pruned segments are denoted by 
\[
\Xi_k := \xi_k\big(S[0,N^{-1}(u+1)-1],E\big),\qquad 0\le k\le u,
\]
so that
\[
S[0, N^{-1}(u+1)-1] = S'[0, u] \oplus (\Xi_k)_{k=0}^{u}.
\]
We further set
\begin{equation}\label{eq:def-Wu}
\mathcal{W}=(\mathcal{W}_0,\ldots,\mathcal{W}_{u})
:=\mathcal{E}^\otimes\big((\Xi_k)_{k=0}^{u}\big)
=\bigl(\mathcal{E}(\Xi_0),\dots,\mathcal{E}(\Xi_{u})\bigr)
\in \bigl(\mathbf{K}(|E|)^{\mathcal U}\bigr)^{u+1}.
\end{equation}

The main result of this subsection is the following theorem, which quantifies the probability loss caused by ``forcibly suppressing loop insertions''; this is a key input for the downward-deviation estimates in the next section.
Recall the notation \(V^{\partial}(W)\), \(V_{\mathrm{exp}}^{v,W}\), and \(\mathfrak H^{v,W}\) from Section~\ref{subsec:decomposition-fiber-factorization}.

\begin{theorem}\label{thm:stop-prob}
Fix \(u\in\mathbb N\) and an \(E\)-pruned finite path \(s'\) with \(|s'|=u\). 
Fix \(i\in\{0,\ldots,u\}\) and
\[
h=(v,A,\omega^\partial)\in\mathcal H_{\mathrm{tr}}^{s'[0,i]}.
\]
Assume we condition on the following information:
\[
S'[0,u]=s',\qquad \mathcal W_0,\ldots,\mathcal W_{i-1},\qquad v\in V^{\partial}(W^i),\qquad
\mathfrak H^{v,\mathcal W_i}=h,
\]
and denote this conditioning event by \(\mathcal G\). Then
\[
\mathbb P(\operatorname{Next}_{\mathcal W_i}(v)=0\mid \mathcal G)\ge \gamma.
\]
\end{theorem}

For the proof of Theorem~\ref{thm:stop-prob}, we need two preliminary results.

\begin{lemma}\label{lem:Seg-summable}
For every finite loop set \(E\),
\[
\sum_{\eta\in\mathsf{Seg}(E)} p(\eta)\le\gamma^{-1}.
\]

\end{lemma}

\begin{proof}
Let
\[
\mathcal L_0:=\bigl\{\eta:\eta\text{ is a finite path with }\eta_0=\eta_{|\eta|}=0\bigr\}
\]
be the set of rooted finite loops at the origin. The identity
\(
\sum_{\eta\in\mathcal L_0}p(\eta)=\gamma^{-1}
\)
is classical; see Theorem~4.1.1 in \cite{lawler2010random}. Since \(\mathsf{Seg}(E)\subset\mathcal L_0\), the claim follows.
\end{proof}

Recall the notation \(\mathsf{P}(s', E)\) and \(\mathsf{Seg}^\otimes(s', E)\), and their
relation from Proposition~\ref{pro:one-to-one}.

\begin{proposition}\label{prop:condi}
For any \(u \in \mathbb{N}\), conditionally on \(S'[0, u]\), we have that for every
\((\xi_k)_{k=0}^{u} \in \mathsf{Seg}^\otimes\big(S'[0, u], E\big)\),
\begin{equation}\label{eq:condi}
\mathbb{P}\Big\{(\Xi_k)_{k=0}^{u} = (\xi_k)_{k=0}^{u} \,\Big|\, S'[0, u]\Big\}
\ \propto\ \prod_{k=0}^u p(\xi_k).
\end{equation}
\end{proposition}

\begin{proof}
Fix an \(E\)-pruned finite path \(s'\) with \(|s'|=u\). Let
\(\eta\in \mathsf{P}(s', E)\) be the path corresponding to $(\xi_k)_{k=0}^{u}$ via the
one-to-one correspondence in Proposition~\ref{pro:one-to-one}.

Observe that
\[
\big\{(\Xi_k)_{k=0}^{u} = (\xi_k)_{k=0}^{u},\ S'[0, u]=s'\big\}
=
\big\{S[0,|\eta|]=\eta\big\}\cap A(s'),
\]
where \(A(s')\) is the event that the step \((S_{|\eta|},S_{|\eta|+1})\) is retained in \(\Prune(S,E)\), and that the portion \(s'\) is retained under the
subsequent pruning, i.e.,
\(N_n\big(s'\circ_t (S[|\eta|,k]-S_{|\eta|}),E\big)\ge |s'|\) for all
\(k\ge |\eta|\).
Note that $A(s')$ is measurable with respect to the future path
$\sigma\big(S_k-S_{|\eta|}:k\ge |\eta|\big)$.
Therefore, by the Markov property at time $|\eta|$, we obtain \eqref{eq:condi}.
\end{proof}

\begin{proof}[Proof of Theorem \ref{thm:stop-prob}]
By Proposition~\ref{prop:condi}, for every \(\eta\in\mathsf{Seg}(s'[0,i],E)\),
\begin{equation}\label{eq:conditional-fiber-proportionality}
\begin{aligned}
\mathbb P(\Xi_i=\eta\mid\mathcal G)
&\propto
p(\eta)\mathbf 1{\{\eta\in\mathsf{Path}^{s'[0,i]}(h)\}}\\
&\propto
p(\eta_{\mathrm{par}}^{v,W})
\mathbf 1{\{\eta_{\mathrm{par}}^{v,W}\in
\mathsf{Seg}(s'[0,i]\circ_t \pi_{\mathrm{par}(v)}^h,E)\}}\\
&\quad
\prod_{q=1}^{|v|-1}\prod_{u\in\mathcal Y_q^h}
\left[
p(\eta_u^W)
\mathbf 1{\{\eta_u^W\in\mathsf{Seg}(s'[0,i]\circ_t \pi_u^h,E)\}}
\right],
\qquad W=\mathcal E(\eta).
\end{aligned}
\end{equation}
Here \(\mathcal Y_q^h\) denotes the common value of \(\mathcal Y_q^{v,W}\) on the path fiber \(\mathsf{Path}^{s'[0,i]}(h)\), and \(\pi_a^h\) denotes the corresponding common pruned prefix, as given by Theorem~\ref{thm:fiber-factorization-fixed-H}. 
The second proportionality follows from Proposition~\ref{prop:path-fiber-reconstruction} and the admissible classes in Theorem~\ref{thm:fiber-factorization-fixed-H}.

Notice that, on the path fiber,
\[
\{\operatorname{Next}_W(v)=0\}=\{\eta_{\mathrm{par}}^{v,W}=(0)\}.
\]
Therefore by \eqref{eq:conditional-fiber-proportionality}, 
\[\mathbb P(\operatorname{Next}_{\mathcal W_i}(v)=0\mid \mathcal G)
=
p((0))/{\displaystyle\sum_{\eta'\in \mathsf{Seg}(s'[0,i]\circ_t \pi_{\mathrm{par}(v)}^h,E)}p(\eta')}
\ge
\displaystyle\Big(\sum_{\eta'\in \mathsf{Seg}(E)}p(\eta')\Big)^{-1}
\ge \gamma,
\]
where the last inequality follows from Lemma~\ref{lem:Seg-summable}.
This completes the proof.
\end{proof}

\section{Proof of Theorem \ref{thm:downward}}\label{sec:proof-downward}
In this section, we prove Theorem~\ref{thm:downward}. For simplicity, we only prove the case $u=0$
namely,
\begin{equation}\label{eq:downward_u0_target}
\mathbb{P}\big(\xi^*(n)\le \beta\alpha\log n\big)
\ge \exp\left\{-(1+o(1))\,c_{\beta,n,0}\,\gamma\,n^{1-\beta}\right\},
\end{equation}
where we recall that
\(c_{\beta,n,0}:=n^{\beta}(1-\gamma)^{\lfloor \beta\alpha\log n\rfloor}\in[1,(1-\gamma)^{-1}).\) The proof for a general $u\in\mathbb{R}$ is analogous.

The proof proceeds by exhibiting a strategy that realizes the downward
deviation event, and then estimating the probability of this strategy.

\subsection{Strategy realizing the downward deviation}

In this subsection, we reduce the target downward-deviation estimate
to an event in the loop-pruning framework that is more amenable to probability
estimation.

\subsubsection{Reduction to an estimation-friendly event}
We first perform a reduction step. Fix for the moment a finite loop set \(E\), whose choice will be specified in \eqref{eq:Echoice}; roughly speaking, \(E\) will be rich enough to make the loop-pruned path have small local times. Let
\(S':=\mathsf{Prune}(S,E)\), and let \(\kappa=\kappa(E)\) be the constant from
Theorem~\ref{thm:moderate}. Fix \(a\in(1-\beta/2,1)\), so that \(2a-1>1-\beta\), and fix a constant \(\delta\in(0,\kappa)\). Set \(u_n:=\big\lfloor \kappa n+\delta n^a\big\rfloor\).
Recall that, as in Section~\ref{subsec:conditional-ES}, \(N^{-1}(u_n+1)-1\) is exactly the time at which the pruned segments \(\xi_0,\ldots,\xi_{u_n}\) have all been completely generated.
By Theorem~\ref{thm:moderate}, there exists \(c=c(E,\delta)>0\) such that, for all sufficiently large \(n\),
\[
\begin{aligned}
\mathbb P\big(N^{-1}(u_n+1)-1\le n\big)
&\le \mathbb P\big(N_{n+1}\ge u_n+1\big)\\
&\le \mathbb P\big(|N_{n+1}-\kappa(n+1)|>\tfrac12\delta n^a\big)
\le e^{-c n^{2a-1}}.
\end{aligned}
\]
Therefore,
\[
\mathbb P\big(\xi^*(n)\le \beta\alpha\log n\big)
\ge \mathbb P\big(\xi^*(N^{-1}(u_n+1)-1)\le \beta\alpha\log n\big)-e^{-c n^{2a-1}}.
\]
Since \(2a-1>1-\beta\), it is enough to prove
\[
\mathbb P\big(\xi^*(N^{-1}(u_n+1)-1)\le \beta\alpha\log n\big)
\ge \exp\left\{-(1+o(1))\,c_{\beta,n,0}\,\gamma\,n^{1-\beta}\right\}.
\]

Before closing this part, we record a useful upper-tail estimate for
\(N^{-1}(u_n+1)-1\), which will be used below.
Its proof is the same moderate-deviation argument as above.

\begin{lemma}\label{lem:inverse_clock_upper_tail}
There exists \(c=c(E,\delta)>0\) such that, for all sufficiently large \(n\),
\begin{equation}\label{eq:inverse_clock_upper_tail}
\mathbb P\!\left(N^{-1}(u_n+1)-1>n+n^a\right)
\le \exp\left\{-c\,n^{2a-1}\right\}.
\end{equation}
\end{lemma}

\subsubsection{Strategy event}

To describe the strategy event, it is useful to compare the original walk \(S\) and the loop-pruned walk \(S'\), and, when reconstructing \(S\) from \(S'\), to view the erased segments as being revealed in a canonical order. More precisely, for each erased segment \(\xi_i\) (indexed by \(0\le i\le u_n\)), write
\(W_i:=\mathcal{E}(\xi_i)\)
for its ES-representation, and imagine revealing first the root sequence \(W_i(\varnothing)\). This determines the vertex addresses of generation \(1\) (that is, addresses of length \(1\)). We then reveal, among all currently available vertex addresses, the one with minimal lexicographic order, and iterate this rule until all vertex addresses of \(W_i\) have been revealed. Revealing \(W_i\) is equivalent to revealing \(\xi_i\). Using this procedure, we reveal the pruned segments successively as
\(
\xi_0,\xi_1,\dots,\xi_{u_n}.
\)

In terms of the loop construction, this procedure reveals the erased segment loop by loop. The root sequence records the loops inserted directly into the pruned path, and revealing a child address amounts to opening one of those loops and exposing the loops inserted inside it. Thus the lexicographic exploration is simply a canonical order for uncovering all loop insertions in the tree representation of \(\xi_i\).

In the following, we shall use the notation from Section~\ref{sec:loop-pruned-decomposition}: the ES-representation \(\mathcal E\) from Section~\ref{sec:ES_represent}, the boundary-address set \(V^{\partial}(W)\) from Definition~\ref{def:boundary-addresses}, the path \(\eta_{\mathrm{exp}}^{v,W}\) and time \(\tau_{\mathrm{exp}}^{v,W}\) from Section~\ref{subsubsec:induced-path-decomposition}, and the insertion operation \(\triangleleft_j\) from \eqref{eq:loop-insertion-operation}.
Throughout this section, set \(\Lambda_n:=\lfloor\beta\alpha\log n\rfloor\).

\begin{definition}[Exceptional boundary pairs]\label{def:exceptional-boundary-pairs}
For each pair \((i,v)\in [0,u_n]\times \mathcal{U}\) with \(v\in V^{\partial}(W_i)\), called a boundary pair, define
\[
\eta^{(i,v)}
:=\bigl(S[0,N^{-1}(i)]\circ_t \eta_{\mathrm{exp}}^{v,W_i}\bigr)\circ_d S'[i,u_n],
\]
Equivalently, in the loop-pruned decomposition notation,
\[
\eta^{(i,v)}
=S'[0,u_n]\oplus(\xi_k^{(i,v)})_{0\le k\le u_n},
\qquad
\xi_k^{(i,v)}
=
\begin{cases}
\xi_k, & 0\le k<i,\\
\eta_{\mathrm{exp}}^{v,W_i}, & k=i,\\
(0), & i<k\le u_n.
\end{cases}
\]
Thus \(\eta^{(i,v)}\) is obtained from the pruned path \(S'[0,u_n]\) by adding back the pruned loops in their canonical order, stopping at the time corresponding to the boundary address \(v\) in the \(i\)-th erased segment.

For \(x\in\mathbb Z^d\), we say that \(x\) is \textit{critical for the boundary pair \((i,v)\)} if
\[
\begin{aligned}
&\ell_x\bigl(\eta^{(i,v)}\bigr)\ge \Lambda_n
\quad\text{and}\quad
\exists\,e\in E\ \text{such that }
\ell_x\bigl(\eta^{(i,v)}\triangleleft_j e\bigr)
>\ell_x\bigl(\eta^{(i,v)}\bigr),\\
&\text{with }j=N^{-1}(i)+\tau_{\mathrm{exp}}^{v,W_i}.
\end{aligned}
\]

We then define the exceptional boundary-pair set
\[
\mathsf{Exc}_n
:=
\Bigl\{
\text{boundary pairs }(i,v):
\exists\,x\in\mathbb Z^d
\text{ that is critical for }(i,v)
\Bigr\}.
\]
\end{definition}

Intuitively, \((i,v)\in\mathsf{Exc}_n\) means that, after adding back the pruned loops up to the time corresponding to the boundary address \(v\), some site has already reached the critical level \(\Lambda_n\), and inserting an \(E\)-loop at that time would increase its local time further.
\begin{remark}\label{rem:exceptional-boundary-pair-strategy}
The following two observations are useful consequences of this definition and the simplicity of loops in \(E\):
\begin{enumerate}[label=(\roman*)]
  \item Fix \(x\in\mathbb Z^d\), and let \((i,v)\) be the lexicographically first boundary pair for which \(x\) is critical. Then
  \[
  \ell_x\bigl(\eta^{(i,v)}\bigr)=\Lambda_n.
  \]
  Otherwise, one can find a lexicographically smaller boundary pair for which \(x\) is already critical.
  \item Suppose that whenever \(x\) is critical for some \((i,v)\), we do not ``attach any further loop'' \(e\in E\) that would worsen the local time at \(x\) (namely, increase it beyond its current value at that stage). Then the local time at \(x\) remains at
  \(\Lambda_n\).
\end{enumerate}
\end{remark}

We now define the strategy event \(G_n\) by
\begin{align}\label{def:G_n}
    G_n
:=
\Bigl\{
\xi^*(S'[0,u_n])\le \Lambda_n
\Bigr\}
\cap
\bigcap_{(i,v)\in \mathsf{Exc}_n}
\{\operatorname{Next}_{W_i}(v)=0\},
\end{align}
Here \(\xi^*(S'[0,u_n])\) denotes the maximum local time of the path \(S'[0,u_n]\).
In other words, on \(G_n\), whenever \((i,v)\in\mathsf{Exc}_n\), there is no further loop at the current parent after the prefix leading to \(v\).
Hence by Remark \ref{rem:exceptional-boundary-pair-strategy}, under this rule the local time at every site stays bounded by \(\Lambda_n\).

\begin{proposition}\label{prop:strategy-event-implies-no-overflow}
For any $\beta\in(0,1]$ and $n\in\mathbb{N}^+$,
\[
G_n
\subset
\Bigl\{
\xi^*\bigl(N^{-1}(u_n+1)-1\bigr)\le \beta\alpha\log n
\Bigr\}.
\]
\end{proposition}

Thus it remains to prove that
\begin{equation}\label{eq:Gn_lower_bound}
\mathbb{P}(G_n)\ge \exp\left\{-(1+o(1))\,c_{\beta,n,0}\,\gamma\,n^{1-\beta}\right\}.
\end{equation}

\subsection{Proof of \eqref{eq:Gn_lower_bound}}

We prove \eqref{eq:Gn_lower_bound} in three steps. First, we introduce the auxiliary law \(\mathbb Q\), analogous to the tilted measure in large deviations.
Next, we reduce the target bound to two propositions (Propositions~\ref{prop:Nn_small} and \ref{prop:QGn_conditional_lower}) that estimate quantities under \(\mathbb Q\): one estimates \(\#\mathsf{Exc}_n\), equivalently the number of forced no-loop operations in the strategy, and the other quantifies the resulting probability loss.
Finally, we prove these two propositions and conclude the desired lower bound.

\subsubsection{Tilted measure}\label{subsec:pruned-time-Q}

As in Section~4.3 of
\cite{li2026ldmaxlocal}, the first step is to introduce an auxiliary measure
\(\mathbb Q\). The measure \(\mathbb Q\) is a change-of-measure device, analogous to the standard tilting in large deviations. 

By Lemma~\ref{lem:inverse_clock_upper_tail}, \(N^{-1}(u_n+1)-1\le n+n^a\) with negligible failure probability. Fix \(\beta_1\in(0,\beta)\) and let
\[
T_n:=\lfloor n+n^a\rfloor,\qquad
\ell_n:=\lfloor n^{\beta_1}\rfloor,
\qquad
K_n:=\left\lfloor\frac{T_n}{\ell_n}\right\rfloor,
\qquad
t_j:=j\ell_n\ \ (0\le j\le K_n).
\]
Define the block intervals
\[
I_j:=[t_{j-1},t_j),\qquad 1\le j\le K_n,
\qquad I_{K_n+1}:=[t_{K_n},T_n].
\]
Next choose \(\beta_2\in\bigl(\tfrac{2}{d}\beta_1,\beta_1\bigr)\), and set \(r_n:=\lfloor n^{\beta_2}\rfloor\). For each block, define the inner interval
\[
I'_j:=[t_{j-1}+r_n,\,t_j),\qquad 1\le j\le K_n,
\qquad I'_{K_n+1}:=[t_{K_n}+r_n,T_n],
\]
with empty final intervals ignored. This is the same block construction as in \cite{li2026ldmaxlocal}.

We define the tilted event by requiring
that the maximum local time on each inner interval stays below a
common threshold. For a time interval \(I\), write
\[
\xi^*(I):=\max_{x\in\mathbb Z^d}\#\{m\in I:S_m=x\}.
\]
For \(1\le j\le K_n+1\), define
\begin{align}\label{def:An}
    A_{n,j}:=\Bigl\{\xi^*(I'_j)\le \Lambda_n\Bigr\},
\qquad
A_n:=\bigcap_{j=1}^{K_n+1}A_{n,j}.
\end{align}
We then define the probability measure \(\mathbb Q=\mathbb Q_n\) by
\[
\mathbb Q_n(\cdot):=\mathbb P(\cdot\mid A_n).
\]

\subsubsection{Reduction of \eqref{eq:Gn_lower_bound} to two propositions}
By the upward-deviation estimate in
\cite[Theorem~1.1]{li2026ldmaxlocal},
\[
\mathbb P\!\left(\xi^*(I'_j)>\Lambda_n\right)
=\bigl(1+o(1)\bigr)\,\gamma\,|I'_j|\,(1-\gamma)^{\Lambda_n+1},
\]
uniformly in $1\le j\le K_n+1$. Since $|I'_j|=\ell_n-r_n=(1+o(1))\ell_n$ for
all but the last block, applying the Markov property yields
\[
\mathbb P(A_n)
=\prod_{j=1}^{K_n+1} \big[1- \mathbb P\!\left(\xi^*(I'_j)>\Lambda_n\right)\big]
=\exp\left\{-(1+o(1))\,c_{\beta,n,0}\,\gamma\,n^{1-\beta}\right\}.
\]

Recall \eqref{eq:Gn_lower_bound}. The proof is reduced to the following two propositions. Set \(\mathcal{N}_n:=\#\mathsf{Exc}_n\).

\begin{proposition}\label{prop:Nn_small}
Fix a finite loop set \(E\) satisfying \eqref{eq:E-assumption-for-Nn-small} below. Then there exists a constant \(b<1-\beta\), depending only on \(E\), such that
\begin{equation}\label{eq:Nn_small}
\mathbb{Q}\bigl(\mathcal{N}_n\le n^b,\,
\xi^*(S'[0,u_n])\le \Lambda_n\bigr)=1-o(1).
\end{equation}
Consequently,
\begin{equation}\label{eq:Nn_small_P_version}
\mathbb{P}\bigl(\mathcal{N}_n\le n^b,\,
\xi^*(S'[0,u_n])\le \Lambda_n\bigr)
\ge
\exp\left\{-(1+o(1))\,c_{\beta,n,0}\,\gamma\,n^{1-\beta}\right\}.
\end{equation}
\end{proposition}
\begin{remark}
The exponent \(b\) in Proposition~\ref{prop:Nn_small} is allowed to be negative. In that case, \(\mathcal{N}_n\le n^b\) is equivalent to \(\mathcal{N}_n=0\).
\end{remark}

\begin{proposition}\label{prop:QGn_conditional_lower}
For the same \(b<1-\beta\) as in Proposition~\ref{prop:Nn_small}, and for every \(n\in\mathbb{N}^+\),
\begin{equation}\label{eq:QGn_conditional_lower}
\mathbb{P}(G_n\,|\,\mathcal{N}_n\le n^b,\,\xi^*(S'[0,u_n])\le \Lambda_n)\ge \gamma^{n^b}.
\end{equation}
\end{proposition}
\begin{proof}[Proof of \eqref{eq:Gn_lower_bound} assuming Propositions \ref{prop:Nn_small} and \ref{prop:QGn_conditional_lower}]
Multiplying \eqref{eq:Nn_small_P_version} and \eqref{eq:QGn_conditional_lower} yields the conclusion.
\end{proof}
\subsubsection{Preliminaries for Propositions \ref{prop:Nn_small} and \ref{prop:QGn_conditional_lower}}

The proofs of Propositions \ref{prop:Nn_small} and \ref{prop:QGn_conditional_lower}
follow the same general strategy as in \cite[Section~4.3.2]{li2026ldmaxlocal}.
In the present setting, one again needs some preliminary estimates controlling the probability
that a given point is visited by several time blocks $I_j$.
Since the arguments are essentially the same, we only record the corresponding statement
and omit its proof.

For \(x\in\mathbb Z^d\), define
\(\mathcal J_x:=\{1\le j\le K_n+1:\ x\in S(I_j)\}\), namely the set of blocks whose trajectories visit \(x\).

\begin{proposition}\label{prop:block-multiplicity}
There exists a constant \(M=M(\beta_2)<\infty\), independent of \(n\), such that
\[
\mathbb Q\Bigl(\sup_{x\in\mathbb Z^d}\#\mathcal J_x \le M\Bigr)=1-o\Big(\frac{1}{n^2}\Big).
\]
\end{proposition}

We fix \(M\) as in Proposition~\ref{prop:block-multiplicity}. We now make the choice of \(E\) announced above. By Proposition~\ref{prop:Sprime-localtime-tail}, there exists a finite loop set
\begin{equation}\label{eq:Echoice}
E=E(\beta,\beta_2,M)    
\end{equation}
such that for all \(N\in\mathbb{N}^+\) and all sufficiently large \(n\),
\begin{equation}\label{eq:E-assumption-for-Nn-small}
\mathbb{P}\bigl(\ell_0(\Prune(S[0,N],E))>\tfrac{\beta\alpha}{M}\log n\bigr)\le \frac{1}{n^2}.
\end{equation}
This is the assumption on \(E\) required in Proposition~\ref{prop:Nn_small}.

\subsubsection{Proof of Proposition \ref{prop:Nn_small}}
To prove Proposition~\ref{prop:Nn_small}, it is enough to prove the following two statements:
\begin{proposition}\label{prop:Nn_small_local_time}
  As $n\rightarrow\infty$,
  \begin{align}\label{eq:Nn_small_local_time}
    \mathbb{Q}(\xi^*(S'[0,u_n]) > \Lambda_n) = o(1).
  \end{align}
\end{proposition}

\begin{proposition}\label{prop:Nn_small_expectation}
  There exists $b<1-\beta$ such that for all $n\in\mathbb{N}^+$,
  \begin{align}\label{eq:Nn_small_expectation}
    \mathbb{Q}(\mathcal{N}_n>n^b)=o(1).
  \end{align}
\end{proposition}

We begin with Proposition~\ref{prop:Nn_small_local_time}. By Lemma \ref{lem:inverse_clock_upper_tail} and Proposition \ref{prop:block-multiplicity}, it suffices to show that
\begin{align}\label{eq:Nn_small_local_time_reduction}
    \mathbb{Q}\big(\xi^*(S'[0,u_n]) > \Lambda_n,\,N^{-1}(u_n+1)-1\le n+n^a,\,\sup_{x\in\mathbb{Z}^d}\#\mathcal{J}_x\le M\big) = o(1).
  \end{align}
Recall that for any path \(\eta\) and \(x\in\mathbb Z^d\), we write
\(
\ell_x(\eta)=\sum_j \mathbf 1\{\eta_j=x\}
\)
for the local time of \(\eta\) at \(x\). Write \(I_i=[I_i^-,I_i^+]\). For \(0\le k\le n+n^a\), let \(i_k\) be the unique index such that \(k\in I_{i_k}\). The following lemma is the key reduction.

\begin{lemma}\label{lem:block-reduction-local-time}
We have
\begin{align}\label{eq:block-reduction}
\begin{aligned}
&\big\{\xi^*(S'[0,u_n])>\Lambda_n,\,
\sup_{x\in\mathbb{Z}^d}\#\mathcal{J}_x\le M,\,N^{-1}(u_n+1)-1\le n+n^a\big\}\\
\subset&\Big\{\exists\,0\le k\le n+n^a,\,
\ell_{S_k}\big(\Prune(S[k,I_{i_k}^+],E)\big)
> \tfrac{\beta\alpha}{M}\log n\Big\}.    
\end{aligned}
\end{align}
\end{lemma}

\begin{proof}
Assume that \(\xi^*(S'[0,u_n]) > \Lambda_n\), \(\sup_{x\in\mathbb{Z}^d}\#\mathcal{J}_x \le M\), and \(N^{-1}(u_n+1)-1\le n+n^a\).
Then there exists some \(x\in\mathbb{Z}^d\) such that
\[
\ell_x(S'[0,u_n])>\Lambda_n.
\]
Let
\[
0\le r_1<r_2<\cdots<r_L\le u_n,
\qquad
L=\ell_x(S'[0,u_n])>\Lambda_n,
\]
be all times such that \(S'_{r_m}=x\).
For each \(1\le m\le L\), define
\[
R_m:=N^{-1}(r_m).
\]
Then \(S_{R_m}=x\) for every \(m\). Because \(\#\mathcal J_x\le M\), the pigeonhole principle yields a block \(I_j\) containing more than
\(
\frac{\beta\alpha}{M}\log n
\)
of the times \(R_m\).

Let \(k\) and \(k'\) be respectively the smallest and largest of those \(R_m\)'s lying in \(I_j\). Then
\[
k,k'\in I_j=I_{i_k}\quad\text{and}\quad
\ell_x\big(S'[N_k,N_{k'}]\big)>\frac{\beta\alpha}{M}\log n.
\]
By the loop-pruned decomposition,
\[
S'[N_k,N_{k'}]=\Prune(S[k,k'],E).
\]
Hence
\[
\ell_{S_k}\big(\Prune(S[k,k'],E)\big)
=
\ell_x\big(\Prune(S[k,k'],E)\big)
>
\frac{\beta\alpha}{M}\log n.
\]

Since \(k'\in I_{i_k}\), we have
\(
k'\le I_{i_k}^+.
\)
Applying the loop-pruned decomposition once more, we see that
\(
\Prune(S[k,k'],E)
\)
is an initial segment of
\(
\Prune(S[k,I_{i_k}^+],E).
\)
It follows that
\[
\ell_{S_k}\big(\Prune(S[k,I_{i_k}^+],E)\big)
\ge
\ell_{S_k}\big(\Prune(S[k,k'],E)\big)
>
\frac{\beta\alpha}{M}\log n.
\]
This proves the desired inclusion.
\end{proof}
\begin{proof}[Proof of Proposition~\ref{prop:Nn_small_local_time}]
    It follows from Lemma \ref{lem:block-reduction-local-time} that
    \begin{align}\label{eq:case_a_reduction_to_pruned_local_time}
    \begin{aligned}
        &\mathbb{Q}\Big(\xi^*(S'[0,u_n]) > \Lambda_n,\,N^{-1}(u_n+1)-1\le n+n^a,\,\sup_{x\in\mathbb{Z}^d}\#\mathcal{J}_x\le M\Big)\\
        \le\,& \sum_{k=0}^{\lfloor n+n^a\rfloor} \mathbb{Q}\Big(\ell_{S_k}\big(\Prune(S[k,I_{i_k}^+],E)\big)> \tfrac{\beta\alpha}{M}\log n\Big).        
    \end{aligned}
    \end{align}
    Note that the event on the right-hand side depends only on the trajectory segment \(S[k,I_j^+]\). In particular, among the conditioning events appearing in the definition of \(\mathbb Q\), it depends only on \(A_{n,i_k}\) and is
independent of all other events \(A_{n,j}\), $j\neq i_k$. Therefore, 
\begin{align*}
&\mathbb Q\Bigl(
\ell_{S_k}\bigl(\Prune(S[k,I_{i_k}^+],E)\bigr)>\tfrac{\beta\alpha}{M}\log n
\Bigr)\\
=\,&
\mathbb P\Bigl(
\ell_{S_k}\bigl(\Prune(S[k,I_{i_k}^+],E)\bigr)>\tfrac{\beta\alpha}{M}\log n
\,\Big|\,  A_{n,i_k}
\Bigr) \\
\le\,&(1+o(1))\,
\mathbb P\Bigl(
\ell_{S_k}\bigl(\Prune(S[k,I_{i_k}^+],E)\bigr)>\tfrac{\beta\alpha}{M}\log n
\Bigr)\\
=\,&(1+o(1))\,
\mathbb P\Bigl(
\ell_{0}\bigl(\Prune(S[0,I_{i_k}^+-k],E)\bigr)>\tfrac{\beta\alpha}{M}\log n
\Bigr),
\end{align*}
since
\(\mathbb{P}(A_{n,i_k})=1-o(1).\)
By the assumption \eqref{eq:E-assumption-for-Nn-small} on $E$, we get
\[\mathbb P\Bigl(
\ell_{0}\bigl(\Prune(S[0,I_{i_k}^+-k],E)\bigr)>\tfrac{\beta\alpha}{M}\log n
\Bigr)\le (1+o(1))\frac{1}{n^2}.\]
Substituting the above bounds into \eqref{eq:case_a_reduction_to_pruned_local_time} yields \eqref{eq:Nn_small_local_time_reduction}, and thus completes the proof of Proposition~\ref{prop:Nn_small_local_time}.
\end{proof}

Now turn to the proof of Proposition \ref{prop:Nn_small_expectation}. We first record a tail estimate for the maximum local time under \(\mathbb{Q}\). Its proof is entirely analogous to that of Proposition~\ref{prop:Nn_small_local_time}: follow the same argument, with Lemma~\ref{lem:N-geometric-tail} taking the place of Proposition~\ref{prop:Sprime-localtime-tail}. We therefore omit the details.

\begin{lemma}\label{lem:Q-max-local-time-tail}
There exists \(\rho>0\) such that, for all $n$ sufficiently large,
\begin{equation}\label{eq:Q-max-local-time-tail}
\mathbb{Q}\bigl(\xi^*(\lfloor n+n^a\rfloor)>\rho\log n\bigr)<\frac{1}{n}.
\end{equation}
\end{lemma}
\begin{proof}[Proof of Proposition~\ref{prop:Nn_small_expectation}.]
In view of Lemmas~\ref{lem:inverse_clock_upper_tail} and \ref{lem:Q-max-local-time-tail}, it remains to prove that there exists \(b<1-\beta\) such that
\begin{equation}\label{eq:reduction-to-Nn-expectation}
\mathbb{E}^{\mathbb{Q}}\Bigl[\mathcal{N}_n\mathbf{1}\bigl\{N^{-1}(u_n+1)-1\le n+n^a,\,\xi^*(\lfloor n+n^a\rfloor)\le \rho\log n\bigr\}\Bigr]\le n^b.
\end{equation}

Recall the definitions of \(\mathsf{Exc}_n\) and \(\eta^{(i,v)}\) from Definition~\ref{def:exceptional-boundary-pairs}. 
We denote
\[
D_E:=\max_{e\in E,\;1\le i\le |e|}|e_i|,
\]
the maximal distance from the base point attained by the loops in \(E\).
It follows directly from the definition of the exceptional boundary-pair set that, for every \((i,v)\in \mathsf{Exc}_n\), the point \(S_{N^{-1}(i)+\tau_{\mathrm{exp}}^{v,W_i}}\) lies within distance \(D_E\) of some point \(x\) satisfying
\[
\xi\bigl(\lfloor n+n^a \rfloor,x\bigr)\ge \Lambda_n .
\]
Thus we obtain the following lemma.

\begin{lemma}\label{lem:Nn-neighborhood-bound}
Denote
\[\mathcal{A}_n:=\{x\in\mathbb{Z}^d: \xi\bigl(\lfloor n+n^a \rfloor,x\bigr)\ge \Lambda_n\}.\]
Then on the event $\{N^{-1}(u_n+1)-1\le n+n^a\}$,
\begin{align*}
    \mathcal N_n \le
\sum_{m=0}^{\lfloor n+n^a\rfloor}
\mathbf 1\Bigl\{
S_m\in
\bigcup_{x\in\mathcal A_n}
B(x,D_E)
\Bigr\}\le
\sum_{x\in \mathcal{A}_n}\sum_{y\in B(x,D_E)} \xi(\lfloor n+n^a\rfloor,y).
\end{align*}
Here \(B(x,D_E):=\{y\in\mathbb Z^d:|y-x|\le D_E\}\),
and we use that each exceptional boundary pair corresponds to a distinct exploration time \(N^{-1}(i)+\tau_{\mathrm{exp}}^{v,W_i}\). 
Consequently,
\begin{align}\label{eq:Nn-trunc-bound}
\mathcal{N}_n\mathbf{1}\bigl\{N^{-1}(u_n+1)-1\le n+n^a,\,
\xi^*(\lfloor n+n^a\rfloor)\le \rho\log n\bigr\}
\le
\#B(0,D_E)\cdot \rho\log n \cdot \#\mathcal{A}_n.
\end{align}
\end{lemma}
Recalling \eqref{eq:reduction-to-Nn-expectation} and using \eqref{eq:Nn-trunc-bound}, it remains to prove that there exists \(b<1-\beta\) such that
\begin{equation}\label{eq:An-expectation-bound}
\mathbb{E}^{\mathbb{Q}}\bigl[\#\mathcal{A}_n\bigr]\le n^b.
\end{equation}
This is the same thick-point counting estimate as Proposition~4.5 in
\cite{li2026ldmaxlocal}, adapted to the present tilted measure
\(\mathbb Q\), and we omit the proof. This completes the proof of Proposition~\ref{prop:Nn_small_expectation}.
\end{proof}

\subsubsection{Proof of Proposition \ref{prop:QGn_conditional_lower}}
We enumerate the elements of \(\mathsf{Exc}_n\) in lexicographical order:
\begin{equation}\label{eq:Exc-lex-order}
\mathsf{Exc}_n
=
\{(i_1,v_1),(i_2,v_2),\ldots,(i_{\mathcal N_n},v_{\mathcal N_n})\}.
\end{equation}
For each \(1\le j\le \mathcal N_n\), define
\[
E_j:=\bigl\{\operatorname{Next}_{W_{i_j}}(v_j)=0\bigr\},
\]
and \(\tilde{E}_j=E_j\cup \{j>\mathcal{N}_n\}\).
For the proof, we need the following two lemmas. In what follows, we fix \(b<1-\beta\) as in Proposition~\ref{prop:Nn_small}.
We write \(\overline{\mathbb{P}}:=\mathbb{P}(\cdot | \xi^*(S'[0,u_n])\le \Lambda_n)\).
\begin{lemma}\label{lem:conditional-Aj-lower}
For every \(j,n\in\mathbb{N}^+\),
\begin{equation}\label{eq:product-tildeAj}
\overline{\mathbb{P}}\Bigl(
\tilde{E}_j \,\Big|\, 
\bigcap_{l=1}^{j-1}\tilde{E}_l
\Bigr)\ge \gamma,\quad \mbox{ and consequently }\quad \overline{\mathbb{P}}\Bigl(
\bigcap_{j=1}^{\lfloor n^b\rfloor}\tilde{E}_j 
\Bigr)\ge \gamma^{n^b}.
\end{equation}
\end{lemma}
\begin{lemma}\label{lem:Aj-Nn-factorization}
For any \(n\in\mathbb{N}^+\), we have
\begin{equation}\label{eq:Aj-Nn-factorization}
\overline{\mathbb{P}}\Bigl(
\bigcap_{j=1}^{\lfloor n^b\rfloor}\tilde{E}_j,\,\mathcal{N}_n\le n^b \Bigr)
\ge
\overline{\mathbb{P}}\Bigl(
\bigcap_{j=1}^{\lfloor n^b\rfloor}\tilde{E}_j \Bigr)\cdot 
\overline{\mathbb{P}}\Bigl(
\mathcal{N}_n\le n^b \Bigr).
\end{equation}
\end{lemma}

\begin{proof}[Proof of Proposition~\ref{prop:QGn_conditional_lower} assuming Lemmas~\ref{lem:conditional-Aj-lower} and \ref{lem:Aj-Nn-factorization}]
Note that
\[
G_n\cap\{\mathcal{N}_n\le n^b\}
=
\Bigl(\bigcap_{j=1}^{\lfloor n^b\rfloor}\tilde{E}_j\Bigr)\cap\{\mathcal{N}_n\le n^b\}.
\]
Therefore,
$$
\overline{\mathbb{P}}\bigl(
G_n \,\big|\, \mathcal{N}_n\le n^b
\bigr)
\ge
\overline{\mathbb{P}}\Bigl(
\bigcap_{j=1}^{\lfloor n^b\rfloor}\tilde{E}_j\,\Big|\,
\mathcal{N}_n\le n^b
\Bigr)
\overset{\eqref{eq:Aj-Nn-factorization}}{\ge}\overline{\mathbb{P}}\Bigl(
\bigcap_{j=1}^{\lfloor n^b\rfloor}\tilde{E}_j \Bigr)\overset{\eqref{eq:product-tildeAj}}{\ge} \gamma^{n^b}.
$$
This completes the proof.
\end{proof}
\begin{proof}[Proof of Lemma~\ref{lem:conditional-Aj-lower}]
It suffices to prove that
\begin{equation}\label{eq:conditional-Ej-lower}
\mathbb{P}\Bigl(
E_j \,\Big|\, \xi^*(S'[0,u_n])\le \Lambda_n,\,
\bigcap_{l=1}^{j-1}E_l,\,
j\le \mathcal N_n
\Bigr)\ge \gamma.
\end{equation}

Note that the conditioning event
\[
\Bigl\{
\xi^*(S'[0,u_n])\le \Lambda_n,\,
\bigcap_{l=1}^{j-1}E_l,\,
j\le \mathcal N_n
\Bigr\}
\]
can be written as a disjoint union of all possible values of the following data up to the exceptional boundary pair \((i_j,v_j)\):
\begin{align*}
&S'[0,u_n]=s',\qquad (i_j,v_j),\qquad W_0,\ldots,W_{i_j-1},\qquad
\mathfrak H^{v_j,W_{i_j}}=h_j.
\end{align*}
Fix one such atom \(\Gamma\) with positive probability. By Theorem~\ref{thm:stop-prob}, 
\[
\mathbb{P}(E_j\mid \Gamma)\ge \gamma.
\]
Averaging over all atoms \(\Gamma\), we obtain \eqref{eq:conditional-Ej-lower}.
\end{proof}

\begin{proof}[Proof of Lemma~\ref{lem:Aj-Nn-factorization}]

It suffices to show that for each \(1\le m\le \lfloor n^b\rfloor\),
\begin{equation*}
\overline{\mathbb{P}}\Bigl(\mathcal{N}_n\le n^b \,\Big|\, \bigcap_{j=1}^{m}\tilde{E}_j\Bigr)
\ge
\overline{\mathbb{P}}\Bigl(\mathcal{N}_n\le n^b \,\Big|\, \bigcap_{j=1}^{m-1}\tilde{E}_j\Bigr).
\end{equation*}
To this end, it is enough to prove
\begin{equation}\label{eq:Nn-monotone-with-m}
\mathbb{P}^m(\mathcal{N}_n\le n^b \mid E_m)\ge
\mathbb{P}^m(\mathcal{N}_n\le n^b),\quad\mbox{where}\quad \mathbb{P}^m:=\overline{\mathbb{P}}\Bigl(\,\cdot\,\Big|\, \bigcap_{j=1}^{m-1}E_j,\,
m\le \mathcal{N}_n\Bigr).
\end{equation}
Recall the decomposition of \( V(W)\) relative to \(v\) from Definition~\ref{def:V-decomposition-relative-v}. We write the event
\[
\Bigl\{\bigcap_{j=1}^{m-1}E_j,\,
m\le \mathcal{N}_n,\,
\xi^*(S'[0,u_n])\le \Lambda_n\Bigr\}
\]
as a disjoint union over all possible realizations of the following data:
\begin{align*}
&S'[0,u_n]=s',\qquad (i_m,v_m),\qquad
W_0,\ldots,W_{i_m-1},W_{i_m+1},\ldots,W_{u_n},\\
&\mathfrak H^{v_m,W_{i_m}}=h_m,\qquad
\bigl(\eta_u^{W_{i_m}}:1\le q\le |v_m|-1,\ u\in\mathcal Y_q^{h_m}\bigr).
\end{align*}
In other words, we condition on all the information except the still-unrevealed parent-level piece \(\eta_{\mathrm{par}}^{v_m,W_{i_m}}\).

Fix one such atom \(\Gamma\) with positive probability. 
By the reconstruction in Proposition~\ref{prop:path-fiber-reconstruction}, on \(\Gamma\) there are a deterministic path \(\eta_\Gamma\) and a deterministic time \(t_\Gamma\), corresponding to the address \((i_m,v_m)\), such that
\[
S[0,N^{-1}(u_n+1)-1]
=
\eta_\Gamma\triangleleft_{t_\Gamma}
\eta_{\mathrm{par}}^{v_m,W_{i_m}},
\]
where \(\eta_{\mathrm{par}}^{v_m,W_{i_m}}\) is the only remaining random path piece.

On \(\Gamma\), if \(E_m\) occurs, then the unexplored parent-level tail after the prefix encoded by \(h_m\) is trivial, i.e.,
\(\eta_{\mathrm{par}}^{v_m,W_{i_m}}=(0)\).
Thus, conditionally on \(\Gamma\cap E_m\), the path is deterministic:
\[
S^{(1)}[0,N^{-1}(u_n+1)-1]
:=
\eta_\Gamma\triangleleft_{t_\Gamma}(0)=\eta_\Gamma.
\]
We couple it with the path under
\(\mathbb P^m(\,\cdot\,\mid \Gamma)\) by sampling
\[
S^{(0)}[0,N^{-1}(u_n+1)-1]
=
\eta_\Gamma\triangleleft_{t_\Gamma}
\eta_{\mathrm{par}}^{v_m,W_{i_m}}
\]
under \(\mathbb P^m(\,\cdot\,\mid \Gamma)\). It is easy to see that, in this coupling,
\(\mathcal N_n(S^{(1)})\le \mathcal N_n(S^{(0)})\). Averaging over \(\Gamma\)
gives \eqref{eq:Nn-monotone-with-m}.
\end{proof}
\section{Proof of Theorem~\ref{thm:cut}}\label{sec:cut}
In this section, we prove Theorem~\ref{thm:cut}. Throughout, we work with the two-sided random walk \(S(-\infty,\infty)\).
The key idea is to single out a convenient subclass of \(E\)-cut steps, which we call \textit{rod \(E\)-cut steps}.
We then relate the mechanism that generates rod \(E\)-cut steps to a stack interpretation.
This correspondence allows us to show that rod \(E\)-cut steps occur with positive density, which yields Theorem~\ref{thm:cut}.

We now introduce rod \(E\)-cut steps.

\subsection{Rod paths and rod \(E\)-cut steps}

Fix a unit vector \(x_0\in\mathbb{Z}^d\).

\begin{definition}[Rod paths]\label{def:rod_paths}
A finite path \(\eta=(\eta_0,\ldots,\eta_{2L})\) is called a \textit{rod path} with parameter \(L\) if
\[
\eta_L=0,
\qquad
\eta_j-\eta_{j-1}=x_0\quad\text{for all }1\le j\le 2L.
\]
\end{definition}

Thus, a rod path is a straight trajectory of length \(2L\) in the direction \(x_0\) that passes through the origin at its midpoint.

\begin{definition}[Rod points and rod \(E\)-cut points]\label{def:rod_points}
Fix \(L\in\mathbb{N}^+\). A time \(j\in\mathbb{Z}\) is called a \textit{rod point with parameter \(L\)} if the translated path
\[
S[j-L,\,j+L]-S_j
\]
is a rod path with parameter \(L\).
If, in addition, the step \((S_{j-1},S_j)\) is an \(E\)-cut step in the sense of Definition~\ref{def:E_cut_steps}, then \(j\) is called a \textit{rod \(E\)-cut point with parameter \(L\)}.
In this case, we also refer to \((S_{j-1},S_j)\) as a \textit{rod \(E\)-cut step with parameter \(L\)}.
\end{definition}

For an integer interval \(I\subset\mathbb{Z}\), we write \(\mathrm{Cut}^{\mathrm{rod}}_L(I)\) for the set of rod \(E\)-cut points with parameter \(L\) contained in \(I\), namely
\[
\mathrm{Cut}^{\mathrm{rod}}_L(I)
:=\big\{\, j\in I:\ j \text{ is a rod \(E\)-cut point with parameter }L\,\big\}.
\]

The following theorem is the main result of this section.
Define
\begin{align}\label{def:LE}
L_E 
:= 1 + \max_{e \in E} \operatorname{diam}(e),
\qquad
\operatorname{diam}(e)
:= \sup_{0 \le i,j \le |e|} |e_i - e_j|.
\end{align}
Thus, \(L_E\) is one plus the maximal spatial diameter of loops in \(E\).
\begin{theorem}\label{thm:rod_cut}
For any \(L=(2K+1)L_E\) with \(K\) large enough satisfying \eqref{eq:K} below, there exist constants \(\theta=\theta(E,L)\in(0,1)\) and \(c=c(E,L)>0\) such that, for every integer interval \(I\subset\mathbb{Z}\) whose length \(|I|\) is sufficiently large and an integer multiple of \(2L\),

\[
\mathbb{P}\!\left(\#\,\mathrm{Cut}^{\mathrm{rod}}_L(I)>\theta |I|\right)\ge 1-\exp(-c|I|).
\]
\end{theorem}
The integer-multiple assumption on \(|I|\) is imposed only for convenience in the proof. Since every rod \(E\)-cut step is an \(E\)-cut step, Theorem~\ref{thm:rod_cut} immediately implies Theorem~\ref{thm:cut}.

The proof proceeds by identifying a necessary condition for a rod point to be pruned by the pruning procedure. 
This condition is formulated in terms of an explicitly checkable return pattern, whose probability can be controlled uniformly. 
As a consequence, most rod points are not pruned, which yields the desired positive-density lower bound for \(\mathrm{Cut}^{\mathrm{rod}}_L(I)\).

Henceforth, we fix \(L=(2K+1)L_E\) with \(K\in\mathbb{N}^+\) chosen sufficiently large, and write \(I=[I^-,I^+]\), where \(|I|\) is an integer multiple of \(2L\).
\subsection{Preliminary estimates on rod points}
We begin with a preliminary estimate on the abundance of rod points. Set \(M:=|I|/(2L)\), and divide \(I\) into intervals of length \(2L\):
\[
I_r:=[I^-+2(r-1)L,\,I^-+2rL]\cap\mathbb{Z},
\qquad r=1,\ldots,M.
\]
Let
\[
\bar\rho_r:=I^-+(2r-1)L
\]
be the midpoint of \(I_r\). We then define \(\rho_1,\ldots,\rho_m\) to be the midpoints of those intervals \(I_r\) on which the path is a rod path, in chronological order:
\begin{align}\label{def:rho}
\begin{aligned}
\{r_1<\cdots<r_m\}
&:=\big\{\,1\le r\le M:\ S[I_r]-S_{\bar\rho_r}
\text{ is a rod path with parameter }L\,\big\},\\
\rho_i=\rho_i(I)&:=\bar\rho_{r_i},\qquad i=1,\ldots,m.
\end{aligned}
\end{align}
We use the convention \(\rho_i:=\infty\) for \(i>m\).

The i.i.d. Bernoulli structure of the rod-block indicators gives the following high-probability lower bound.

\begin{lemma}\label{lem:many-rod-points}
There exist constants \(\theta'=\theta'(L)\in(0,1)\) and \(c=c(L,\theta')>0\) such that for any
integer interval \(I\subset\mathbb{Z}\) whose length \(|I|\) is sufficiently large and an integer multiple of \(2L\),
\[
\mathbb{P}\!\left(m>\theta'|I|\right)\ge 1-\exp(-c|I|).
\]
\end{lemma}

In view of Lemma~\ref{lem:many-rod-points}, it suffices to work on the event
\(\{m>\theta'|I|\}\). Note also that, deterministically,
\(
m \le \tfrac{|I|}{2L},
\)
since there is at most one candidate midpoint in each interval \(I_r\).
It therefore remains to show that, on \(\{m>\theta'|I|\}\), pruning more than half
of the rod points is exponentially unlikely.

\subsection{A necessary condition for many pruned rod points}
For the remainder of this subsection, we restrict attention to the event
\(\{m>\theta'|I|\}\).

In this part, we derive a necessary condition for the \(E\)-pruning procedure
to remove more than half of the rod points. We work with the ordered collection
\(\rho_1,\ldots,\rho_m\) of rod points introduced in \eqref{def:rho}.
Roughly speaking, if more than half of these rod points are pruned, then the walk
must exhibit a comparably large number of returns to successive subsegments of the
associated rod paths, in an order compatible with the pruning dynamics.
We encode this requirement via return patterns, which are explicitly checkable and admit
uniform probability bounds.
We begin by introducing the necessary notation.
\begin{definition}[Rod sub-segments]\label{def:rod_subsegments}
For \(i=1,\ldots,m\) and \(j=1,\ldots,K\), define the time intervals
\begin{align*}
R^-_{ij} &:= \bigl\{ u \in \mathbb{Z} : (2j-1)L_E \le \rho_i-u \le 2jL_E \bigr\}, \\
R^+_{ij} &:= \bigl\{ u \in \mathbb{Z} : (2j-1)L_E \le u-\rho_i \le 2jL_E \bigr\}.
\end{align*}
The corresponding path segments are
\[
S(R^-_{ij}) := (S_u:\ u\in R^-_{ij}),
\qquad
S(R^+_{ij}) := (S_u:\ u\in R^+_{ij}).
\]
We refer to these path segments as the \textit{rod sub-segments}. See Figure~\ref{fig:rod_subsegments} for an illustration.
\end{definition}

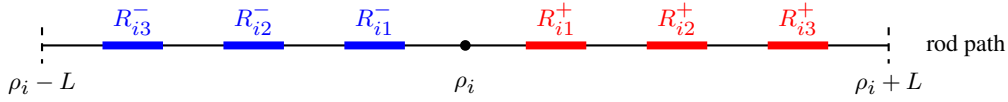
\begin{figure}[ht]
\centering
\begin{tikzpicture}[thick,scale=0.8]

\draw[black] (0,0) -- (14,0) node[right,xshift=0.3cm] {\ rod path};

\filldraw (7,0) circle (2pt) node[below,yshift=-0.3cm] {$\rho_i$};

\draw[blue, line width=3pt] (1,0) -- (2,0);
\draw[blue, line width=3pt] (3,0) -- (4,0);
\draw[blue, line width=3pt] (5,0) -- (6,0);

\draw[red, line width=3pt] (8,0) -- (9,0);
\draw[red, line width=3pt] (10,0) -- (11,0);
\draw[red, line width=3pt] (12,0) -- (13,0);

\draw[dashed] (0,0.3) -- (0,-0.3) node[below] {$\rho_i-L$};
\draw[dashed] (14,0.3) -- (14,-0.3) node[below] {$\rho_i+L$};

\node[above, blue] at (1.5, 0.02) {$R^-_{i3}$};
\node[above, blue] at (3.5, 0.02) {$R^-_{i2}$};
\node[above, blue] at (5.5, 0.02) {$R^-_{i1}$};
\node[above, red] at (8.5, 0.02) {$R^+_{i1}$};
\node[above, red] at (10.5, 0.02) {$R^+_{i2}$};
\node[above, red] at (12.5, 0.02) {$R^+_{i3}$};

\end{tikzpicture}
\caption{Illustration of the rod path $S[\rho_i-L,\rho_i+L]$ with rod time intervals $R^-_{ij}$ (blue) and $R^+_{ij}$ (red).}
\label{fig:rod_subsegments}
\end{figure}
The necessary condition will be stated in terms of the following return patterns.

\begin{definition}[Return patterns]\label{def:return_pattern_subset}
Fix a nonempty subset \(J\subset [|I|]=\{1,\ldots,|I|\}\) (in later applications, mainly \(J\subset [m]\)), and write its elements in increasing order as
\[
J=\{j_1<j_2<\cdots<j_{\# J}\}.
\]

\smallskip
\noindent\textbf{Return patterns.}
A forward return pattern on \(J\) is a vector \(\mathbf a=(a_1,\ldots,a_{\# J})\in\mathbb{N}^{\# J}\) such that
\begin{equation}\label{eq:pattern_J_forward}
\sum_{k=1}^{\ell} a_k\le \ell K \qquad \text{for all } \ell=1,\ldots,\# J.
\end{equation}
We denote the set of all such forward patterns by \(\mathcal P^{+}(J)\).

A backward return pattern on \(J\) is a vector \(\mathbf a=(a_1,\ldots,a_{\# J})\in\mathbb{N}^{\# J}\) such that the reversed sequence
\[
\check{\mathbf a}:=(a_{\# J},a_{\# J-1},\ldots,a_1)
\]
satisfies \eqref{eq:pattern_J_forward}. We denote the set of all such backward patterns by \(\mathcal P^{-}(J)\).

\smallskip
\noindent\textbf{Priority orders.}
We order index pairs \((p,q)\in [m]\times[K]\) by
\[
(p,q)\succ_+(p',q')
\quad\Longleftrightarrow\quad
\bigl(p>p'\bigr)\ \text{or}\ \bigl(p=p'\ \text{and}\ q>q'\bigr),
\]
and
\[
(p,q)\succ_-(p',q')
\quad\Longleftrightarrow\quad
\bigl(p<p'\bigr)\ \text{or}\ \bigl(p=p'\ \text{and}\ q>q'\bigr).
\]

\smallskip
\noindent\textbf{Forward realization.}
The following definition is somewhat involved; the stack interpretation in Remark~\ref{rem:urgent-file-stack} gives an intuitive explanation of the forward realization rule below.
Let \(J\subset [m]\) and \(\mathbf a\in\mathcal P^{+}(J)\).
We define lists \(\Lambda^{+}_\ell(\mathbf a)\) recursively via used-index sets \(\mathcal V^{+}_\ell(\mathbf a)\).
Set \(\mathcal V^{+}_0(\mathbf a):=\varnothing\). For each \(\ell=1,\ldots,\# J\), given \(\mathcal V^{+}_{\ell-1}(\mathbf a)\), define
\[
\mathcal I^{+}_\ell:=\{(p,q):\ p\in J,\ p\le j_\ell,\ 1\le q\le K\},
\qquad
\mathcal U^{+}_\ell(\mathbf a):=\mathcal I^{+}_\ell\setminus \mathcal V^{+}_{\ell-1}(\mathbf a),
\]
and let \(\Lambda^{+}_\ell(\mathbf a)\) be the ordered list consisting of the first \(a_\ell\) elements of \(\mathcal U^{+}_\ell(\mathbf a)\)
when \(\mathcal U^{+}_\ell(\mathbf a)\) is arranged in decreasing order \(\succ_+\).
Finally, set
\[
\mathcal V^{+}_\ell(\mathbf a):=\mathcal V^{+}_{\ell-1}(\mathbf a)\cup \{\text{elements appearing in }\Lambda^{+}_\ell(\mathbf a)\}.
\]

Recall the right rod sub-segments \(R^{+}_{pq}\).
We say that \(S\) realizes \(\mathbf a\) on \(J\) in the forward sense, and denote this event by \(\mathsf{Ret}^{+}_J(\mathbf a)\), if for every \(\ell=1,\ldots,\# J\),
writing
\[
\Lambda^{+}_\ell(\mathbf a)=\bigl((p^{(\ell)}_1,q^{(\ell)}_1),\ldots,(p^{(\ell)}_{a_\ell},q^{(\ell)}_{a_\ell})\bigr),
\]
there exist times \(u^{(\ell)}_1\le \cdots\le u^{(\ell)}_{a_\ell}\) with
\[
u^{(\ell)}_k\in(\rho_{j_\ell}+L,\rho_{j_{\ell+1}}+L],
\quad\text{for }k=1,\ldots,a_\ell,
\qquad \rho_{j_{\# J+1}}:=+\infty,
\]

such that
\[
S_{u^{(\ell)}_k}\in S\!\left(R^{+}_{p^{(\ell)}_k\,q^{(\ell)}_k}\right),
\qquad k=1,\ldots,a_\ell.
\]

\smallskip
\noindent\textbf{Backward realization.}
Let \(\mathbf a\in \mathcal{P}^-(J)\), and define the time-reversed path
\[
\check S_n := S_{-n},\qquad n\in\mathbb{Z}.
\]
Then the points
\[
\check\rho_i:=-\rho_{m+1-i},\qquad i=1,\ldots,m,
\]
form the corresponding sequence of rod points for \(\check S\) (along the \(-x_0\) direction) on the time interval $-I$.
We say that \(S\) realizes \(\mathbf a\) on \(J\) in the backward sense, and we denote this event by \(\mathsf{Ret}^{-}_J(\mathbf a)\), if the time-reversed path \(\check S\) realizes \(\check{\mathbf a}\) on the index set \(\check J:=\{m+1-j:\ j\in J\}\) in the forward sense.

Under this time reversal, the right rod sub-segments for \(\check S\) correspond to the left rod sub-segments for \(S\), namely
\[
R^{+}_{pq}(\check S)=R^{-}_{(m+1-p)\,q}(S),
\]
and the forward priority order \(\succ_+\) for \(\check S\) matches the backward priority order \(\succ_-\) for \(S\).
\end{definition}

\begin{remark}[Stack interpretation]\label{rem:urgent-file-stack}
The following stack interpretation may help explain the ordering convention.
In particular, the dictionary below corresponds to the forward notion in Definition~\ref{def:return_pattern_subset}.

\smallskip
\noindent\textbf{The model.}
Fix integers \(m,K\ge 1\) and a nonempty subset \(J\subset[m]\). A clerk processes urgent files stored in a single stack. The stack is empty before day \(j_1\).

\smallskip
\noindent\textbf{Daily arrivals.}
On each day \(i\in J\), a batch of \(K\) new files arrives, labeled \((i,1),\ldots,(i,K)\).
They are placed on top of the current stack in the order \((i,1)\) up to \((i,K)\), so that \((i,1)\) is the lowest and \((i,K)\) is the topmost among the new files.

\smallskip
\noindent\textbf{Processing rule.}
The clerk always works from the top of the stack, and each processed file is immediately removed.
During a day, the clerk may process any number of files; in particular, the clerk can return to unfinished files from earlier days only after clearing the more recent batches stacked above them.

\smallskip
\noindent\textbf{Link with return patterns.}
We view the interval \((\rho_i+L,\rho_{i+1}+L]\) for \(i<m\), and \((\rho_m+L,\infty)\) for \(i=m\), as day \(i\); each rod sub-segment \(R^+_{pq}\) is then viewed as the file \((p,q)\) from day \(p\) (with arrivals only for \(p\in J\)).
The order \(\succ_+\) encodes the stack priority: files with larger \(p\) sit above those with smaller \(p\), and within the same day, larger \(q\) sits above smaller \(q\).
Processing the file \((p,q)\) corresponds to the walk returning to the rod sub-segment \(R^+_{pq}\) in the prescribed manner.
With this identification, a pattern \(\mathbf a=(a_1,\ldots,a_{\# J})\in\mathcal P^+(J)\) specifies how many files are processed from day \(j_\ell\) up to day \(j_{\ell+1}-1\),
and the associated list \(\Lambda^{+}_\ell(\mathbf a)\) records which files are processed (and in which order).
Accordingly, the event \(\mathsf{Ret}^{+}_J(\mathbf a)\) corresponds to a processing history in which the clerk processes exactly the files prescribed by the lists \(\Lambda^{+}_\ell(\mathbf a)\), for all \(\ell=1,\ldots,\# J\).
\end{remark}

We are now ready to state the main theorem of this subsection. Its proof is deferred to Section~\ref{sec:pf_necessary}.
For a vector \(\mathbf a=(a_1,\ldots,a_{\# J})\in\mathbb{N}^{\# J}\), we write
\[
\|\mathbf a\|_1:=\sum_{j=1}^{\# J} a_j .
\]
\begin{theorem}[Return-pattern reduction]\label{thm:many-pruned-implies-global-pattern}
If \(m>16\vee\theta'|I|\) and
\[
\#\big\{1\le i\le m:\ \rho_i\notin \mathrm{Cut}^{\mathrm{rod}}_L(I)\big\}>m/2
\]
then at least one of the following holds:
\begin{enumerate}[label=(\alph*)]
\item
there exist \(J\subset[m]\) and \(\mathbf a\in \mathcal{P}^+(J)\) with \(\|\mathbf a\|_1\ge \tfrac1{64} Km\) such that \(\mathsf{Ret}^{+}_J(\mathbf a)\) occurs;
\item
there exist \(J\subset[m]\) and \(\mathbf a\in \mathcal{P}^-(J)\) with \(\|\mathbf a\|_1\ge \tfrac1{64} Km\) such that \(\mathsf{Ret}^{-}_J(\mathbf a)\) occurs.
\end{enumerate}
\end{theorem}

\subsection{Probability bounds for \(\mathsf{Ret}^{\pm}_J(\mathbf a)\)}
We now bound the probabilities of the realization events \(\mathsf{Ret}^{\pm}_J(\mathbf a)\) appearing in Theorem~\ref{thm:many-pruned-implies-global-pattern}.

\begin{proposition}\label{prop:GRet_bound}
There exists \(c=c(E)>0\) such that for any \(K\ge 256\) and any integer interval \(I\subset\mathbb{Z}\) whose length \(|I|\) is sufficiently large and an integer multiple of \(2L\),
\begin{align}
&\mathbb{P}\big(m=r,\, \mathsf{Ret}^{+}_J(\mathbf a)\big)\le e^{-cKr}
\ \  \forall\,r\in\mathbb{N}^+,\,J\subset[r],\,\mathbf a\in \mathcal{P}^+(J)\text{ with }\|\mathbf a\|_1\ge \tfrac1{64} Kr,\label{eq:GRet_bound1}\\
&\mathbb{P}\big(m=r,\, \mathsf{Ret}^{-}_J(\mathbf a)\big)\le e^{-cKr}
\ \  \forall\,r\in\mathbb{N}^+,\,J\subset[r],\,\mathbf a\in \mathcal{P}^-(J)\text{ with }\|\mathbf a\|_1\ge \tfrac1{64} Kr.\label{eq:GRet_bound2}
\end{align}
\end{proposition}
\begin{proof}
We prove \eqref{eq:GRet_bound1}; the bound \eqref{eq:GRet_bound2} follows by time reversal, e.g.\ by considering \(\check S_n=S_{-n}\) and noting that backward realizations for \(S\) correspond to forward realizations for \(\check S\). For the proof of \eqref{eq:GRet_bound1}, we begin with a lemma.
\begin{lemma}
    Fix a nonempty finite set \(J=\{j_1<\cdots<j_{\#J}\}\subset [|I|]\) and \(\mathbf a=(a_1,\ldots,a_{\#J})\in\mathcal P^+(J)\).
For each \(\ell=1,\ldots,\#J\), write
\[
\Lambda^{+}_\ell(\mathbf a)=\bigl((p^{(\ell)}_1,q^{(\ell)}_1),\ldots,(p^{(\ell)}_{a_\ell},q^{(\ell)}_{a_\ell})\bigr),
\] 
and define
\[N=N(J,\mathbf a)=\#\big\{(\ell,k):\ 1\le \ell\le \#J,\ 2\le k\le a_\ell,\ p^{(\ell)}_{k-1}=p^{(\ell)}_k\big\}.\]
Then we have 
\begin{equation}\label{eq:tower_J}
\mathbb{P}\big(J\subset[m],\ \mathsf{Ret}^{+}_J(\mathbf a)\big)
\le
\exp\!\big(-c_0\,N\big).
\end{equation}
\end{lemma}
\begin{proof}
    Recall the definition of $\mathsf{Ret}^{+}_J(\mathbf a)$ in Definition \ref{def:return_pattern_subset}.
Let \(D_\ell\) be the event that the required ordered hits occur within the corresponding time window:
\[
\begin{aligned}
D_\ell:=\Big\{&
\rho_{j_\ell}<\infty,\ \exists\,u^{(\ell)}_1\le \cdots\le u^{(\ell)}_{a_\ell}
\ \text{with }u^{(\ell)}_k\in(\rho_{j_\ell}+L,\rho_{j_{\ell+1}}+L],\\
&S_{u^{(\ell)}_k}\in S\!\left(R^{+}_{p^{(\ell)}_k\,q^{(\ell)}_k}\right)
\ \forall\,k=1,\ldots,a_\ell
\Big\}.
\end{aligned}
\]
Then we have
\[
\big\{J\subset [m],\ \mathsf{Ret}^{+}_J(\mathbf a)\big\}\subset \bigcap_{\ell=1}^{\#J} D_\ell.
\]

For any set \(A\subset\mathbb{Z}^d\), write \(T_A:=\inf\{n\ge 0:\ S_n\in A\}\).
By the strong Markov property at time \(\rho_{j_\ell}+L\), we obtain
\begin{align}\label{eq:Dell_cond}
\begin{aligned}
&\mathbb{P}\big(D_\ell\,\big|\,\rho_{j_\ell}<\infty,\ S[I^-,\rho_{j_\ell}+L]\big)\\
\le\,&
\mathbb{P}_{S_{\rho_{j_\ell}+L}}\!\big(\exists\,u_1\le\cdots\le u_{a_\ell}\ \text{s.t. } S_{u_k}\in I_k,\ k=1,\ldots,a_\ell\big)\Big|_{I_k=S\big(R^{+}_{p^{(\ell)}_k\,q^{(\ell)}_k}\big)}\\
\le\,&
\mathbb{E}_{S_{\rho_{j_\ell}+L}}\!\left(
\mathbf 1_{\{T_{I_1}<\infty\}}
\mathbb{E}_{T_{I_1}}\!\left(
\mathbf 1_{\{T_{I_2}<\infty\}}\cdots
\mathbb{E}_{T_{I_{a_\ell-1}}}\!\big(\mathbf 1_{\{T_{I_{a_\ell}}<\infty\}}\big)
\right)\right)\Big|_{I_k=S\big(R^{+}_{p^{(\ell)}_k\,q^{(\ell)}_k}\big)}.
\end{aligned}
\end{align}

If \(p^{(\ell)}_{k-1}=p^{(\ell)}_k\), then (by the \(\succ_+\)-ordering) necessarily \(q^{(\ell)}_{k-1}=q^{(\ell)}_k+1\), so the walk must hit two successive rod sub-segments on the same rod path.
By transience and and the fixed separation of successive rod sub-segments, there exists \(c_0=c_0(E)>0\) such that for every \(k\) and \(\ell\) satisfying
\(p^{(\ell)}_{k-1}=p^{(\ell)}_k\),
\[
\mathbb{P}_{T_{I_{k-1}}}\big(T_{I_k}<\infty\big)\big|_{I_{k-1}=S\big(R^{+}_{p^{(\ell)}_{k-1}\,q^{(\ell)}_{k-1}}\big),\,I_k=S\big(R^{+}_{p^{(\ell)}_k\,q^{(\ell)}_k}\big)}\le e^{-c_0}\in(0,1).
\]
Substituting into \eqref{eq:Dell_cond} yields
\[
\mathbb{P}\big(D_\ell\,\big|\,\rho_{j_\ell}<\infty,\ S[I^-,\rho_{j_\ell}+L]\big)
\le
\exp\!\big(-c_0\,N_\ell\big),
\quad
N_\ell:=\#\big\{2\le k\le a_\ell:\ p^{(\ell)}_{k-1}=p^{(\ell)}_k\big\}.
\]
The desired conclusion \eqref{eq:tower_J} then follows by applying the tower property over \(\ell=1,\ldots,\# J\) and noting that $N=\sum_{\ell=1}^{\# J} N_\ell$. 
\end{proof}

Return to the proof of \eqref{eq:GRet_bound1}. By \eqref{eq:tower_J}, for any $r\in\mathbb{N}^+$, $J\subset [r]$ and $\mathbf a\in\mathcal{P}^+(J)$,
\begin{align}\label{eq:tower_J_r}
    \mathbb{P}\big(m=r,\, \mathsf{Ret}^{+}_J(\mathbf a)\big)\le \mathbb{P}\big(J\subset [m],\, \mathsf{Ret}^{+}_J(\mathbf a)\big)\le \exp\!\big(-c_0\,N\big)
\end{align}

It remains to lower bound \(N\). For each \(\ell\), within the list \(\Lambda^{+}_\ell(\mathbf a)\) the first coordinate \(p^{(\ell)}_k\) is non-increasing in \(k\).
Moreover, \(\Lambda^{+}_\ell(\mathbf a)\) is an initial segment of \(\mathcal U^{+}_\ell(\mathbf a)\) arranged in decreasing order \(\succ_+\).
Consequently, if \(p^{(\ell)}_{k-1}\neq p^{(\ell)}_k\), then the remaining indices with first coordinate \(p^{(\ell)}_{k-1}\) have just been completely exhausted at stage \(\ell\).
Since each \(p\in J\) can be exhausted at most once over the entire construction, we obtain the crude bound
\[
\#\big\{(\ell,k):\ 1\le \ell\le \# J,\ 2\le k\le a_\ell,\ p^{(\ell)}_{k-1}\neq p^{(\ell)}_k\big\}\le \#J.
\]
Therefore,
\begin{align*}
N
&=\#\big\{(\ell,k):\ 1\le \ell\le \#J,\ 2\le k\le a_\ell,\ p^{(\ell)}_{k-1}=p^{(\ell)}_k\big\}\\
&=\#\big\{(\ell,k):\ 1\le \ell\le \#J,\ 2\le k\le a_\ell\big\}
-\#\big\{(\ell,k):\ 1\le \ell\le \#J,\ 2\le k\le a_\ell,\ p^{(\ell)}_{k-1}\neq p^{(\ell)}_k\big\}\\
&\ge \Big(\sum_{\ell=1}^{\#J} a_\ell - \#J\Big)-\#J
= \|\mathbf a\|_1-2\#J.
\end{align*}
In particular, for \(J\subset [r]\) and \(\mathbf a\in\mathcal P^+(J)\) with \(\|\mathbf a\|_1\ge \tfrac1{64} Kr\),
\[
N\ge \|\mathbf a\|_1-2r\ge \Big(\tfrac1{64}K-2\Big)r\ge \frac{K}{128}\,r,
\qquad\text{for }K\ge 256.
\]
Combining this with \eqref{eq:tower_J_r} yields
\[
\mathbb{P}\big(m=r,\  \mathsf{Ret}^{+}_J(\mathbf a)\big)\le \exp\!\Big(-\frac{c_0}{128}\,Kr\Big),
\]
which proves \eqref{eq:GRet_bound1}.
\end{proof}
\subsection{Proof of Theorem~\ref{thm:rod_cut}}

\begin{proof}[Proof of Theorem~\ref{thm:rod_cut}]
Since \(|I|\) is sufficiently large, the event \(m>\theta'|I|\) implies \(m>16\vee\theta'|I|\).
By Theorem~\ref{thm:many-pruned-implies-global-pattern}, we have
\begin{align*}
&\mathbb{P}\!\left(\#\,\mathrm{Cut}^{\mathrm{rod}}_L(I)\le \frac12 \theta'|I|\right)\\
\le\,& \mathbb{P}\big(m\le \theta'|I|\big)
 +\mathbb{P}\!\left(m> \theta'|I|,\,
\#\big\{1\le i\le m:\ \rho_i\notin \mathrm{Cut}^{\mathrm{rod}}_L(I)\big\}>m/2\right)\\
\le\,& \mathbb{P}\big(m\le \theta'|I|\big)
+\sum_{*\in\{\pm\}}\sum_{r,J,\mathbf a}
\mathbb{P}\big(m=r,\  \mathsf{Ret}^{*}_J(\mathbf a)\big),
\end{align*}
where the second sum is over all $r>\theta'|I|$, $J\subset [r]$ and $\mathbf a\in \mathcal{P}^*(J)$ with $\|\mathbf a\|_1 \ge \frac1{64} K r$.

By Lemma~\ref{lem:many-rod-points}, the first term \(\mathbb{P}(m\le \theta'|I|)\) is exponentially small in \(|I|\), so it remains to show that there exist \(K\in\mathbb{N}^+\) and \(c=c(E,K,\theta')>0\) such that for all intervals \(I\) whose length \(|I|\) is sufficiently large and an integer multiple of \(2L\), and each $*\in\{\pm\}$,
\begin{equation}\label{eq:Ret_suffice}
\begin{aligned}
\sum_{r,J,\mathbf a}
\mathbb{P}\big(m=r,\  \mathsf{Ret}^{*}_J(\mathbf a)\big)\le e^{-c|I|}.	
\end{aligned}
\end{equation}

Proposition~\ref{prop:GRet_bound} bounds each probability inside the sums. Thus we only need to control the number of return profiles that appear.
For \(\mathbf a\in\mathcal{P}^+(J)\cup\mathcal{P}^-(J)\) we have \(\|\mathbf a\|_1\le K\# J\) by definition. Hence, for any fixed \(J\subset[r]\), the number of admissible forward or backward patterns is bounded by
\[
\#\Big\{\mathbf a\in \mathbb{N}^{\# J}:\ 0\le \|\mathbf a\|_1 \le Kr\Big\}
= \sum_{s=0}^{Kr}\binom{s+r-1}{r-1}
=\binom{Kr+r}{r}\le \big(e(K+1)\big)^{r},\]
where we used \(\sum_{s=0}^{M}\binom{s+k-1}{k-1}=\binom{M+k}{k}\) and the standard bound
\(\binom{n}{k}\le (en/k)^k\).
Combining this counting estimate with Proposition~\ref{prop:GRet_bound} yields that for each $*\in\{\pm\}$,
\begin{align}\label{eq:Ret_estimate}
\begin{aligned}
\sum_{r,J,\mathbf a}
\mathbb{P}\big(m=r,\  \mathsf{Ret}^{*}_J(\mathbf a)\big)&\le \sum_{r=\theta'|I|+1}^\infty \sum_{J\subset [r]} \big(e(K+1)\big)^{r} e^{-c_0Kr}\\
&= \sum_{r=\theta'|I|+1}^\infty \big(2e(K+1)e^{-c_0K}\big)^r.
\end{aligned}
\end{align}
for some \(c_0=c_0(E)>0\) and all \(K\ge 256\). Choosing \(K\) sufficiently large so that
\begin{align}\label{eq:K}
    2e(K+1)e^{-c_0K}<1,
\end{align}
we obtain \eqref{eq:Ret_suffice}, which completes the proof.
\end{proof}

\subsection{Preliminaries for the proof of Theorem~\ref{thm:many-pruned-implies-global-pattern}}\label{sec:pf_necessary}
It remains to prove Theorem~\ref{thm:many-pruned-implies-global-pattern}.
We introduce the notions of pruning times and the associated pruning intervals.
Using the pruning times, we define the return patterns that satisfy the requirements in
Theorem~\ref{thm:many-pruned-implies-global-pattern}.
\subsubsection{\texorpdfstring{Pruning times and basic properties}{Pruning times and basic properties}}
\begin{definition}[Pruning times]\label{def:pruning_times}
Let \(I=[I^-,I^+]\subset\mathbb{Z}\) be an integer interval and let \(j\in(I^-,I^+]\).
If the step \((S_{j-1},S_j)\) is pruned in \(\mathsf{Prune}(S(I),E)\), we define its forward and backward pruning times by
\begin{align*}
\zeta^+(j)=\zeta^+(j,I)
&:=\inf\Big\{i\in [j,I^+]:\ (S_{j-1},S_j)\text{ is pruned in }\mathsf{Prune}(S[I^-,i],E)\Big\},\\
\zeta^-(j)=\zeta^-(j,I)
&:=\sup\Big\{i\in (I^-,j]:\ (S_{i-1},S_i)\text{ is retained in }\mathsf{Prune}\big(S[I^-,\zeta^+(j,I)],E\big)\Big\},
\end{align*}
where \(\sup\emptyset:=I^-\).

If \((S_{j-1},S_j)\) is not pruned in \(\mathsf{Prune}(S(I),E)\), we set \(\zeta^+(j):=+\infty\) and \(\zeta^-(j):=-\infty\).
\smallskip
\noindent We will also use the shorthand
\[
\zeta(j)=\zeta(j,I):=\bigl[\zeta^-(j),\,\zeta^+(j)\bigr]
\]
for the associated pruning interval.
\end{definition}
\begin{remark}[Interpretations of \(\zeta^\pm(j,I)\) in the loop-pruned decomposition]\label{rem:pruning_times_last_segment}
Recall the loop-pruned decomposition introduced in Section~\ref{sec:basic}. 
Assume that the step \((S_{j-1},S_j)\) is pruned in \(\mathsf{Prune}(S(I),E)\).
Apply the decomposition to the truncated path \(S[I^-,\zeta^+(j)]\). By definition of \(\zeta^+(j)\), this is the first time at which \((S_{j-1},S_j)\) is pruned, namely when extending the path from \(S[I^-,\zeta^+(j)-1]\) to \(S[I^-,\zeta^+(j)]\).
Consequently, in the loop-pruned decomposition of \(S[I^-,\zeta^+(j,I)]\), the final pruned segment is exactly
\[
S\big(\zeta(j)\big)-S_{\zeta^-(j)}.
\]
\end{remark} 
We record a structural property of pruning intervals.
\begin{lemma}\label{lem:pruning-intervals-laminar-and-order}
Assume $i<j$ in $I$. Then $\zeta(i)$ and $\zeta(j)$ are laminar: either they are disjoint up to endpoints, or one is contained in the other.
Moreover, if both steps are pruned and the intervals are in the containment case, then exactly one of the following alternatives holds:
\[
\textup{(C1)}\ \zeta(j)\subsetneq \zeta(i),\ i\le \zeta^-(j);
\qquad
\textup{(C2)}\ \zeta(i)\subsetneq \zeta(j),\ \zeta^+(i)\le j-1;
\qquad
\textup{(C3)}\ \zeta(i)=\zeta(j).
\]
\end{lemma}

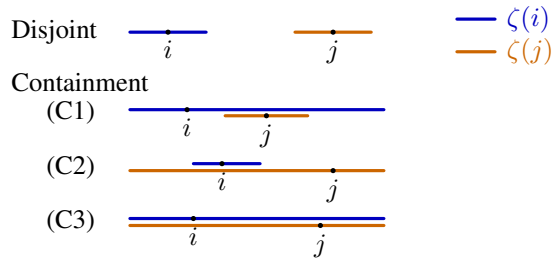
\begin{figure}[htbp]
    \centering
    \begin{tikzpicture}[scale=0.8,
  x=1.05cm,y=0.72cm,
  line cap=round,
  >=latex,
  zi/.style={line width=1.2pt, draw=blue!75!black},
  zj/.style={line width=1.2pt, draw=orange!80!black},
  lab/.style={font=\small}
]
  \draw[zi] (7.15,-0.70) -- (7.75,-0.70);
  \node[lab,anchor=west,text=blue!75!black] at (7.80,-0.70) {$\zeta(i)$};
  \draw[zj] (7.15,-1.45) -- (7.75,-1.45);
  \node[lab,anchor=west,text=orange!80!black] at (7.80,-1.45) {$\zeta(j)$};
  \node[lab,anchor=west] at (0, -1.0) {Disjoint};
  \draw[zi] (2.0,-1.0) -- (3.2,-1.0);
  \draw[zj] (4.6,-1.0) -- (5.8,-1.0);
  \fill (2.6,-1.0) circle (1.2pt) node[lab,anchor=north] {$i$};
  \fill (5.2,-1.0) circle (1.2pt) node[lab,anchor=north] {$j$};
  \node[lab,anchor=west] at (0, -2.15) {Containment};
  \node[lab,anchor=west] at (0.55, -2.85) {(C1)};
  \draw[zi] (2.0,-2.78) -- (6.0,-2.78);
  \draw[zj] (3.5,-2.92) -- (4.8,-2.92);
  \fill (2.9,-2.78) circle (1.2pt) node[lab,anchor=north] {$i$};
  \fill (4.15,-2.92) circle (1.2pt) node[lab,anchor=north] {$j$};
  \node[lab,anchor=west] at (0.55, -4.10) {(C2)};
  \draw[zj] (2.0,-4.18) -- (6.0,-4.18);
  \draw[zi] (3.0,-4.02) -- (4.05,-4.02);
  \fill (3.45,-4.02) circle (1.2pt) node[lab,anchor=north] {$i$};
  \fill (5.2,-4.18) circle (1.2pt) node[lab,anchor=north] {$j$};
  \node[lab,anchor=west] at (0.55, -5.35) {(C3)};
  \draw[zi] (2.0,-5.28) -- (6.0,-5.28);
  \draw[zj] (2.0,-5.44) -- (6.0,-5.44);
  \fill (3.0,-5.28) circle (1.2pt) node[lab,anchor=north] {$i$};
  \fill (5.0,-5.44) circle (1.2pt) node[lab,anchor=north] {$j$};
\end{tikzpicture}
\caption{Visualization on the time axis of the disjoint case and the three containment alternatives for $\zeta(i)$ and $\zeta(j)$.}
\label{fig:pruning-intervals}
\end{figure}

\begin{proof}
If one of the two steps is not pruned in \(\mathsf{Prune}(S(I),E)\), the conclusion is immediate from the convention in Definition~\ref{def:pruning_times}.
Assume that both steps are pruned.
We first rule out the interlacing configuration
\begin{equation}\label{eq:interlacing_zeta}
\zeta^-(i)<\zeta^-(j)<\zeta^+(i)<\zeta^+(j).
\end{equation}
Suppose for contradiction that \eqref{eq:interlacing_zeta} holds.
By Proposition~\ref{prop:stage}, pruning up to time $\zeta^+(j)$ can be performed in two stages: first prune up to time $\zeta^+(i)$, and then continue pruning up to time $\zeta^+(j)$.
Since $\zeta^-(i)<\zeta^-(j)<\zeta^+(i)$, the step $(S_{\zeta^-(j)-1},S_{\zeta^-(j)})$ is already pruned in $\mathsf{Prune}(S[I^-,\zeta^+(i)],E)$, and hence it is also pruned in $\mathsf{Prune}(S[I^-,\zeta^+(j)],E)$.
This contradicts the definition of $\zeta^-(j)$, which requires $(S_{\zeta^-(j)-1},S_{\zeta^-(j)})$ to be retained in $\mathsf{Prune}(S[I^-,\zeta^+(j)],E)$.
Thus \eqref{eq:interlacing_zeta} is impossible.
The opposite crossing order is also impossible: if
\[
\zeta^-(j)<\zeta^-(i)<\zeta^+(j)<\zeta^+(i),
\]
then the step \((S_{i-1},S_i)\) lies in the final pruned segment \(S(\zeta(j))\) in the loop-pruned decomposition of \(S[I^-,\zeta^+(j)]\). Hence \(\zeta^+(i)\le \zeta^+(j)\), contradicting \(\zeta^+(j)<\zeta^+(i)\).
	Therefore $\zeta(i)$ and $\zeta(j)$ are laminar.

	It remains to identify the containment cases. We need the following observation.
	\begin{lemma}\label{lem:same-forward-pruning-time}
	If \(\zeta^+(i)=\zeta^+(j)\), then \(\zeta(i)=\zeta(j)\).
	\end{lemma}
	\begin{proof}
	By Remark~\ref{rem:pruning_times_last_segment}, applied to the same truncated path \(S[I^-,\zeta^+(i)]\), both intervals correspond to the final pruned segment in its loop-pruned decomposition.
	\end{proof}
	If the two intervals \(\zeta(i)\) and \(\zeta(j)\) are equal, then alternative (C3) holds.
	Suppose next that \(\zeta(j)\subsetneq\zeta(i)\). If \(i>\zeta^-(j)\), then the step \((S_{i-1},S_i)\) lies in the final pruned segment \(S(\zeta(j))\), and therefore \(\zeta^+(i)\le\zeta^+(j)\). The containment gives the reverse inequality, so \(\zeta^+(i)=\zeta^+(j)\), and Lemma~\ref{lem:same-forward-pruning-time} gives \(\zeta(i)=\zeta(j)\), a contradiction. Hence \(i\le \zeta^-(j)\), giving alternative (C1).
	The case \(\zeta(i)\subsetneq\zeta(j)\) is analogous and gives alternative (C2).
\end{proof}

We shall also use the following corollary.
\begin{corollary}\label{cor:pruning_interval_contained}
Suppose $u\in I'$ and the step $(S_{u-1},S_u)$ is contained in a pruned segment $[S_{l_1},S_{l_2}]$ (before translation) in the loop-pruned decomposition of $S(I')$, i.e.\ $l_1<u\le l_2$. Then $\zeta(u,I')\subset [l_1,l_2]$.
\end{corollary}
\begin{proof}
If \(u=l_2\), the conclusion is immediate. Assume \(u<l_2\).
By the definitions of \(\zeta^+\) and the loop-pruned decomposition, \(\zeta^+(u,I')\le l_2\) and \([l_1,l_2]=\zeta(l_2,I')\). Since \(u\in [l_1,l_2]\), Lemma~\ref{lem:pruning-intervals-laminar-and-order} gives either \(\zeta(u,I')\subset \zeta(l_2,I')\), or \(\zeta(l_2,I')\subset \zeta(u,I')\). In the latter case, strict containment is impossible by alternative (C1) of Lemma~\ref{lem:pruning-intervals-laminar-and-order}, since it would force \(u\le \zeta^-(l_2,I')=l_1\). Hence equality holds in the latter case, and therefore \(\zeta(u,I')\subset \zeta(l_2,I')=[l_1,l_2]\).
\end{proof}

\subsubsection{\texorpdfstring{Rod-point pruning profiles}{Rod-point pruning profiles}}
This subsubsection states the two structural propositions for rod-point pruning profiles; their proofs are given in the next subsubsection.

Recall the rod points $\rho_1,\ldots,\rho_m$. Define
\[
U := \{(u,v)\in \mathbb{Z}^2 : 1\le u\le m,\ |v|\le K\},\qquad
t_{uv} := \rho_u + 2vL_E \quad \text{for } (u,v)\in U.
\]

Fix a sub-interval \(I_0=[I_0^-,I_0^+]\subset I\).
For \((u,v)\in U\) with \(t_{uv}\in (I_0^-,I_0^+]\), set
\[
\zeta_{uv}^\pm := \zeta^\pm(t_{uv}, I_0),
\qquad
\zeta_{uv} := \zeta(t_{uv}, I_0).
\]

Lemma~\ref{lem:pruning-intervals-laminar-and-order} says that pruning intervals form a laminar family: they do not cross.
Equivalently, after grouping intervals that are linked by containment, one obtains disjoint laminar clusters, with the only remaining relations inside each cluster being containments.
The next proposition applies this structure, for each fixed \(u\), to the finite family \(\{\zeta_{uv}:|v|\le K\}\).
It shows that the relevant containments are organized by a single index \(k_u\): the intervals on each side of \(k_u\) form a nested chain.

\begin{proposition}\label{prop:rod_laminar_structure}
For any sub-interval \(I_0=[I_0^-,I_0^+]\subset I\) and any \(u\in[m]\) such that \([\rho_u-L,\rho_u+L]\subset I_0\),
there exists $k_u \in [-K,K]$ such that:

\begin{enumerate}[label=(\arabic*)]

\item for $k_u \le v \le K$, the pruning intervals are nested as
\[
\zeta_{uK} \subset \zeta_{u,K-1} \subset \cdots \subset \zeta_{u,k_u}.
\]

\item for $-K \le v < k_u$, the pruning intervals are nested as
\[
\zeta_{u,-K} \subset \zeta_{u,-K+1} \subset \cdots \subset \zeta_{u,k_u-1}.
\]

\end{enumerate}
\end{proposition}
Figure~\ref{fig:rod_nesting_two_sides} illustrates the left- and right-nesting patterns
of the pruning intervals $\zeta_{uv}$ described in Proposition~\ref{prop:rod_laminar_structure}.
While the intervals are totally ordered by inclusion within each of the two families
$\{\zeta_{uv}: v\le k_u-1\}$ and $\{\zeta_{uv}: v\ge k_u\}$, the relative position between a left interval
and a right interval may vary.

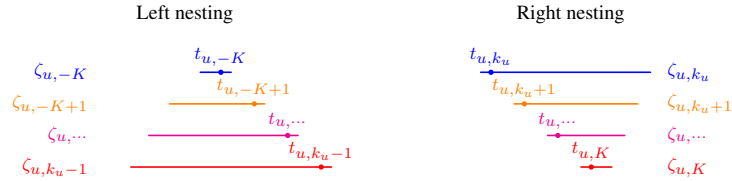
\begin{figure}[htbp]
\centering
\resizebox{0.6885\linewidth}{!}{
\begin{tikzpicture}[x=0.9cm,y=0.85cm,line cap=round]

\def\gap{3.5}
\def\L{2.4}
\def\grow{0.3}
\def\ysep{0.65}

\begin{scope}[shift={(-\gap/2,0)}]

\node at (-2,1.2) {\small Left nesting};

\foreach \j/\lab/\col/\xl/\xr/\tx in
{0/{-K}/blue/-1.70/-1.10/-1.30,
 1/{-K+1}/orange/-2.30/-0.45/-0.65,
 2/{\cdots}/magenta/-2.70/0.20/0.00,
 3/{k_u-1}/red/-3.05/0.85/0.65}{
    \pgfmathsetmacro{\y}{-\ysep*\j }
    \draw[thick,draw=\col] (\xl,\y) -- (\xr,\y);
    \node[left,text=\col] at (-3.75,\y) {$\zeta_{u,\lab}$};

    \fill[fill=\col] (\tx,\y) circle (1.3pt);
    \node[above,text=\col] at (\tx,\y) {${t_{u,\lab}}$};
}

\end{scope}

\begin{scope}[shift={(\gap/2,0)}]

\node at (2,1.2) {\small Right nesting};

\foreach \j/\lab/\col/\xl/\xr/\tx in
{0/{k_u}/blue/0.25/3.55/0.45,
 1/{k_u+1}/orange/0.90/3.30/1.10,
 2/{\cdots}/magenta/1.55/3.05/1.75,
 3/{K}/red/2.20/2.80/2.40}{
    \pgfmathsetmacro{\y}{-\ysep*\j}
    \draw[thick,draw=\col] (\xl,\y) -- (\xr,\y);
    \node[right,text=\col] at (3.75,\y) {$\zeta_{u,\lab}$};

    \fill[fill=\col] (\tx,\y) circle (1.3pt);
    \node[above,text=\col] at (\tx,\y) {$t_{u,\lab}$};
}

\end{scope}

\end{tikzpicture}
}
\caption{Left and right nesting patterns of the pruning intervals $\zeta_{uv}$ for fixed $u$.}
\label{fig:rod_nesting_two_sides}
\end{figure}

Recall the rod segments defined in Definition \ref{def:rod_subsegments}.
The next proposition, together with Proposition~\ref{prop:rod_laminar_structure}, provides the input used later to build return patterns.

\begin{proposition}\label{prop:visit_rod_segment}
    For any sub-interval \(I_0=[I_0^-,I_0^+]\subset I\) and any \(u\in[m]\) such that
    \([\rho_u-L,\rho_u+L]\subset I_0\),
    \begin{align}
        &S_{\zeta^+_{uv}}\in S(R^+_{uv})
        \quad\text{whenever }1\vee (k_u+1)\le v\le K\text{ and }\zeta^+_{uv}<+\infty,\label{eq:visit_rod_segment_+}\\
        &S_{\zeta^-_{uv}}\in S(R^-_{uv})
        \quad\text{whenever }-K\le v\le (-1)\wedge (k_u-1)\text{ and }\zeta^-_{uv}>-\infty.\label{eq:visit_rod_segment_-}
    \end{align}
\end{proposition}
We also record the following consequences of Lemma~\ref{lem:pruning-intervals-laminar-and-order} and Proposition~\ref{prop:visit_rod_segment}.
In particular, the strict inequalities in the chains below are ensured by Proposition~\ref{prop:visit_rod_segment}.

\begin{corollary}\label{cor:one-sided-zeta-chain}
The following hold.
\begin{enumerate}
\item Let \(u\in[m]\), \(s>\rho_u+L\), and
\(I_0=[t_{u0},s]\subset I\). If \(\zeta^+_{u1}<+\infty\), then
\[
\zeta^+_{u1}>\zeta^+_{u2}>\cdots>\zeta^+_{uK}.
\]
\item Let \(u\in[m]\), \(s<\rho_u-L\), and
\(I_0=[s,t_{u,-1}]\subset I\). If \(\zeta^-_{u,-1}>-\infty\), then
\[
\zeta^-_{u,-1}<\zeta^-_{u,-2}<\cdots<\zeta^-_{u,-K}.
\]
\end{enumerate}
\end{corollary}

\begin{corollary}\label{cor:zeta-interlacing}
For any sub-interval \(I_0=[I_0^-,I_0^+]\subset I\) and any indices
\(u<u'\) in \([m]\) such that \([\rho_u-L,\rho_{u'}+L]\subset I_0\),
the following hold:
\begin{enumerate}
\item If $\zeta^+_{uv}<+\infty\text{ and }\zeta^+_{uv}\ge t_{u',K}$ for some $-K\le v\le K$, then
\[
\zeta^+_{uv}\ge \zeta^+_{u',k_{u'}}>\zeta^+_{u',k_{u'}+1}>\cdots> \zeta^+_{u',K}.
\]
\item If $\zeta^-_{u',v}>-\infty\text{ and }\zeta^-_{u',v}<t_{u,-K}$ for some $-K\le v\le K$, then
\[
\zeta^-_{u',v}\le \zeta^-_{u,k_u-1}<\zeta^-_{u,k_u-2}<\cdots< \zeta^-_{u,-K}.
\]
\end{enumerate}
\end{corollary}
\subsubsection{\texorpdfstring{Proofs of the rod-point pruning profiles}{Proofs of the rod-point pruning profiles}}
For the proof of Proposition~\ref{prop:rod_laminar_structure}, we will need the following lemma.

\begin{lemma}\label{lem:rod_retention}
Fix a sub-interval \(I_0=[I_0^-,I_0^+]\subset I\) and \(u\in[m]\) such that \([\rho_u-L,\rho_u+L]\subset I_0\).
Suppose that, for some \(v_0\in\{-K,\ldots,K-1\}\), \(\zeta^+_{u v_0}\ge t_{u,v_0+1}\).
Then
\begin{equation}\label{eq:rod-retention-forward-times}
\zeta^+_{uw}>t_{uK}\qquad\text{for all }w\in\{v_0,\ldots,K-1\},
\end{equation}
and consequently
\begin{equation}\label{eq:rod-retention-forward-nesting}
\zeta_{u,K}\subset \zeta_{u,K-1}\subset\cdots\subset \zeta_{u v_0}.
\end{equation}
\end{lemma}
\begin{proof}
We first prove \(\zeta^+_{u v_0}>t_{uK}\).
Assume that the step at time $t_{u v_0}$ is pruned in $\mathsf{Prune}(S(I_0),E)$, and let $\zeta^+_{u v_0}<\infty$ be its pruning time.
Then at time $\zeta^+_{u v_0}$, a loop $e\in E$ containing the step $(S_{t_{u v_0}-1},S_{t_{u v_0}})$ is pruned.
In particular, recalling the definition of $L_E$ in \eqref{def:LE}, we have
\begin{equation}\label{eq:close_to_pruning_time}
|S_{t_{u v_0}}-S_{\zeta^+_{u v_0}}|\le L_E-1.
\end{equation}

On the other hand, by the definition of rod paths, for every integer $l$ with $t_{u,v_0+1}\le l\le t_{uK}$,
\begin{equation}\label{eq:far_along_rod}
|S_{t_{u v_0}}-S_l|\ge 2L_E\,|x_0|\ge 2L_E.
\end{equation}
Comparing \eqref{eq:close_to_pruning_time} with \eqref{eq:far_along_rod} shows that $\zeta^+_{u v_0}$ cannot belong to $[t_{u,v_0+1},t_{uK}]$.
Since $\zeta^+_{u v_0}\ge t_{u,v_0+1}$ by assumption, we conclude that $\zeta^+_{u v_0}>t_{uK}$.

Now fix \(w\in\{v_0+1,\ldots,K-1\}\). We prove \eqref{eq:rod-retention-forward-times} for this \(w\); by the preceding argument with \(w\) in place of \(v_0\), it suffices to show \(\zeta^+_{uw}\ge t_{u,w+1}\).
If \(\zeta^+_{uw}<t_{u,w+1}\), then the self-avoidance of the rod segment \(S[t_{u v_0},t_{u,w+1}]\) forces \(\zeta^-_{uw}<t_{u v_0}\).
The intervals \(\zeta_{u v_0}\) and \(\zeta_{uw}\) overlap. By Lemma~\ref{lem:pruning-intervals-laminar-and-order}, applied to the times \(t_{u v_0}<t_{uw}\), alternative \textup{(C1)} is impossible because \(\zeta^-_{uw}<t_{u v_0}\). Hence \(\zeta_{u v_0}\subset \zeta_{uw}\), and therefore
\[
\zeta^+_{u v_0}\le \zeta^+_{uw}<t_{u,w+1}\le t_{uK},
\]
contradicting \(\zeta^+_{u v_0}>t_{uK}\).
Thus \(\zeta^+_{uw}\ge t_{u,w+1}\), and hence \eqref{eq:rod-retention-forward-times} holds for this \(w\).

The nesting assertion \eqref{eq:rod-retention-forward-nesting} follows directly from Lemma~\ref{lem:pruning-intervals-laminar-and-order}, applied successively.
\end{proof}
\begin{proof}[Proof of Proposition \ref{prop:rod_laminar_structure}]
Define
\[
k_u:=\inf\Bigl\{v\in\{-K,\dots,K-1\}:\ \zeta^+_{uv}\ge t_{u,v+1}\Bigr\},
\]
with the convention that $k_u:=K$ if the set is empty.

\medskip
\noindent\textbf{(1) Nesting for $k_u\le v\le K$.}
If \(k_u=K\), there is nothing to prove.
Suppose \(k_u<K\). By definition of \(k_u\),
\(\zeta^+_{u,k_u}\ge t_{u,k_u+1}\).
Lemma~\ref{lem:rod_retention}, applied with \(v_0=k_u\), gives the desired nesting.

\medskip

\noindent\textbf{(2) Nesting for $-K \le v < k_u$.}
For any $-K\le v<k_u$, we have $t_{uv}\le \zeta^+_{uv}<t_{u,v+1}$. Using again the self-avoidance of the rod paths, it follows that
\(
\zeta^-_{uv}< t_{u,-K}.
\)
The nesting conclusion then follows from Lemma~\ref{lem:pruning-intervals-laminar-and-order}.
\end{proof}
\begin{proof}[Proof of Proposition \ref{prop:visit_rod_segment}]
We only prove \eqref{eq:visit_rod_segment_+}; the proof of \eqref{eq:visit_rod_segment_-} is analogous.

We claim that for every \(v\) with \(1\vee(k_u+1)\le v\le K\), and every \(l\in (t_{u,v-1},t_{uv})\), we have
\(
\zeta^+(l,I_0)\ge t_{uv}.
\)
Indeed, if not, then \(\zeta^+(l,I_0)<t_{uv}\). Exactly as in the proof of Lemma~\ref{lem:rod_retention}, using the self-avoidance of the rod path and Lemma~\ref{lem:pruning-intervals-laminar-and-order}, this implies
\[
\zeta^+_{u,v-1}\le \zeta^+(l,I_0)<t_{uv},
\]
which contradicts \eqref{eq:rod-retention-forward-times}, applied with \(v_0=k_u\), since \(v-1\ge k_u\).

Fix \(v\) with \(1\vee(k_u+1)\le v\le K\) and \(\zeta^+_{uv}<+\infty\).
By this claim and Lemma~\ref{lem:pruning-intervals-laminar-and-order}, for every
\(l\in (t_{u,v-1},t_{uv})\) we have
\(\zeta_{uv}\subset \zeta(l,I_0).
\)
In particular, all these steps \((S_{l-1},S_l)\) are retained in
\(\mathsf{Prune}(S[I_0^-,\zeta^+_{uv}-1],E)\).

Now suppose, for contradiction, that \(S_{\zeta^+_{uv}}\notin S(R^+_{uv})\). Since
\(
S_{\zeta^-_{uv}}=S_{\zeta^+_{uv}}\notin S(R^+_{uv}),
\)
we must have \(\zeta^-_{uv}\le \rho_u-L_E\). It then follows from
Corollary~\ref{cor:pruning_interval_contained} that, when extending the path from
\(S[I_0^-,\zeta^+_{uv}-1]\) to \(S[I_0^-,\zeta^+_{uv}]\), all steps
\[
(S_{l-1},S_l),\qquad \rho_u-L_E\le l<\rho_u,
\]
are pruned. Consequently, the pruned loop created at that time contains all these
steps. Hence its diameter is at least \(L_E = 1 + \max_{e \in E} \operatorname{diam}(e)\), contradicting the fact that this pruned
loop belongs to \(E\). This proves \eqref{eq:visit_rod_segment_+}.
\end{proof}

\subsection{Proof of Theorem~\ref{thm:many-pruned-implies-global-pattern}}
We shall construct the return patterns that satisfy the requirements in
Theorem~\ref{thm:many-pruned-implies-global-pattern}. 
For this, we first record a dichotomy for the event
\(\rho_i\notin \mathrm{Cut}^{\mathrm{rod}}_L(I)\).
It follows easily from Proposition~\ref{prop:stage} that a sufficient condition for \(\rho_i\) to be a rod \(E\)-cut
point is the following: for every integer interval \(R=[R^-,R^+]\supset[\rho_i-L,\rho_i+L]\),
one has (recall \(L=(2K+1)L_E\))
\begin{enumerate}[label=(\roman*)]
\item the segment \(S[\rho_i-2L_E,\rho_i]\) is fully retained in \(\mathsf{Prune}(S[R^-,\rho_i],E)\);
\item the segment \(S[\rho_i,\rho_i+2L_E]\) is fully retained in \(\mathsf{Prune}(S[\rho_i,R^+],E)\).
\end{enumerate}

Conversely, if \(\rho_i\notin \mathrm{Cut}^{\mathrm{rod}}_L(I)\), then at least one of the following holds:
\begin{enumerate}
\item[\textup{(L)}]\customitemlabel{\textup{L}}{item:rod-cut-left}
\(\exists\ s<\rho_i-2L_E\) s.t. \(S[\rho_i-2L_E,\rho_i]\) is not fully retained in \(\mathsf{Prune}(S[s,\rho_i],E)\);
\item[\textup{(R)}]\customitemlabel{\textup{R}}{item:rod-cut-right}
\(\exists\ s>\rho_i+2L_E\) s.t.  \(S[\rho_i,\rho_i+2L_E]\) is not fully retained in \(\mathsf{Prune}(S[\rho_i,s],E)\).
\end{enumerate}
By the self-avoidance of the rod path, the conditions \(s<\rho_i-2L_E\) in \eqref{item:rod-cut-left} and \(s>\rho_i+2L_E\) in \eqref{item:rod-cut-right} can equivalently be replaced by \(s<\rho_i-L\) and \(s>\rho_i+L\), respectively.
Define
\begin{align*}
J^-&:=\Big\{1\le i\le m:\ \rho_i\notin \mathrm{Cut}^{\mathrm{rod}}_L(I)\ \text{and case }(\ref{item:rod-cut-left})\text{ occurs}\Big\},\\
J^+&:=\Big\{1\le i\le m:\ \rho_i\notin \mathrm{Cut}^{\mathrm{rod}}_L(I)\ \text{and case }(\ref{item:rod-cut-right})\text{ occurs}\Big\}.
\end{align*}

The above dichotomy immediately yields the following elementary observation.

\begin{lemma}\label{lem:Jpm_large}
If
\(
\#\big\{1\le i\le m:\ \rho_i\notin \mathrm{Cut}^{\mathrm{rod}}_L(I)\big\}>\frac{m}{2},
\)
then either \(\#J^->\frac{m}{4}\) or \(\#J^+>\frac{m}{4}\).
\end{lemma}

Thus, to prove Theorem~\ref{thm:many-pruned-implies-global-pattern}, it suffices to prove the following implication.

\begin{proposition}\label{prop:Jpm_implies_global_pattern}
If \(m\ge16\) and either \(\#J^->\frac{m}{4}\) or \(\#J^+>\frac{m}{4}\) occurs, then at least one of the
following holds:
\begin{enumerate}[label=(\alph*), ref=\alph*]
\item\label{item:global-forward}
there exist \(J\subset[m]\) and \(\mathbf a\in \mathcal{P}^+(J)\) with \(\|\mathbf a\|_1\ge \tfrac1{64} Km\)
such that \(\mathsf{Ret}^{+}_J(\mathbf a)\) occurs;
\item\label{item:global-backward}
there exist \(J\subset[m]\) and \(\mathbf a\in \mathcal{P}^-(J)\) with \(\|\mathbf a\|_1\ge \tfrac1{64} Km\)
such that \(\mathsf{Ret}^{-}_J(\mathbf a)\) occurs.
\end{enumerate}
\end{proposition}

\begin{proof}
We prove that if \(\#J^+>\frac{m}{4}\), then either \eqref{item:global-forward} or
\eqref{item:global-backward} holds. The case \(\#J^->\frac{m}{4}\) is symmetric and is omitted. 

Assume $\# J^+>\tfrac{m}{4}$. We first introduce some notation. We write
\[
J^+=\{j_1<j_2<\cdots<j_{\# J^+}\}.
\]
For each \(i\), choose a witness \(s_i>\rho_{j_i}+L\) for the equivalent form of case \(\eqref{item:rod-cut-right}\), so that
\[
S[\rho_{j_i},\rho_{j_i}+2L_E]
\quad\text{is not fully retained in}\quad
\mathsf{Prune}(S[\rho_{j_i},s_i],E).
\]
By the self-avoidance of the rod segment \(S[\rho_{j_i},\rho_{j_i}+L]\), the last step \((S_{t_{j_i,1}-1},S_{t_{j_i,1}})\), where \(t_{j_i,1}=\rho_{j_i}+2L_E\), is pruned only after time \(\rho_{j_i}+L\).
We write
\[
\zeta^+_{j_i,1}:=\zeta^+(t_{j_i,1},[\rho_{j_i},s_i])<+\infty.
\]
In particular, \(\zeta^+_{j_i,1}>\rho_{j_i}+L\).

We use the times \(\rho_{j_u}+L\) as cut points, and record roughly where a selected subsequence of the pruning times
\(\zeta^+_{j_i,1}\) are located relative to these cut points as follows.
\begin{align*}
\ell_1&:=1,\qquad
\ell_{i+1}:=\inf\big\{\ell_i< u\le \# J^+: \rho_{j_{u-1}}+L<\zeta^+_{j_{\ell_i},1}\le \rho_{j_u}+L\big\}\quad\text{for }i\ge 1,\\
q&:=\#\big\{i\ge 1:\ell_i<\infty\big\}.
\end{align*}
with the convention \(\inf\emptyset:=\# J^+ + 1\), so that $\ell_{q+1}=\# J^+ + 1$.
We then define
\[
R=\big\{r_1<r_2<\cdots<r_{\# R}\big\}:=\big\{\ell_i: i\in[q-1],\,\ell_{i+1}-\ell_i=1\big\}.
\]

\begin{figure}[htbp]
\centering
\scalebox{0.75}{
\begin{tikzpicture}[
  x=1.1cm, y=1.05cm,
  tick/.style={line width=0.9pt},
  axis/.style={line width=1.1pt},
  Y/.style={text=orange!85!black},
  G/.style={text=green!60!black},
  B/.style={text=blue!70!black},
  lab/.style={font=\large},
]

\draw[axis] (0,3.2) -- (12.5,3.2);

\foreach \x/\t in {0/1, 1/2, 2/3, 3/4, 12.5/m} {
  \fill[black] (\x,3.2) circle (2.2pt);
  \node[lab,anchor=north] at (\x,2.95) {$\t$};
}

\coordinate (j1) at (1,0);
\coordinate (j2) at (3,0);
\coordinate (j3) at (6.2,0);
\coordinate (j4) at (7.2,0);
\coordinate (j5) at (8.2,0);
\coordinate (j6) at (10.4,0);
\coordinate (j7) at (11.4,0);

\coordinate (jl1) at (1,0);
\coordinate (jl2) at (6.2,0);
\coordinate (jl3) at (7.2,0);
\coordinate (jl4) at (8.2,0);
\coordinate (jl5) at (11.4,0);

\coordinate (jr1) at (6.2,0);
\coordinate (jr2) at (7.2,0);

\newcommand{\CrossLabel}[4]{
  \node[#3,lab] at ($(#1)+(0,#2)$) {$\times$};
  \node[#3,lab,anchor=north] at ($(#1)+(0,#2-0.15)$) {$#4$};
}

\CrossLabel{j1}{2.05}{Y}{j_1}
\CrossLabel{j2}{2.05}{Y}{j_2}
\CrossLabel{j3}{2.05}{Y}{j_3}
\CrossLabel{j4}{2.05}{Y}{j_4}
\CrossLabel{j5}{2.05}{Y}{j_5}
\CrossLabel{j6}{2.05}{Y}{j_6}
\CrossLabel{j7}{2.05}{Y}{j_7}

\CrossLabel{jl1}{1.05}{G}{j_{\ell_1}}
\CrossLabel{jl2}{1.05}{G}{j_{\ell_2}}
\CrossLabel{jl3}{1.05}{G}{j_{\ell_3}}
\CrossLabel{jl4}{1.05}{G}{j_{\ell_4}}
\CrossLabel{jl5}{1.05}{G}{j_{\ell_5}}

\CrossLabel{jr1}{0.05}{B}{j_{r_1}}
\CrossLabel{jr2}{0.05}{B}{j_{r_2}}

\end{tikzpicture}
}
\caption{A schematic of the indices \(j_i\) and the extracted subsequences \(\{j_{\ell_i}\}\), \(\{j_{r_i}\}\).
The green row records the selected subsequence \(\{j_{\ell_i}\}\). For instance, in this schematic \(\ell_2=3\), which indicates that \(\zeta^+_{j_{\ell_1},1}=\zeta^+_{j_1,1}\in (\rho_{j_2}+L,\rho_{j_3}+L]\). Since \(\ell_2,\ell_3,\ell_4\) are consecutive indices \(3,4,5\), the indices \(\ell_2\) and \(\ell_3\) belong to \(R\), and hence appear in the blue row.}
\end{figure}
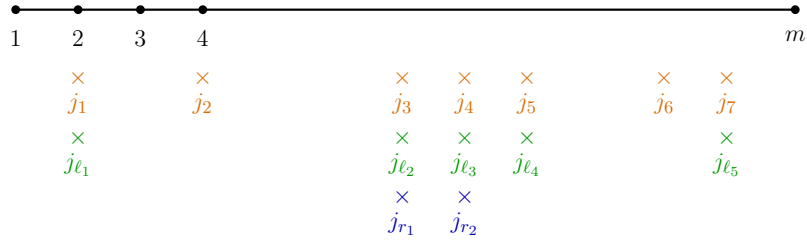

The following analysis will be divided into the cases: (1) $\# R\ge \tfrac{m}{8}$, and (2) $\# R< \tfrac{m}{8}$.

\smallskip
\textbf{Case (1) $\# R\ge \tfrac{m}{8}$}.
For each \(r\in R\), the definition of \(R\) gives
\[
\rho_{j_r}+L<\zeta^+_{j_r,1}\le \rho_{j_{r+1}}+L.
\]
Together with Corollary~\ref{cor:one-sided-zeta-chain} and the self-avoidance of the rod path, this gives
\[\rho_{j_r}+L<\zeta^+_{j_r,K}<\zeta^+_{j_r,K-1}<\cdots<\zeta^+_{j_r,1}\le \rho_{j_{r+1}}+L.\]
Thus taking
\[J:=\{j_r:r\in R\}\qquad\text{and}\qquad \mathbf{a}:=(\overbrace{K,K,\ldots,K}^{\# J})\in \mathcal{P}^+(J),\]
we have \(\|\mathbf{a}\|_1=K\#R\ge \tfrac{1}{8}Km\ge \tfrac1{64}Km\), and the event \(\mathsf{Ret}^{+}_J(\mathbf a)\) occurs.

\smallskip
\textbf{Case (2) $\# R< \tfrac{m}{8}$.}
Set
\[
\widetilde J^+:=J^+\setminus \{j_{\ell_1},\ldots,j_{\ell_q}\}.
\]
For \(i=1,\ldots,q\), write
\[
\mathsf W_i^+:=[\rho_{j_{\ell_i}},\zeta^+_{j_{\ell_i},1}].
\]
\begin{lemma}\label{lem:Jtilde-large}
If \(\#R<m/8\) and \(m\ge16\), then \(\#\widetilde J^+\ge m/32\).
\end{lemma}
\begin{proof}
Observe that, whenever \(i\in[q-1]\) and \(j_{\ell_i}\notin \{j_r:r\in R\}\), we have \(\ell_{i+1}-\ell_i>1\). Hence \(j_{\ell_i+1}\in \widetilde J^+\), and different such \(i\)'s give different elements of \(\widetilde J^+\). Therefore, 
\[
\#\Big(\{j_{\ell_1},\ldots,j_{\ell_{q-1}}\}\setminus\{j_r:r\in R\}\Big)
\le \#\widetilde J^+
\qquad\Longleftrightarrow\qquad
q-1-\#R\le \#\widetilde J^+.
\]
Since \(\#\widetilde J^+=\#J^+-q\), \(\#J^+>m/4\), and \(\#R<m/8\), it follows that
\[
\#\widetilde J^+
\ge \frac12(\#J^+-\#R-1)
> \frac{m-8}{16}
\ge \frac{m}{32},
\qquad m\ge16.
\]
This finishes the proof.
\end{proof}

Observe that for every $j\in \widetilde J^+$, there exists a unique \(\iota(j)\ge 1\) such that
\(
j\in (j_{\ell_{\iota(j)}},j_{\ell_{\iota(j)+1}}).
\)
Writing \(i=\iota(j)\), all the steps $t_{jv}$, $-K\le v\le K$, are contained in the pruning interval
\(\zeta(t_{j_{\ell_i},1})\). Therefore, by Lemma~\ref{lem:pruning-intervals-laminar-and-order},
we have \(\zeta(t_{jv},\mathsf W_i^+)\subset \zeta(t_{j_{\ell_i},1})\) for all such $(j,v)$, and consequently
\begin{align}\label{eq:finite_zeta_+}
    \rho_{j_{\ell_i}}+L
    < \zeta^-(t_{j_{\ell_i},1})
    \le \zeta^-(t_{jv},\mathsf W_i^+)
    < \zeta^+(t_{jv},\mathsf W_i^+)
    \le \zeta^+(t_{j_{\ell_i},1})
    \le \rho_{j_{\ell_{i+1}}}+L.
\end{align}

Recall the index $k_u$ defined in Proposition~\ref{prop:rod_laminar_structure}. We further divide into the following two subcases:
\begin{enumerate}[label=(\roman*),ref=\roman*]
\item\label{item:case2-many-nonpositive}
$\#\{j\in \widetilde J^+:k_{j}\le 0\}\ge \frac12 \#\widetilde J^+\ge \frac{m}{64}$;
\item\label{item:case2-many-positive}
$\#\{j\in \widetilde J^+:k_{j}> 0\}\ge \frac12 \#\widetilde J^+\ge \frac{m}{64}$.
\end{enumerate}

We shall prove that, if \eqref{item:case2-many-nonpositive} occurs, then \eqref{item:global-forward} in the statement of the proposition holds.
The point is that, in this subcase, the times
\[
\bigl\{\zeta^+(t_{jv},\mathsf W_{\iota(j)}^+):j\in\widetilde J^+,\ k_j\le0,\ 1\le v\le K\bigr\}
\]
provide the return times for the forward return pattern that we will construct below.
In the same way, if \eqref{item:case2-many-positive} occurs, then the corresponding \(\zeta^-\)-times provide the return times for a backward return pattern, giving \eqref{item:global-backward}; we omit the symmetric argument.

Now assume that \eqref{item:case2-many-nonpositive} occurs. Let
\[
J=\{w_1<w_2<\cdots<w_{\# J}\}:=\bigl\{j\in \widetilde J^+:k_{j}\le 0\bigr\},
\]
and define
\[
\begin{aligned}
\mathbf a&=(a_1,\ldots, a_{\# J}),\\
a_i&:=\#\bigl\{(j,v):j\in J,\, 1\le v\le K,\,
\rho_{w_i}+L<\zeta^+(t_{jv},\mathsf W_{\iota(j)}^+)\le \rho_{w_{i+1}}+L\bigr\},
\end{aligned}
\]
with the convention \(\rho_{w_{\#J+1}}:=+\infty\). It is clear that $\mathbf a\in \mathcal{P}^+(J)$.
By the definition of the subcase \eqref{item:case2-many-nonpositive},
\[
\|\mathbf a\|_1=\#\bigl\{(j,v):j\in J,\, 1\le v\le K\bigr\}=K\#J
\ge \frac12 K\# \widetilde J^+
\ge \frac1{64}Km.
\]

On the other hand, recall the priority order $\succ_+$ defined in Definition~\ref{def:return_pattern_subset}.
By Proposition~\ref{prop:rod_laminar_structure}, Corollary~\ref{cor:zeta-interlacing}, and the pairwise disjointness of the intervals
\(\mathsf W_r^+\), \(1\le r\le q\), coming from the construction of the indices \(j_{\ell_r}\), for each $i$,
the $a_i$ forward pruning times $\zeta^+(t_{jv},\mathsf W_{\iota(j)}^+)$ must be associated with those $t_{jv}$ for which $(j,v)$ are the first $a_i$ elements of
\[
\begin{aligned}
&\bigl\{(j,v):j\in J,\,j\le w_i,\, 1\le v\le K\bigr\}\\
&\quad\setminus
\Bigl\{(j,v):j\in J,\,j\le w_{i-1},\, 1\le v\le K,\,
\zeta^+(t_{jv},\mathsf W_{\iota(j)}^+)\le \rho_{w_i}+L\Bigr\}
\end{aligned}
\]
according to the order $\succ_+$.
The set after \(\setminus\) is the set of already used indices in the present construction.
It corresponds to \(\mathcal V^+_{i-1}(\mathbf a)\) in Definition~\ref{def:return_pattern_subset}.

Moreover, these pruning times occur in the same order as the corresponding pairs $(j,v)$
under $\succ_+$. Combining with Proposition \ref{prop:visit_rod_segment}, we have the event $\mathsf{Ret}^{+}_J(\mathbf a)$ occurs, and thus \eqref{item:global-forward} holds.
\end{proof}

\bibliographystyle{imsart-number}
\bibliography{localtime}

\end{document}